\documentclass[twoside,11pt]{article}

\usepackage{blindtext}

\usepackage{jmlr2e}

\usepackage{lastpage}

\usepackage{microtype}
\usepackage{graphicx}
\usepackage{subfigure}
\usepackage{booktabs} 
\usepackage{bbold}
\usepackage{longtable}
\usepackage{supertabular}
\usepackage[font={small,it}]{caption}

\usepackage{hyperref}

\usepackage{geometry}
\usepackage{amsmath}
\usepackage{amssymb}
\usepackage{mathtools}
\usepackage{amsfonts}
\usepackage{tikz}
\usetikzlibrary{decorations.pathreplacing}

\usepackage[ruled,vlined]{algorithm2e}

\newcommand{\abs}[1]{\lvert #1 \rvert}
\newcommand{\Abs}[1]{\left\lvert #1 \right\rvert}
\newcommand{\norm}[1]{\lVert #1 \rVert}
\newcommand{\Norm}[1]{\left\lVert #1 \right\rVert}
\newcommand{\brac}[1]{\langle #1 \rangle}
\newcommand{\Brac}[1]{\Big\langle #1 \Big\rangle}

\DeclareMathOperator{\softmax}{softmax}
\DeclareMathOperator{\Disc}{Disc}
\DeclareMathOperator{\id}{id}

\newcommand{\ie}{{\em i.e.,~}}
\newcommand{\R}{\mathbb{R}}
\newcommand{\Obj}{\mathcal{O}}
\newcommand{\J}{\Bar{J}}
\newcommand{\Q}{\Tilde{\mathcal{Q}}}
\usepackage[capitalize,noabbrev]{cleveref}

\newtheorem{assumption}[theorem]{Assumption}

\usepackage[most]{tcolorbox}

\newtheorem{defi}{{\bf Intuition \& Sketch}}

\usepackage[textwidth=2.9cm, textsize=tiny]{todonotes}

\newcommand{\new}[1]{\textcolor{blue}{#1}}

\ShortHeadings{CHANI}{Sophie Jaffard, Samuel Vaiter and Patricia Reynaud-Bouret}
\firstpageno{1}

\begin{document}

\title{CHANI: Correlation-based Hawkes Aggregation of Neurons with bio-Inspiration}

\author{\name Sophie Jaffard \email jaffard@mpi-cbg.de \\
       \addr Center for Systems Biology Dresden \\
       Max Planck Institute of Molecular Cell Biology and Genetics (MPI-CBG)\\
       Dresden, Germany
       \AND
       \name Samuel Vaiter \email samuel.vaiter@univ-cotedazur.fr \\
       \addr Laboratoire J. A. Dieudonné (LJAD)\\
       CNRS, Université Côte d'Azur\\
       Nice, France
        \AND
       \name Patricia Reynaud-Bouret \email patricia.reynaud-bouret@univ-cotedazur.fr \\
       \addr Laboratoire J. A. Dieudonné (LJAD)\\
       CNRS, Université Côte d'Azur\\
       Nice, France}

\editor{}

\maketitle

\begin{abstract}

The present work aims at proving mathematically that a neural network inspired by biology can learn a classification task thanks to local transformations only. In this purpose, we propose a spiking neural network named CHANI (Correlation-based Hawkes Aggregation of Neurons with bio-Inspiration), whose neurons activity is modeled by Hawkes processes. Synaptic weights are updated thanks to an expert aggregation algorithm, providing a local and simple learning rule. We were able to prove that our network can learn on average and asymptotically. Moreover, we demonstrated that it automatically produces neuronal assemblies in the sense that the network can encode several classes and that a same neuron in the intermediate layers might be activated by more than one class, and we provided numerical simulations on synthetic datasets. This theoretical approach contrasts with the traditional empirical validation of biologically inspired networks and paves the way for understanding how local learning rules enable neurons to form assemblies able to represent complex concepts.

\end{abstract}


\section{Introduction}

Biological neurons and their organization into networks served as the inspiration for the perceptron \citep{rosenblatt1957perceptron}, which pioneered artificial neural networks (ANNs). This initial brain inspiration continued to drive the development of more sophisticated models such as the neocognitron \citep{fukushima1982neocognitron} and later convolutional neural networks \citep{lecun1989handwritten}, all drawing from the biological processes observed in the visual cortex by \cite{hubel1962receptive}. Conversely, machine learning techniques can shed light on brain learning processes. Spiking neural networks (SNNs) \citep{Tavanaei_2019} illustrate well these interactions between the two domains: these networks model neurons activity by sequences of spikes representing the electrical pulses of biological neurons. More relevant than ANNs if one wants to study the brain neural code, they also provide insights into the development of energy-efficient algorithms \citep{stone2018principles}. To be realistic, SNNs are trained thanks to local learning rules. However, there is a gap between the empirical results provided by these rules and the underlying theory.

In the present study, as an initial stride towards demonstrating mathematically that local learning rules enable neurons to form assemblies and induce global learning, we propose a network of spiking neurons called CHANI for Correlation-based Hawkes Aggregation of Neurons with bio-Inspiration. The task is as follows: the network should learn to classify objects into one of several classes by identifying relevant feature correlations, and its learning process should involve local transformations only. The model involves
hidden and output nodes as postsynaptic neurons that produce spikes as a multivariate discrete-time Hawkes process \citep{ost2020sparse}, whose spiking probability is a function of the weighted sum of the activity of presynaptic neurons at the previous time step. The learning algorithm used to update synaptic weights comes from the expert aggregation field \citep{cesa2006prediction} thanks to this local paradigm: presynaptic neurons can be seen as experts and the strength of the connections between them and the postsynaptic neuron varies based on gains derived from these connections. Hidden neurons are trained to identify neuronal synchronization among presynaptic neurons, and a pruning process is employed to retain only those neurons that encode meaningful correlations, whereas output neurons are trained to respond to the classes in which the objects are classified.
\newline

\noindent \textbf{Contributions.} We present the first biologically inspired network which provably learns a classification task thanks to a local learning rule. Furthermore, our algorithm inherently generates neuronal assemblies: the network can encode multiple classes, and a single neuron in the intermediate layers may be activated by several classes.\\
$\bullet$ The main contributions of this article are summarized in Theorem~\ref{theo principal}. In a specific regime, referred to as \textsc{CHANI EWA}, we derive the explicit limit of the synaptic weights and show that \textsc{CHANI} converges to an idealized setting in which the hidden layers encode feature correlations. Moreover, after neuron selection, only the relevant feature correlations remain. When, in addition, classes are defined by feature correlations, we prove that CHANI succeeds the learning task asymptotically. \\
$\bullet$ We establish several intermediate results leading to Theorem~\ref{theo principal}, as well as complementary results refining its conclusions. In particular, Theorems~\ref{theorem lim hid} and~\ref{theorem lim out} provide rates of convergence for the synaptic weights. We also show that the network exhibits learning capabilities on average (Corollary~\ref{coro average}), and we compute the VC-dimension of \textsc{CHANI} (Proposition~\ref{prop vcdim}). Finally, under more general assumptions, we derive regret bounds on the learning capabilities of hidden and output neurons (Propositions \ref{prop reg hid} and \ref{prop reg class disc}). \\
$\bullet$ We illustrate our theoretical analysis with numerical results on the classification of simulated data and handwritten digits (section \ref{sec num res}).
\newline

\noindent\textbf{Related work.} In our prior work \citep{jaffard:hal-04065229}, we introduced a network named HAN (Hawkes Aggregation of Neurons) and established theoretical foundations for its learning capabilities. However, within the framework outlined in \cite{jaffard:hal-04065229}, HAN had a very simple structure, lacking hidden layers, and could only achieve very simple tasks. In the present work, we extend these results to CHANI, which can support an arbitrary number of hidden layers designed to detect neuronal synchronization, and provide novel findings described above.

Our network is inspired by the cognitive model Component-Cue \citep{gluck1988conditioning}, which states that an individual learns to classify objects by finding combinations of features which describe well the several object categories. The article \cite{MEZZADRI2022102691} compared this model to another one called ALCOVE \citep{kruschke2020alcove}, which 
stipulates that people classify new objects by comparing them to previously learned ones, in order to see which one is closer to human behavior. They showed that often Component-Cue is a best fit to human learning. Note that Component-Cue is a classic ANN in the sense that is does not involve spiking neurons nor local learning rule, and it does not comprise hidden layers. 

In the brain, conceptual objects are represented by analyzing and representing relationships among incoming signals. While elementary concepts can be depicted by the responses of individual neurons, more intricate ones are represented by groups of interconnected neurons that work together: we talk about neuronal assemblies \citep{singer1997neuronal}. A single neuron can contribute to multiple assemblies, indicating that it may participate in representing different concepts. A current measure of assembly organization is based on correlations of firing among neurons. It has been shown on recorded and simulated data that this synchronicity can be caused by dynamic changes of synaptic connection strength \citep{gerstein1989}, and neuroscientists have identified several local learning rules explaining these changes. As well-known example, Hebbian learning \citep{hebb2005organization} says that neurons that repeatedly fire together tend to become associated.

Spike-timing-dependent plasticity (STDP) \citep{caporale2008spike}, a refined form of Hebbian learning, states that a connection is strengthened if the presynaptic neuron spiked just before the postsynaptic neuron, and weakened otherwise; this rule has been used in order to simulate neuronal assemblies \citep{litwin2014formation}. However, there is no mathematical proof that these local rules enable to learn, nor that they produce neuronal assemblies. Establishing such theoretical guarantees would aid in comprehending how local mechanisms contribute to the formation of neuronal assemblies and how they function.

Enunciated in \cite{legenstein2005can}, the Spiking Neuron Convergence Conjecture (SNCC) says that SNNs can learn to implement any achievable transformation. This conjecture is inspired by the perceptron convergence theorem \citep{rosenblatt1962principles}, which asserts that as soon as the data to classify is linearly separable (\ie as soon as there exist weights with whom the data can be classified), then the weights given by the perceptron learning algorithm enable to correctly classify the data.
This conjecture has been explored for networks using STDP as learning rule \citep{legenstein2005can, legenstein2008learning}, with no conclusive theoretical result. In section \ref{sec classes feat corr}, we prove a version of this conjecture in a very specific case for CHANI.

The multivariate Hawkes process \citep{hawkes1971spectra} is a self-exciting and mutually exciting point process. Originally applied to the modelling of earthquake data \citep{turkyilmaz2013comparing}, its field of application is very broad: it is well-adapted to model networks of firing neurons \citep{hodara2017hawkes}, financial transactions \citep{hawkes2018hawkes, bacry2015hawkes}, health data \citep{pmlr-v68-bao17a}, social networks \citep{zhou2013learning}, or more generally, any sequences of events such that the occurrence of an event influences the probability of further events to occur. Also known as generalized linear models (GLMs) \citep{gerhard2017stability} in the literature, a lot of studies are about the estimation of its parameters \citep{reynaud2010adaptive, kirchner2017estimation}, its mathematical properties \citep{bremaud1996stability} and its simulation in large networks \citep{bacry2017tick, phi2020event, mascartetal2022, mascart2023efficient}. In neuroscience, it is used for instance to reconstruct functional connectivity \citep{reynaud2013inference, lambert2018reconstructing}, or to model networks of spiking neurons in order to study their mathematical properties \citep{galves2016modeling}. However, a common assumption made when mathematically analyzing neural networks modeled by Hawkes processes is that their synaptic weights remain constant. This enables to achieve a stationary state, but prevents the modeling of learning behavior. 

Hawkes processes have already been used to model events in a context of online learning: \cite{chiang2020hawkes} proposes to solve spatio-temporal event forecasting and detection problems thanks to a multi-armed bandit algorithm, and \cite{hall2016tracking} perform dynamic mirror descent to track how occurred events influence future events. However, none of these works use online learning algorithms to update the parameters of the Hawkes processes as we propose.

In the recent years, the emergence of transformers \citep{vaswani2017attention, NEURIPS2023_0561738a} have 
revolutionized the field of deep learning, especially in natural 
language processing, but not only: for instance, the Transformer Hawkes Process \citep{pmlr-v119-zuo20a} incorporates attention modules into the formula of the
 conditional intensity of the Hawkes process in order to learn event sequence data. 
 Many works have been done to understand them better: for instance, \cite{NEURIPS2023_rohekar} provide 
 a causal interpretation of self-attention, and \cite{Cordonnier2020On} investigate the relationship between 
 self-attention and convolutional layers. In section \ref{sec transformers}, we draw an analogy between 
 transformers and our model to contribute in understanding the remarkable effectiveness of transformers.

\section{Presentation of the network learning algorithm}

The structure and learning process of CHANI involve heavy notation that cannot be simplified without affecting the proofs. To improve readability, we alternate formal sections with ``Intuition and Sketch'' boxes, so that readers can focus on the main ideas. Additionally, Figure \ref{fig:Sketch} gives the main steps of the algorithm.

\begin{figure}
    \centering
~\hspace{-1cm}\begin{tabular}{cc}
{\bf A.}\includegraphics[width=0.5\linewidth]{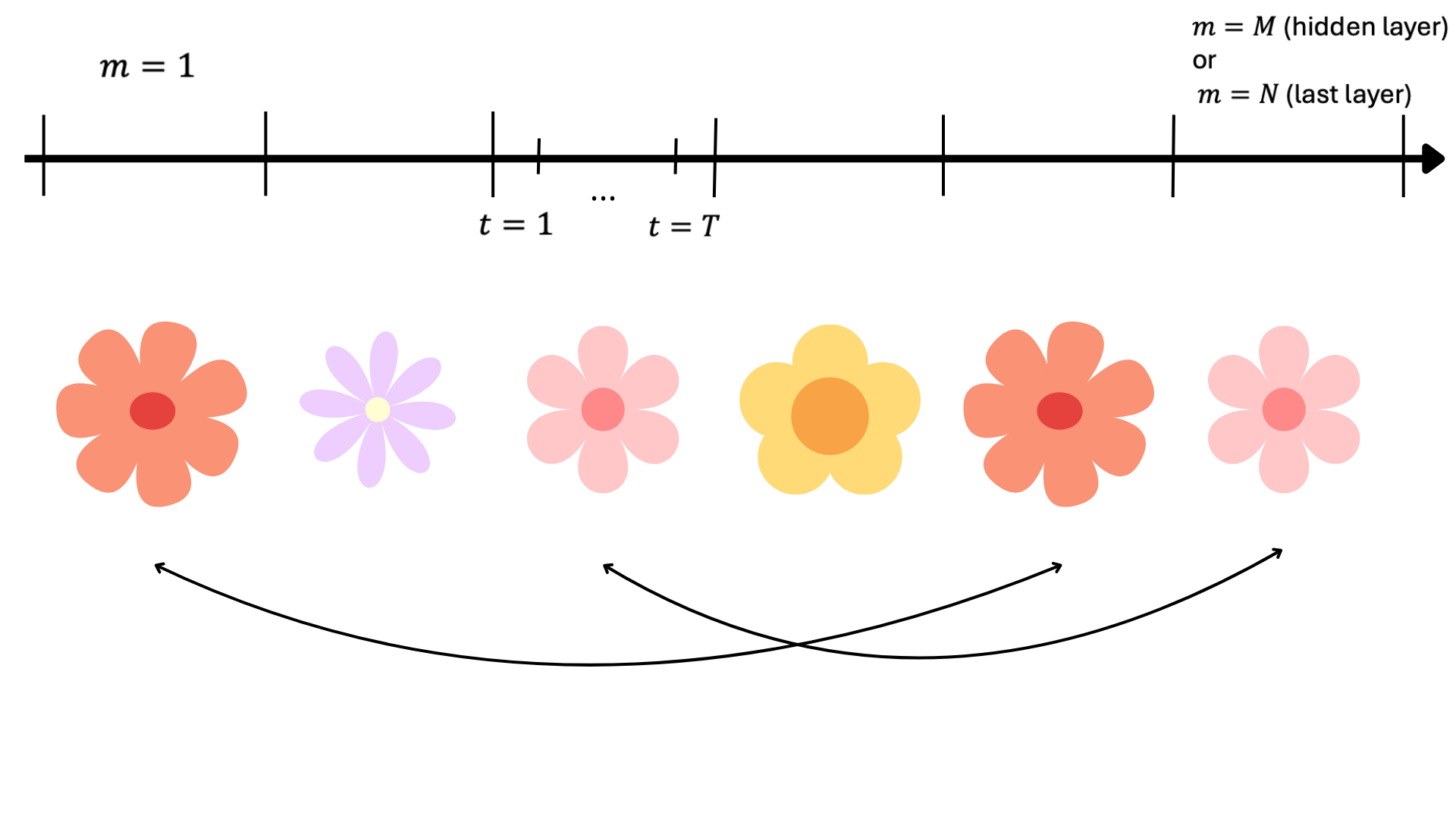}& {\bf B.}\includegraphics[width=0.5\linewidth]{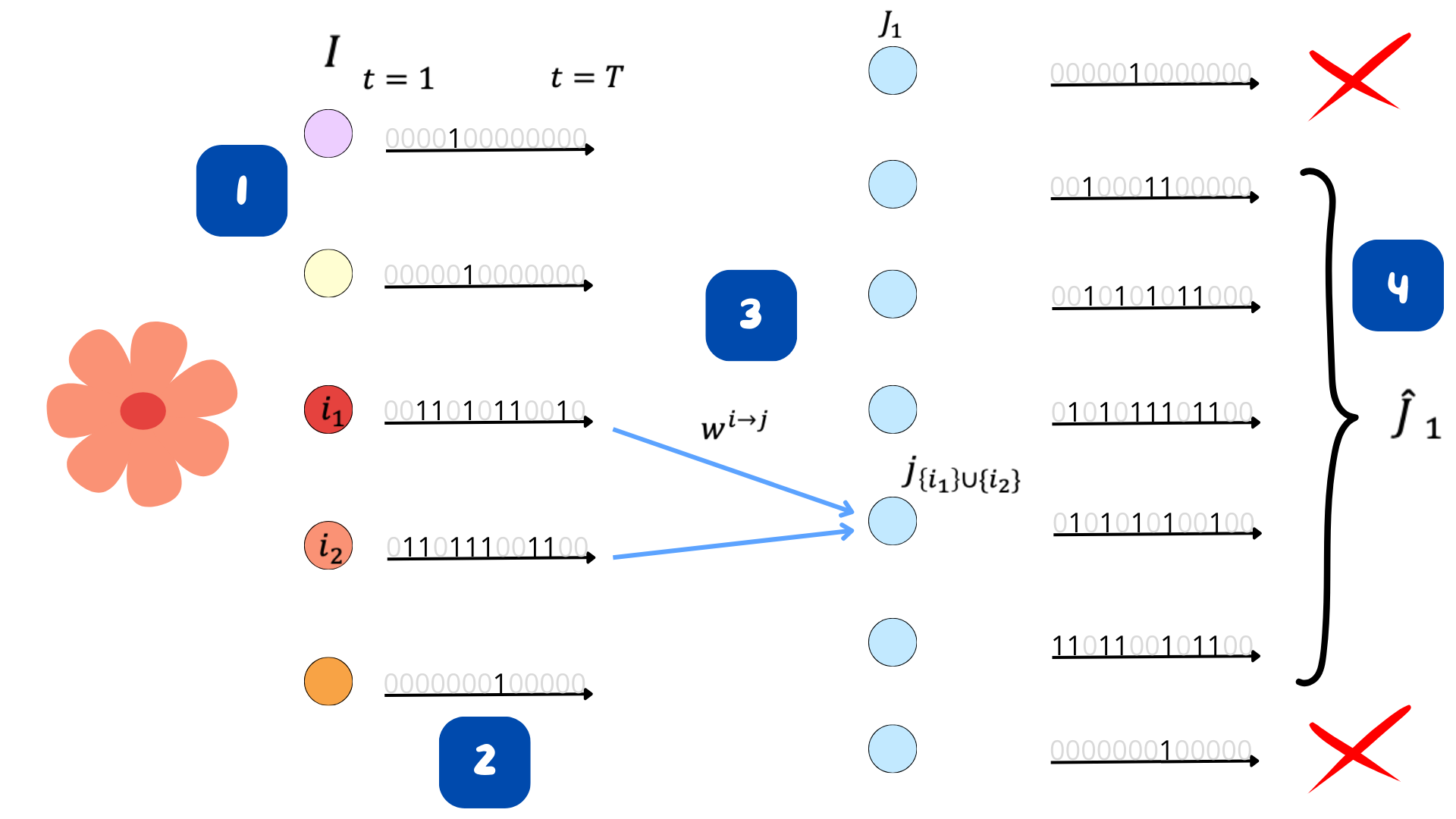}\\
{\bf C.}\includegraphics[width=0.5\linewidth]{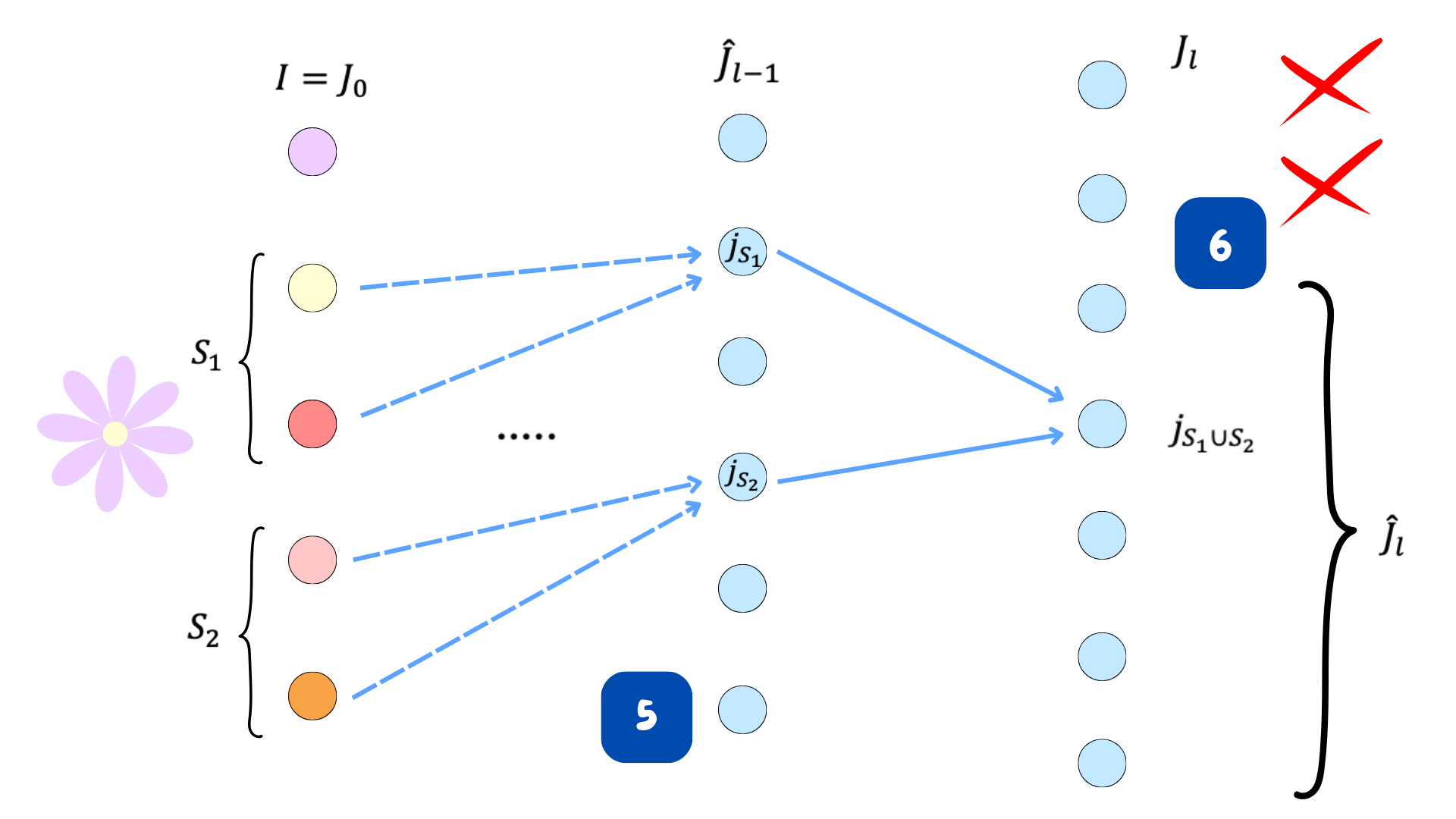} & {\bf D.}\includegraphics[width=0.5\linewidth]{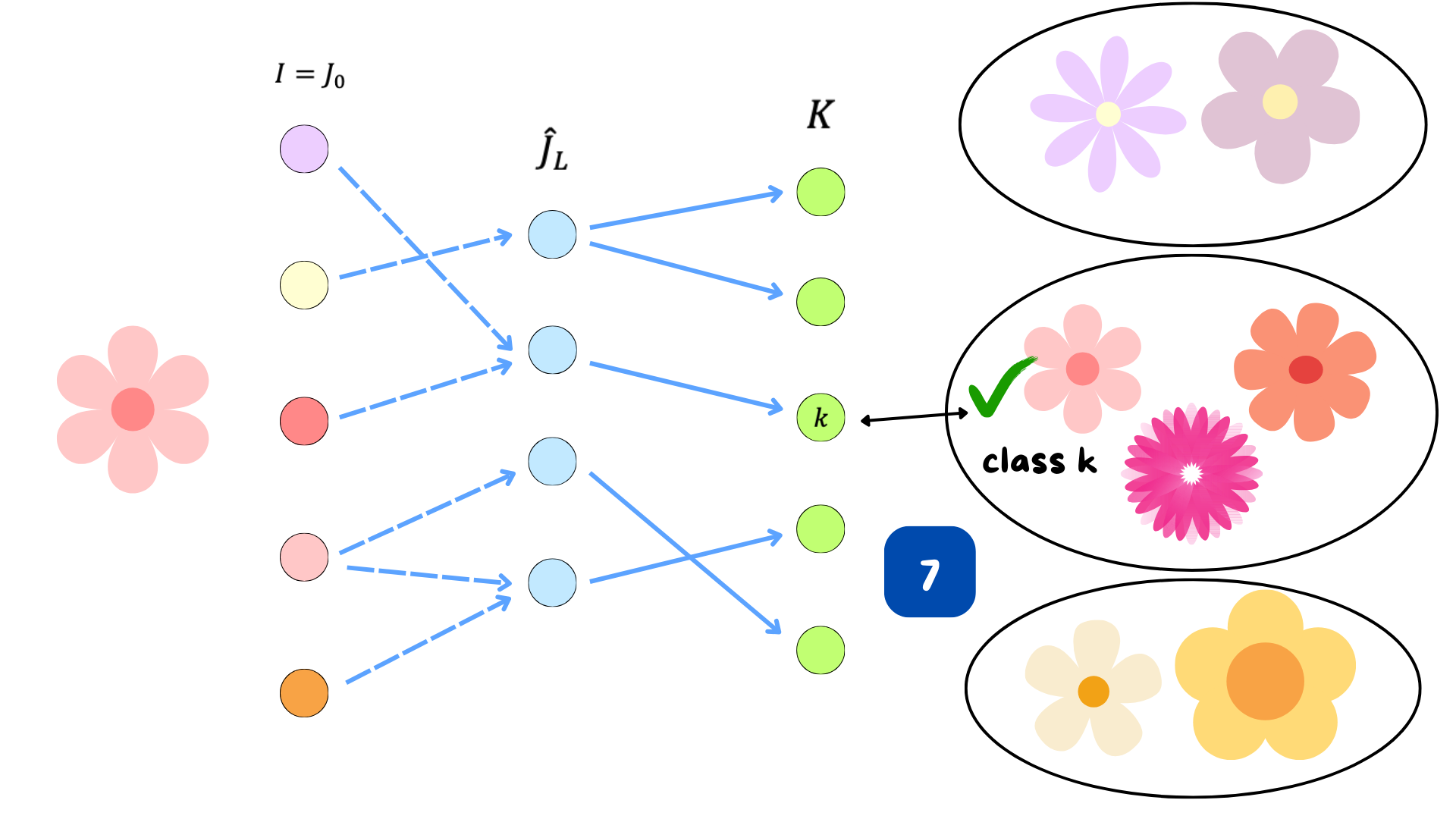}
\end{tabular}
\caption{Sketch of CHANI algorithm. {\bf A}. CHANI’s training consists of the sequential presentation of objects, with possible redundancy: here, objects 1 and 5 (respectively 3 and 6) share the same nature. Each object $m = 1, \dots, M$ or $N$ (depending on whether hidden layers are being trained or not) is presented for a duration $T$. {\bf B}. (Step 1) The network senses the presented object.
(Step 2) Depending on the presence of a given feature $i$ (e.g., color) in the object, neuron $i \in I$ emits more spikes during the presentation time $T$ (i.e., more “1”s in the sequence of length $T$ that it produces).
(Step 3) In the first hidden layer, each neuron $j_{\{i_1\} \cup \{i_2\}} \in J_1$ searches for correlations between features $i_1$ and $i_2$ in the presented objects and updates its synaptic weights $w_m^{i \to j}$ after processing the $m^{th}$ object. Ideally, at the end of training, it spikes only when the object contains both features $i_1$ and $i_2$.
(Step 4) Neurons that do not detect sufficient correlation across all presented objects during the first training phase (presentation of $M$ objects) are pruned. The resulting layer is denoted $\hat{J}_1$.
{\bf C}. (Step 5) As the layers are sequentially trained, a neuron $j_S$ in layer $J_{l-1}$ learns to detect the presence of all features $i \in S \subset I$ in the presented object. Ideally, at the end of training, this neuron spikes only when the object exhibits all features in $S$. The dashed arrows represent composite connections across several layers.
(Step 6) In layer $J_l$, a neuron $j_{S_1 \cup S_2}$ searches for correlations between neurons $j_{S_1}$ and $j_{S_2} \in \hat{J}_{l-1}$. Ideally, after learning, it spikes only when the presented object contains all features in $S_1 \cup S_2$. All layers are successively pruned to retain only the relevant correlations in the presented objects. {\bf D}. (Step 7) Once all hidden layers have been trained, the final layer $K$ is added. Each neuron $k \in K$ identifies the neurons $j_S \in \hat{J}_L$ that are the most sensible to objects of class $k$. At the end of training, ideally, neuron $k$ is connected to all $j_S \in \hat{J}_L $ such that the feature set $S$ is relevant for defining class $k$.
    }
\label{fig:Sketch}
\end{figure}

For a quantity $x_e$ indexed by $e\in E$, we use the following notations: $x_E = (x_e)_{e\in E}$ and $\brac{x_e}_{e\in E} =\frac{1}{\abs{E}}\sum_{e\in E} x_e$.
These notations are also used when the index appears as a superscript. The set of probability distributions over a set $E$ is denoted $\mathcal{P}_E$.
For all $n \in \mathbb{N}^*$, we denote $[n] = \{ 1, \dots, n \}$.
All the other notations are listed in Table \ref{tab_notations} in Appendix \ref{sec notations}.

\subsection{Notation relative to the objects presentation during learning }

We consider a supervised classification setting where the goal is to assign objects to classes $k \in K$ using a Spiking Neural Network (SNN).

More precisely, the SNN learns through sequential exposure to objects and is trained layer by layer. The $l^{\text{th}}$ hidden layer, for $l = 1, \dots, L$, is trained during round $l$, which consists of the sequential presentation of $M$ objects. Once all hidden layers have been trained, the output layer is trained during round $L + 1$ through the presentation of $N$ objects.

Throughout the paper, and for any round —including the one involving $N$ objects— each object is associated with an index $m$, which ranges from $1$ to $M$ for the rounds corresponding to the hidden layers ($l = 1, \dots, L$) and from $1$ to $N$ for the training of the output layer (round $L + 1$).

Each object $m$ in every round is presented for a duration $T \in \mathbb{N}^*$.


\medskip
\begin{tcolorbox}[breakable, width=\textwidth]
\begin{defi}[Time scales]
Figure \ref{fig:Sketch}.{\bf A} provides a schematic illustration of the successive presentation rounds.
There are two time scales: $M$ or $N$, representing the learning time and indexed by the presented object $m$; 
and a more microscopic scale, $T$, which corresponds to the time allowed for the network to spike and is 
indexed by time $t = 1, \ldots, T$.
\end{defi}
\end{tcolorbox}

\subsection{The input layer} \label{sec input}

\begin{tcolorbox}[breakable, width=\textwidth]
\begin{defi}[Sensing of the object]
When an object is presented to the network for a duration $T$, the input neurons sense the object and start to spike (see Figure \ref{fig:Sketch}.{\bf B}, Step 1).
These input neurons are tuned to detect features (see Figure \ref{fig:Sketch}.{\bf B}, Step 2). Depending on the presence or absence of features —or possibly on the degree to which a feature is present— the firing rate of these input neurons changes accordingly.
As a first approximation, one can assume that neuron $i$ remains silent if feature $i$ is absent, and that it spikes at a fixed rate if the feature is present. However, our network can also handle residual firing rates when a feature is absent, or firing rates that vary with the degree of presence of a given feature.
For instance, in the numerical results on classification of handwritten digits (Section \ref{sec:digit}), a feature corresponds to a pixel, and the firing rate is proportional to its darkness.

\end{defi}
\end{tcolorbox}


Neurons of the input layer are always indexed by a $i$.
A neuron of the input layer $i\in I$ produces a sequence of i.i.d. Bernoulli variables $X^{i,l}_{m,t}$, for $t=1,...,T$ during the presentation of object $m=1,...,M$ or $N$ in round $l=1,...,L+1$. This activity is independent of the previous activities of all the neurons of the input layer of the same learning session, \ie $X^{i,l}_{m,t}$ is independent from $X^{i',l}_{m,s}$ for $i' \in I$ and $s<t$. However we do not assume that $X^{i',l}_{m,t}$ is independent from $X^{i,l}_{m,t}$ for $i\not = i'$. 


\begin{tcolorbox}[breakable, width=\textwidth]
\begin{defi}[Nature versus object]
To ensure consistency and enable theoretical results, we need to establish a link between the firing distribution of the input neurons and the presented object. For this reason, we introduce a second concept: the nature of an object.
The nature of an object refers to its intrinsic  characteristics that are impacting the input layer. In particular, the input neurons respond in the same way (in a distributional sense) for all objects sharing that nature. The simplest way to think about this is to imagine that some objects represent exactly the same entity. For instance, in Figure \ref{fig:Sketch}.{\bf A}, the same flower is presented multiple times. If these different occurrences are considered as distinct objects (with different indices $m$), they share the same nature, and the firing rates of the input neurons should therefore be identical.
The nature of an object is indexed by $o \in \mathcal{O}$. The only thing the network can learn is a mapping from the set of natures $\mathcal{O}$ to the set of classes $K$, since only the nature of the object influences the activity of the input neurons.

\end{defi}
\end{tcolorbox}

Let $o_m^l$ and $o_{m'}^{l'} \in \Obj$ be the natures of two objects $m$ and $m'$ presented respectively during rounds $l$ 
and $l'$. If $o_m^l=o_{m'}^{l'}=o$ 
then since the distribution of the input neurons is {\bf nature-only} dependent, we have that
$$\forall t,s=1,...,T, \forall i\in I,  \mathbb{P}(X^{i,l}_{m,t}=1)=\mathbb{P}(X^{i,l'}_{m',s}=1).$$
We call this probability $p^i_o$, this is the firing rate of neuron $i$ when exposed to an object of nature $o$.

Note that the objects $o^l_m$ may be interpreted as training samples; however, we retain the term object to emphasize that, in our setting, the training samples have a particular structure. A training object $o^l_m$ will influence the input neurons in the same manner as a test object of the same nature once learning has taken place.

\begin{tcolorbox}[breakable, width=\textwidth]
\begin{defi}[About the features]
Remark that we use the same notation for the index $i\in I$ of a neuron in the input layer, and the feature $i\in I$ that this neuron senses. If we follow this logic, a feature $i\in I$ is the characteristic of the object that is modifying the firing rate of neuron $i$. For instance, it can be the presence/absence of a certain color, and here the features $i\in I $ are colors, as well as the neuron $i\in I$ that detects the presence/absence of this color (see Figure \ref{fig:Sketch}.{\bf B} Step 1 and 2). It can also be the level of gray of pixels and here the feature $i\in I$ is a  pixel as well as the neuron $i\in I$ that is sensible to the pixel intensity (see Section \ref{sec:digit}). 
\end{defi}
\end{tcolorbox}

\subsection{The hidden layers}

A neuron of a hidden layer is always indexed by a $j$. 

 The hidden layers are constructed recursively, starting from the input layer, which is assimilated to the  hidden layer $J_0$, with the correspondance that a neuron $i\in I$ can also be seen as a $j_{\{i\}} \in J_0$. In layer $J_0$, only singletons are allowed as subscripts.
 Then by recursion, for $l\geq 1$, the $l^{th}$ hidden layer consists in neurons $j_S$, where $S\subset I$ and there exists $S_1\neq S_2\subset I$ such that $j_{S_1}$ and $j_{S_2}$ are in layer $l-1$. By recursion, it follows that a $j_S\in J_l$ corresponds to an $S$ with cardinal $|S|=2^l$.
 
Not all possible sets of cardinal $2^l$ are used. Indeed the layers are trained recursively and pruned. More precisely, to fix notation, let $I_l$ be the set of all subsets of $I$ with cardinal $2^l$. If at the first hidden layer, $J_1$ can be put in  one-to-one correspondence with $I_2$, only a subset $\hat{J}_1$ is kept at the end of the training phase, so that $J_2$ corresponds only to union of pairs of sets appearing in  $\hat{J}_1$. More generally, for all $l$,  we only have $\hat{J}_l\subset J_l$ which can be put in one-to-one correspondence with a subset of $I_l$.


\begin{tcolorbox}[breakable, width=\textwidth]
\begin{defi}[About the indexation by a subset $S$]
In short, a neuron $j_S$ in a hidden layer is associated with a subset $S \subset I$ and aims to be sensitive to all objects that “possess" all features $i\in S$. After training, ideally $j_S$ spikes if and only if all neurons $i \in S$ are sufficiently active (see  Figure \ref{fig:Sketch}.{\bf C} Step 5). If before the training of layer $J_l$ a certain number of possible sets $S$ are envisioned, the layer is pruned and becomes $\hat{J}_l$ to consider only the sets $S$ that are really present in the data. For instance, if there is no blue and pink flower in the data that are presented to the network (see  Figure \ref{fig:Sketch}.{\bf B} Step 4) and if the training went well,  the neuron $j\in J_1$ that should detect the co-occurrence of pink and blue would not produce many spikes and would be removed before going to the training of layer $2$. This pruning happens after each learning round $l$, which  trains layer $l$  (see  Figure \ref{fig:Sketch}.{\bf C} Step 6).
\end{defi}
\end{tcolorbox}


During learning round $l$, at the presentation of object $m$, the activity of a neuron $j$ of an hidden layer $l'\leq l$, at time $t=1,...,T$,  is given by  $X_{m,t}^{j,l} \in \{0,1\}$. It obeys a non-linear discrete Hawkes process, with random synaptic weights. More precisely, if we consider the filtration
 $$\mathcal{F}^l_{m,t} = \sigma\left( X_{m',s}^{j,l'},~ m'\leq m,~s\leq T \text{ or } s\leq t \text{ if } m' = m,~j \in \bigcup_{l"=0,..., l'-1} \hat{J}_{l"} \cup J_{l'},~l'\leq l\right),$$
which represents the $\sigma$-algebra generated by  all the activity of all neurons, in layer less than $l$,  during all the learning rounds $l'\leq l$ until the time $t$ of the presentation of object $m$ in round $l$,
we have that, for a neuron $j\in J_l$,
\begin{equation}\label{defpjm}
p^{j,l}_{m,t}(w^j_m):= \mathbb{P}(X_{m,t}^{j,l}=1| \mathcal{F}^l_{m,t-1})=(\psi^{j,l}_{m,t}(w^j_m))_+
\end{equation}
where
\[\psi^{j,l}_{m,t}(w^j_m):= -\nu + w^j_m \cdot X^{\hat{J}_{l-1},l}_{m,t-1}\]
where

\noindent $\bullet$ $w^j_m =(w^{j'\to j}_m)_{j'\in \hat{J}_{l-1}} \in \mathcal{P}_{\hat{J}_{l-1}}$, are the random presynaptic weights of neuron $j$, where $\mathcal{P}_{\hat{J}_{l-1}}$ is the set of probability vectors  over the set $\hat{J}_{l-1}$. They are changing at each object $m$ that is presented (see Section \ref{sec:weight}) and they are $\mathcal{F}^l_{m-1,T}$ measurable;

\noindent $\bullet$ $\cdot$ refers to the usual scalar product in $\R^{\abs{\hat{J}_{l-1}}}$;

\noindent $\bullet$  the quantity  $\nu \geq 0$ combined with $(\cdot)_+$ acts as a threshold, providing non-linearity to the network response (see Section \ref{sec discuss nu} for a discussion 
about its role);

\noindent $\bullet$  with the convention $X^{j',l}_{m,0}=0$. 

To simplify the exposition, inhibition was not incorporated into the model, although the framework could be extended to include it without difficulty. Self-interactions were also omitted in order to facilitate the theoretical analysis

Moreover at the end of the learning phase $l$,  the weights are frozen to their value $w^j_{M+1}$ for all further use, so that the activity of a neuron $j\in \hat{J}_{l'}$ with $0<l'<l$ during learning round $l$, is given by
$$p^{j,l}_{m,t}(w^j_{M+1}):= \mathbb{P}(X_{m,t}^{j,l}=1| \mathcal{F}^l_{m,t-1})=(\psi^{j,l}_{m,t}(w^j_{M+1}))_+,$$
where the $w^j_{M+1}$ have been computed during the learning round $l'<l$.

\begin{tcolorbox}[breakable, width=\textwidth]
\begin{defi}[About the presynaptic weights and $\nu$]
The presynaptic weigths $w^{j'\to j}_m$, with $j'\in \hat{J}_{l-1}$ and $j\in J_l$  (see Figure \ref{fig:Sketch}.{\bf B} Step 3) are there to link the activities of the layer $l-1$ to the one of layer $l$. They evolve at each new object thanks to a local learning rule described later to mimic biological mecanisms (see Section \ref{sec:weight}), so that in particular during the learning phase, the activity of neuron $j$ is not nature-only dependent.  The quantity $\psi^{j,l}_{m,t}(w^j_{m})$ can be seen as a very crude model of the membrane voltage of a neuron submitted to synaptic integration and in this sense, $\nu$ acts as  a resting potential at which the neuron cannot spike. It is only if excited enough by its presynaptic neurons, that neuron $j$ can spike.
\end{defi}
\end{tcolorbox}


At the end of learning round $l$, and before starting learning round $l+1$, one presents each object $o\in \mathcal{O}$ only once. This is the selection phase. Therefore the activity of all neurons $j$ in  layer  $l'\leq l$ is denoted $X^{j,l'}_{o,t}\in \{0,1\}$ and it obeys
\begin{equation}
    \label{poj}
\mathbb{P}(X_{o,t}^{j,l'}=1| \mathcal{F}^l_{M,T}\mbox{ and } X^{\hat{J}_{l'-1},l}_{o,1:t-1} )=(-\nu + w^j_{M+1} \cdot X^{\hat{J}_{l'-1},l}_{o,t-1})_+.
\end{equation}

For every $l\in [L]$,
a threshold $s_l \in (0,1]$ is fixed. At the end of the selection phase of layer $l$, the set of selected neurons is
\begin{equation}\label{thresh}
\hat{J}_l := \{j\in J_l \text{ s.t. } \exists o\in \mathcal{O}, \brac{X^{j,l}_{o,t}}_{t\in[T]} \geq s_l\}.
\end{equation}
In other words, the selected neurons $\hat{J}_l$ are the ones with an empirical spiking probability $\brac{X^{j,l}_{o,t}}_{t\in[T]}$ larger than the threshold $s_l$ for at least one object of the selection phase. This corresponds to step $11$ of Algorithm \ref{algo}. Then for $l<L$, the set of neurons used to train the next layer is $$J_{l+1}:= \{j_{S_1\cup S_2} \ : \ j_{S_1},j_{S_2}\in \hat{J}_l \ \text{such that} \ S_1\cap S_2 = \emptyset\},$$ which is coding for feature subsets belonging to $I_{l+1}$. This selection encourages sparsity in the network in terms of the number of neurons.

\begin{tcolorbox}[breakable, width=\textwidth]
\begin{defi}[About the biological interpretation] \label{intuit bio}
Our model is a caricature of some phenomenons in the brain. Indeed, if our learning rounds are very strict, it is true that during development, connections between different part of the brain are done in a sequential fashion \citep{hevner2000development}. It is also true that during development and in particular, in the visual system (retina/ cortex), that some dummy signals, called retinal waves,  are used to pretrained part of the system in particular in terms of correlations between presynaptic neurons and connections towards postsynaptic neurons \citep{feller2020retinal}. Finally it is also true that  a phenomenon called pruning happens at the synapse level so that the brain is progressively loosing some neuronal connections and that this is an inherent part of the learning process \citep{lewis2011pruning}.  
\end{defi}
\end{tcolorbox}

\subsection{The output layer} \label{sec output}

A neuron of a hidden layer is always indexed by a $k$ and there is a one-to-one correspondance between the output neurons and the classes in which the objects are classified, so that we denote both by the same letter $k\in K$.

Once all the $L$ hidden layers have been trained, we start the $L+1$ learning round, ie we present $N$ objects to the network to learn the final classification. 
During this round, at the presentation of object $m$, the activity of a neuron $k$  is given by $ X_{m,t}^{k,L+1} \in \{0,1\}$  and  the activity of a neuron $j$ of an hidden layer $l'\leq L$, at time $t=1,...,T$,  is given by  $X_{m,t}^{j,L+1} \in \{0,1\}$. The distribution of the hidden layers is as before, with weights fixed to $w_{M+1}^j$, whereas the distribution of $X_{m,t}^{k,L+1}$ is given by
\begin{equation}\label{formula proba output}
 p^k_{m,t}(w^k_m):= \mathbb{P}(X_{m,t}^{k,L+1}=1| \mathcal{F}^{L+1}_{m,t-1})  =  w^k_m \cdot X^{\hat{J}_{L},L+1}_{m,t-1}.
\end{equation}
Again, the synaptic weights $w^k_m$ form a probability distribution over the set of presynaptic neurons so $ p^k_{m,t}(w^k_m)\in [0,1]$ a.s. At the end of the presentation of each objects, the weights $w^k_m$ are updated and are $\mathcal{F}^{L+1}_{m-1,T}$ measurable (see Section \ref{sec:weight}).
It corresponds to the Solo case of \cite{jaffard:hal-04065229}.
At the end of learning, the weights are frozen and we go back to a nature-only dependent behavior of the output, so that at the presentation of object with nature $o$, we have
\begin{equation}
    \label{pok}
\mathbb{P}(X_{o,t}^{k,L+1}=1| \mathcal{F}^{L+1}_{M,T}\mbox{ and } X^{\hat{J}_{l'-1},L+1}_{o,1:t-1} )=(-\nu + w^j_{M+1} \cdot X^{\hat{J}_{l'-1},l}_{o,t-1})_+.
\end{equation}

The rule to classify objects is as follows:  the object $o_m$ is classified by the network in the class coded by the output neuron $k$ which spiked the most during the presentation of the object.  

\begin{tcolorbox}[breakable, width=\textwidth]
\begin{defi}[Classification rule and identity of neurons]
\label{idneur}
The output layer aims at describing the several classes in which the objects are classified.  Ideally, after its training session, a neuron $k$ would be active when an object belonging to class $k$ is presented to the network, and non-active otherwise. The eventual activity of neuron $k$ when presented with an object which does not belong to its class is considered as noise. Therefore, to correctly classify the object, the neuron coding for its class must overcome the noise coming from other neurons. However, since we are in a local learning rule set-up, neuron $k$ does not have access, when modifying its weights, to the output of the other neurons $k'\neq k$. It has only access to its presynaptic activity as well as feedback from the environment, which is if the presented object belongs or not to class $k$ (see Figure \ref{fig:Sketch}.{\bf D} (Step 7)). Biologically speaking, it is true that some neurons might have access, for instance by dopamine secretion,  to at least a notion of success / failure \citep{fremaux2016neuromodulated, schultz1998predictive}.

More generally, the fact that we associate a class $k$ and a neuron $k$ or a set $S\subset I$ and a hidden neuron $j_S$ is completely artificial. It is more likely that at random, many neurons get attributed  to sets $S$ or to class $k$ with possible redundancy, so that in the end, at least most of the current classes $k$ and set of features $S$ are represented at least once. To avoid additional steps, we made the identifications and decided to have one and only one represent of class $k$ and set $S$ in the family of neurons.

\end{defi}
\end{tcolorbox}

\subsection{Learning and local rules}

\subsubsection{Correlations}
The hidden layers learning is based on correlation. We will need several notions of correlations. Let $I'\subset I$ a subset of features, $o\in \mathcal{O}$. Then the \emph{correlation of $I'$ w.r.t. object} $o$ is
    \[
    \rho_o(I'):=\mathbb{P}\Big(\bigcap_{i\in I'} \{X^{i,1}_{o,T} = 1\}\Big),
    \]
    \ie the probability that the input neurons of the set $I'$ spike together, the \emph{average correlation} of $I'$ is
    $
    \rho(I') := \brac{\rho_o(I')}_{o\in \Obj} ,
    $
    \ie the average correlation of the input neurons of the set $I'$ on every object and the \emph{average correlation of $I'$ w.r.t class} $k\in K$ is
    $
    \rho^k(I') := \brac{\rho_o(I')}_{o\in k},
    $
     \ie the average correlation of the input neurons of the set $I'$ on objects of class $k$.
Let $l\in [L]$, $m\in [M]$, $J$ be a subset of neurons belonging to a layer with depth less than $l$. Then the \emph{empirical correlation} of $J$ when presented with object $o_m^l$ is 
    \[
    \hat{\rho}_m^l(J):=\Brac{\prod_{j\in J} X^{j,l}_{m,t}}_{t\in[T]},
    \]
which is an estimator of the probability that the neurons of the set $J$ spike together when presented with object $o_m^l$.

\subsubsection{Expert aggregation}
The learning rules that are used to update the weights are based on expert aggregation \citep{cesa2006prediction}. Let us recall the main features and the corresponding  notation.
Here is a description of the expert aggregation problem. During $M$ rounds, a forecaster can choose between several experts belonging to a set $E$. At each step $m$, each expert $e\in E$ has an unknown gain $g_m^e\in \mathbb{R}$. Thanks to the past knowledge, the forecaster chooses a probability distribution $p_m \in [0,1]^{\abs{E}}$ over the set of experts $E$, and then receives the aggregated gain $g_m := p_m \cdot g_m^E$
where $g_m^E:=(g_m^e)_{e\in E}$. This aggregated gain can be interpreted as the expectation of the gain that the forecaster would get by choosing one expert with probability $p_m$. Experts accumulate cumulated gains $G^E_m := (G^e_m)_{e\in E}$ were $G^e_m := \sum_{m'=1}^m g^e_{m'}$, as well as the forecaster which accumulates the cumulated gain $G_m := \sum_{m'=1}^m g_{m'}.$
The {\it regret} of the forecaster measures how good its strategy is. It is defined as
\[R_M := \max_{q\in \mathcal{P}_E} \sum_{m=1}^M q \cdot g^E_m - G_m. \]
It compares the best possible cumulated gain with a constant strategy to the actual cumulated gain of the forecaster. It quantifies how close the strategy of the forecaster is to the optimal one.

An expert aggregation algorithm is a function $f : \mathbb{R}\times \mathbb{R}^{\abs{E}} \mapsto [0,1]^{\abs{E}}$ that the forecaster uses to update its strategy: the probability distribution chosen by the forecaster for the next round is
\[p_{m+1} := f(G_m, G^E_m).\]
Expert aggregation algorithms are designed to achieve small regret bounds, typically of the order of $\sqrt{M}$. 
Let us cite two main algorithms:

\noindent $\bullet$ EWA (Exponentially Weighted Average) determines the probability of selecting expert $e$ at round $m$ as
\[p^e_{m+1} = \frac{\exp(\eta G^e_m)}{\sum_{e'\in E} \exp(\eta G^{e'}_m)}.\]
The parameter $\eta\geq 0$ is the learning rate. It quantifies how fast the forecaster learns.

\noindent $\bullet$  PWA (Polynomially Weighted Average) determines the probability of selecting expert $e$ at round $m$ as
    \[p^e_{m+1} = \frac{(G^e_m - G_m)_+^{b-1}}{\sum_{e'\in E} (G^{e'}_m - G_m)_+^{b-1}} , \]
where $b \geq 2$ is a parameter to choose.

These two expert aggregation algorithms employ different strategies: EWA ensures that each expert is assigned a strictly positive probability of selection, whereas PWA assigns a zero probability to experts whose cumulative gains are less than the forecaster's.

\subsubsection{Weights update \label{sec:weight}}
Each neuron $j$ of an hidden layer $l$ learns by updating its synaptic weights $w^j_m = (w^{j'\to j}_m)_{j'\in \hat{J}_{l-1}}$ during its learning round $l$ by running its own expert aggregation algorithm $f^l$, indexed by $l$: the neuron $j$ is the forecaster and the corresponding set of experts is the set of presynaptic neurons $\hat{J}_{l-1}$. A presynaptic neuron $j'$ is attributed a gain $g^{j'\to j}_m \in \R$ for neuron $j$  and the weights of $j$ evolve according to the rule
\[w^j_{m+1} = f^l(G^j_m, (G^{j'\to j}_m)_{j'\in \hat{J}_{l-1}})\]
where $G^j_m:= \sum_{m'=1}^m w^j_{m'} \cdot g^j_{m'}$ is the cumulated gain of thepostsynaptic neuron $j$ and $G^{j'\to j}_m:= \sum_{m'=1}^m g^{j'\to j}_{m'}$ is the cumulated gain of the presynaptic neuron $j'$. This corresponds to steps $7$ and $18$ of Algorithm \ref{algo}.

Most of the theoretical findings regarding expert aggregation algorithms apply to arbitrary bounded sequences of gains. Therefore, we can choose any bounded gains suitable for the learning task. For a hidden neuron $j_{S_1\cup S_2}\in J_l$ where $S_1$ and $S_2$ are feature subsets encoded in the previous layers, the gain of presynaptic neuron $j'$ is defined as
\begin{equation} \label{gain hidd}
    g^{j'\to j_{S_1\cup S_2}}_m := \hat{\rho}_m^l(\{j_{S_1},j_{S_2},j'\}) 
\end{equation}
where $\hat{\rho}_m^l(\{j_{S_1},j_{S_2},j'\}) = \brac{X^{j_{S_1},l}_{m,t}X^{j_{S_2},l}_{m,t}X^{j',l}_{m,t}}_{t\in[T]} \in [0,1]$ is the empirical correlation between neurons $j_{S_1}$, $j_{S_2}$ and $j'$. This choice of gain is designed to achieve the objective of neuron $j$ training, which is to detect correlations between neurons $j_{S_1}$ and $j_{S_2}$.

For an output neuron $k\in K$, the same set-up applies and we extend the gains defined in \cite{jaffard:hal-04065229}.
The gain of presynaptic neuron $j$ is
\begin{equation} \label{eq gain output}
    g^{j\to k}_m := \begin{cases}
          \brac{X^{j,L+1}_{m,t}}_{t\in[T]} \times \frac{N}{N^k} &\text{if } o_m^{L+1}\in k  \\
        -  \brac{X^{j,L+1}_{m,t}}_{t\in[T]} \!\times\! \frac{N}{N^{k'}} \!\times\! \frac{1}{\abs{K}-1} &\text{if } o_m^{L+1} \in k'\neq k
    \end{cases}
\end{equation}
where $N^{k'}$ is the number of objects belonging to class $k'$ used to train the set of output neurons $K$. Therefore,
when the presented object is in class $k$, \ie when neuron $k$ should spike more than the other output neurons, it attributes positive
gains to active presynaptic neurons. When the presented object is in another class, \ie when neuron $k$ should
 stay silent, it attributes losses (negative gains) to active presynaptic neurons. During the training phase of the output layer, the hidden layer $L$ is already trained so the synaptic 
weights of a neuron $j\in \hat{J}_L$ are $w^j_{M+1}$.

\begin{tcolorbox}[breakable, width=\textwidth]
\begin{defi}[Local rules]

These choices of gains establish local rules:  each neuron within a layer operates its own expert aggregation algorithm independently of the other neurons.  The synaptic weights evolution of a neuron is solely influenced by the spikes of its presynaptic neurons.  There are no backward passes, and the information regarding the true class of the presented object only plays a role in the output layer learning rule.

For a neuron $j$ in the hidden layers, the updates of the weights are solely based on correlations between presynaptic neurons. To some extent, it can be seen as a three-neurons pattern,  trying to detect if one neuron responds in synergy with a couple of fixed neurons. If this rule is not close to any rules reported in the biological literature, up to our knowledge, it has the advantage to explain to some extent why synchronization of activity is so important in practice in the neuronal code \citep{singer1997neuronal}. 

By focusing on neural correlations and by applying the pruning process, the hidden layers are trained to learn feature correlations that are strongly expressed across the object natures in the set $\Obj$. Consequently, some of the learned correlations may not be directly relevant to the classification task, even if they are informative for characterizing the set 
$\Obj$. This limitation is unavoidable unless explicit class-related information is provided to the hidden layers.

An output neuron $k$ also evolves thanks to a local learning rule that only depends on the presynaptic neurons behavior and is modulated by a reward factor. In this sense it looks close to the three-factors learning rule introduced by \cite{fremaux2016neuromodulated}.

However note that, despite their local character,  none of the rules we are applying are explicitly phrased as a gradient descent (no $\dot{w}$, derivative of $w$) and none of them take into account the postsynaptic firing rate per se, as it is usual in the biological learning rules model that have been studied in the literature. 
\end{defi}
\end{tcolorbox}

\subsection{The complete CHANI algorithm}

The overall CHANI algorithm is summarized in Algorithm \ref{algo}. Its complexity is determined by the number of calls made to the pseudo-random generator for obtaining Bernoulli variables. We need to simulate
\[
(M+N)T( \abs{I} + \max_{l=1,\dots,L} \abs{\hat{J}_l} )+ T \max_{l=1,\dots,L} \abs{J_l}
\]
random variables. Here the depth $L$ is considered a constant.




\begin{algorithm}
\nl
\textbf{Initialization:} $\hat{J}_0 := I$ \\
\nl
\For{$l=1$ {\bfseries to} $L$}{ \nl
    \textbf{Initialization:} $J_l :=  \{j_{S_1\cup S_2} \ : \ j_{S_1},j_{S_2}\in \hat{J}_{l-1} \ \text{such that} \ S_1\cap S_2 = \emptyset\}$, $\hat{J}_l := \emptyset$, $w^{j' \to j} := 1/\abs{\hat{J}_{l-1}}$ for $j\in J_l$ \\
\nl
\For{$m=1$ {\bfseries to} $M$}{ \nl
Simulate $(X^{I,l}_{m,t})_{t\in [T]}$, $(X^{\hat{J}_1,l}_{m,t})_{t\in [T]}, \dots, (X^{\hat{J}_{l-1},l}_{m,t})_{t\in [T]}$ according to Section \ref{sec input} and \eqref{defpjm} .  \\ \nl
Compute gains $g^{J_l}_m$ according to \eqref{gain hidd}. \\ \nl
\For{$j \in J_l$}{ \nl
$w^{j}_{m+1} \xleftarrow[]{} f^l(G^j_m, (G^{j'\to j}_m)_{j'\in \hat{J}_{l-1}})$ \tcp{aggregate experts}
}
}\nl
\For{$o\in\Obj$}{ \nl
Simulate $(X^{I,l}_{o,t})_{t\in [T]}$, $(X^{\hat{J}_1,l}_{o,t})_{t\in [T]}, \dots, (X^{\hat{J}_{l-1},l}_{o,t})_{t \in [T]}$, $(X^{J_l,l}_{o,t})_{t\in [T]} $ according to Section \ref{sec input} and \eqref{poj}.  \\ \nl
\For{$j \in J_l$}{ \nl
\If{ $\brac{X^{j,l}_{o,t}}_{t\in[T]} \geq s_l$ and $j\notin \hat{J}_l$}{ \nl
add $j$ to $\hat{J}_l$ \tcp{select neuron $j$}
}
}

}

}\nl
    \textbf{Initialization:} $w^{j\to k} := 1/\abs{\hat{J}_{L}}$ for $j\in \hat{J}_{L}$ \\ \nl
\For{$m=1$ {\bfseries to} $N$}{ \nl
Simulate $(X^{I,L+1}_{m,t})_{t\in [T]}$, $(X^{\hat{J}_1,L+1}_{m,t})_{t\in [T]}, \dots, (X^{\hat{J}_L,L+1}_{m,t})_{t\in [T]}$ according to Section \ref{sec output}.  \\ \nl
Compute gains $g^{K}_m$ according to \eqref{eq gain output}. \\ \nl
\For{$k \in K$}{ \nl
$w^{k}_{m+1} \xleftarrow[]{} f^{L+1}(G^k_m, (G^{j\to k}_m)_{j\in \hat{J}_{L}})$ \tcp{aggregate experts}
}}
    \vspace{-1mm}
    \textbf{Output}: $\hat{J}_1,\dots, \hat{J}_L, (w^j_{M+1})_{j\in \hat{J}_1\cup \dots\cup \hat{J}_L}, (w^k_{N+1})_{k\in K} $ \tcp{selected neurons and final weights}
    
   
    \caption{CHANI}
    \label{algo}
\end{algorithm}

\subsection{Measure of the network performance}

\subsubsection{Spiking probability discrepancies}

Since the classification rule of the network is based on who spikes the most, we measure CHANI's performance via discrepancies between firing rates or correlations. All these notions are represented in Figure \ref{fig:disc}.
\begin{figure}
    \centering
\includegraphics[width=1\linewidth]{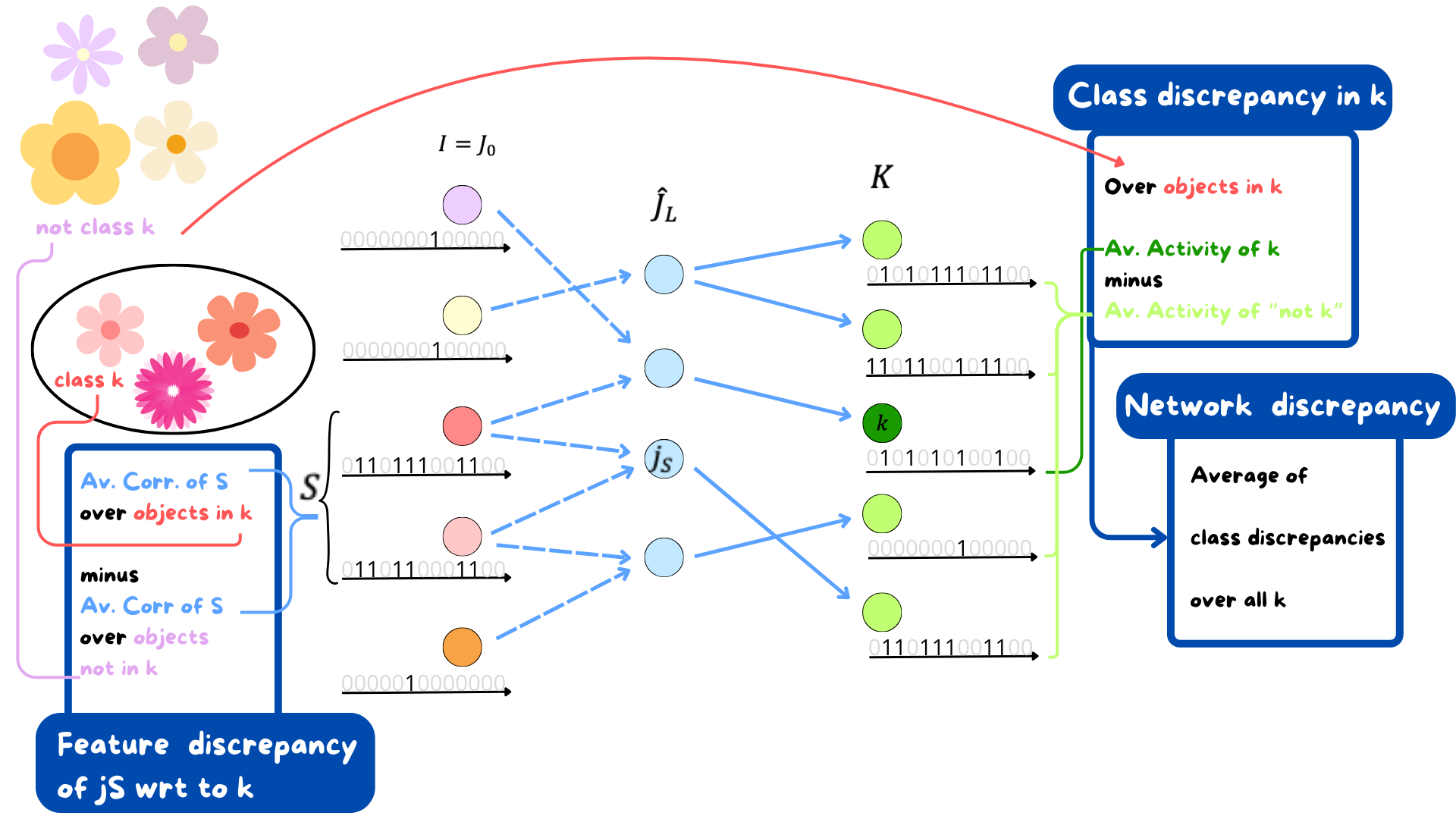}
\caption{\label{fig:disc} The three sorts of discrepancies. Av. stands for average.}
\end{figure}

First, we extend the notion of class discrepancy introduced in \cite{jaffard:hal-04065229}, which compares the activity of an output neuron to the rest of the layer, as follows. We  define {\it the average activity of neuron $k$} connected with weights $q^k$ to the final hidden layer $\hat{J}_L$ during the presentation of object $m=1,...,N$ by
$$ \hat{p}_m^k(q^k) = q^k \cdot \brac{X_{m,t}^{\hat{J}_L,L+1}}_{t=1,...,T}.$$

We want to know the performance of this choice of weights for class $k$. Hence, we  introduce for a neuron $k\in K$ in the ouput layer, the {\it class discrepancy} of neuron $k$ with weights $q^k\in \mathcal{P}_{\hat{J}_L}$:
      \[\Disc_N^k(q^k, X_{[N],[T]}^{\hat{J}_L,L+1}) :=\Brac{\hat{p}^k_m(q^k) - \hat{p}^{k'}_m(q^k) }_{\begin{subarray}{l}
      k'  \ / \ k'\neq k \\  m  \ / \ o^{L+1}_m\in k 
    \end{subarray}} \in \R .
    \]
It compares the average activity of neuron $k$ with the average activity of other neurons when presented with objects of class $k$. The larger the class discrepancy is, the more neuron $k$ overcomes the noise in order to correctly classify the objects of its class. It is a \textbf{global notion} since it involves the activity of all the output neurons. 

The average class discrepancy of the network is called the {\it network discrepancy} with weights $q^K=(q^k)_{k\in K}$
\begin{equation}\label{netdisc}
    \Disc_N(q^K, X_{[N],[T]}^{\hat{J}_L,L+1}) := \brac{ \Disc_N^k(q^k, X_{[N],[T]}^{\hat{J}_L,L+1})}_{k\in K} \in \R.
\end{equation}
The network discrepancy measures the average ability of the network to classify objects from all the classes. 




To explain how the weights are strengthen, we need to find the sets $S$ of features that are the most characteristic of a given class $k$. To do so, we extend the notion of feature discrepancy defined in \cite{jaffard:hal-04065229} for an input neuron coding for a feature to a neuron of the last hidden layer coding for a set of features.
Let $j$ be a neuron of the last hidden layer $L$, coding for a feature subset $S\in I^{L}$. The {\it feature discrepancy} of $j = j_S$ w.r.t. class $k\in K$ is the quantity
\begin{equation}\label{featdisc}
\Disc^{j_S\to k} := \rho^k(S) - \brac{\rho^{k'}(S)}_{k'\neq k} \in \R.
\end{equation}
The feature discrepancy of neuron $j_S$ quantifies how much the correlation of features of the set $S$ is relevant to class $k$: it compares the average correlation of the input neurons coding for the features belonging to the set $S$ w.r.t. class $k$ with its average correlation w.r.t. other classes (see Figure \ref{fig:disc} for an illustration). This new definition of feature discrepancy is much more complex than that of \cite{jaffard:hal-04065229}. In our previous work, the feature discrepancy of an input neuron concerned only its individual activity, with no link to the activity of other input neurons whereas here it can capture complex patterns of input neurons activity.


\subsubsection{Ideal network performance}

In order to establish a performance target for CHANI, we define below ideal layers representing ideal performances.

Let us consider a network with conditional spiking probabilities given by \eqref{poj} and \eqref{pok}, with synaptic weights given by a deterministic, fixed family $q$ on the whole network. In this case, the distribution of a hidden or output neuron $a$ is nature-only dependent and its activity when presented with an object of nature $o$, denoted $(X^a_{o,t})_{t\geq 1}$, is a sequence of Bernoulli variables. In particular,  we denote
$$p^a_o(q^a)=\mathbb{P}(X^a_{o,T}=1),$$
that is the {\it expected mean activity} of neuron $a$ once all the dependencies with respect to the previous layers are taken into account (with $T>L+1$).

\begin{definition}[Ideal activities: layers and network] \quad

\noindent $\bullet$ For all $l\in [L]$, a neuron $j=j_S$ associated to the subset $S\subset I$ of the hidden layer $l$ has an {\bf ideal activity} with constant $\gamma$ if 
$$\forall o \in \Obj, \quad p^j_o(q^j)=\new{\gamma} 
\rho_o(S). $$
A subset $H_l$ of the hidden layer $l$ with weight $q^{H_l}$ is said to be {\bf ideal} with 
constant $\gamma_{l}$ if all its neurons have an ideal activity with the same constant $\gamma_{l}$.

\noindent$\bullet$ A neuron $k\in K$ of the output layer has an {\bf ideal activity} if 
$$ \forall o \in \Obj, \quad p^k_o(q^k) >0 \text{ if and only if } o\in k.$$
The output layer with weight $q^K$ is said to be {\bf ideal} if all its neurons have an ideal activity.

\noindent$\bullet$ The network with weights $q$ is said to be {\bf ideal} when all its hidden and output neurons have an ideal activity.

\end{definition}

\begin{tcolorbox}[breakable, width=\textwidth]
\begin{defi}[About the ideal notion]

The previous notions formalize that in an ideal world,  

\noindent$\bullet$ a hidden neuron $j=j_S$ encoding a subset $S$ of features should be active when all input neurons $i\in S$ are active (\ie when the presented object has all the features of the set $S$ simultaneously), and its spiking probability should be proportional to the probability that they activate together.

\noindent$\bullet$ an output neuron $k$ would code exactly for classes $k$ without any noise.

If being ideal might depend on the weights, the behavior of an ideal layer does not depend per se on the weights  and a priori we do not know if there exist weight families for which our network is ideal. 
\end{defi}
\end{tcolorbox}


To be able to compare more precisely our weights to the best possible weights we need therefore a measure that depends on the weights: the {\it ideal discrepancy}. It measures the performance of a fixed family of weights on the output layer $q^K\in (\mathcal{P}_{H_L})^{\abs{K}}$, 
when connected to an ideal $H_L$, subset of the last layer,  with constant 1:
\begin{equation}\label{discID}
\Disc^{\id}(q^K,H_L) := 
    \brac{ \bar{p}^{k}_o(q^k,H_L)  - \bar{p}^{k'}_o(q^{k'},H_L)  }_{\begin{subarray}{l}
k\in K\\o\in k \\ k'\neq k 
\end{subarray}}
\end{equation}
 with for every $k\in K$, $\bar{p}^{k}_o(q^k,H_L) := q^k \cdot (\rho_o(S(j)))_{j\in H_L}$, where $S(j)$ denotes the subset $S$ of features such that $j=j_S$. This quantity measures the average difference between the spiking probability of an output neuron when presented with an object of its class and the other output neurons spiking probabilities: the larger it is, the better the performance of the output layer at classifying the objects. This notion is similar to the network discrepancy, but holds for an arbitrary output weight family connected to an ideal hidden layer.
 
If the output layer is connected to the ideal hidden layer $H_L$ with an arbitrary constant $\gamma_L$, then by linearity the corresponding quantity is $\gamma_L\Disc^{\id}(q^K, H_L)$. 

Note that in practice the last layer $\hat{J}_L$ is random. With $H_L$, we are considering a deterministic possible subset of neurons of the last layer, that can be put in a one-to-one map with a subset of $I_L$.

\begin{definition}[Feasible weight family]
    Let $H_L$ be a set of neurons coding for feature subsets belonging to $I_L$. We say that $q^K\in (\mathcal{P}_{H_L})^{\abs{K}}$, is 
    
    \noindent $\bullet$ a {\bf feasible weight family} w.r.t $H_L$ when
     \[
    \Disc^{\id}(q^K,H_L) >0
    \]
    We denote by 
    $\mathcal{Q}_{H_L}$ the set of feasible weight families w.r.t. $H_L$.
    
    \noindent $\bullet$ a {\bf strong feasible weight family} w.r.t. the set $H_L$ when 
    \[
    \forall k\in K, \forall o \in \Obj, \quad \bar{p}^{k}_o(q^k,H_L) > 0 \quad \text{if, and only if,} \quad o\in k,
    \]
    \ie when the output layer $K$, connected to the ideal layer $H_L$ with constant $1$ and with weights $q^K$, is ideal. 
\end{definition}


\begin{tcolorbox}[breakable, width=\textwidth]
\begin{defi}[About the feasible weight families.]
The quantity $\bar{p}^{k}_o(q^k,H_L)$ is the expected mean activity of neuron $k$  when connected with weights $q^K$ to an ideal hidden layer with constant 1.

Therefore the ideal discrepancy measures (in a deterministic fashion) the network discrepancy when the output layer is connected with weights $q^K$ to an ideal layer. The fact that there exists feasible weight families means that if one is able to have an ideal last hidden layer, then one can find weights $q^K$ to connect the last layer (encoding  features correlations) to the classes in such a way that the resulting firing rates of the neurons lead to  a good classification on average.

The notion of strong feasible weights, means that in the end, if one presents an object of class $k$ not only the neuron $k$ fires more than the other, but in fact it is the only one to fire. 
\end{defi}
\end{tcolorbox}


Note that if a weight family $q^K$ is a strong feasible family weight w.r.t. $H_L$ then it is a feasible family weight w.r.t. $H_L$.

Another important remark is the following, which links all the notions of ideal activities and feasible weights. If we are given a family $q$ of weights over the network and the last hidden layer $H_L$ is deterministic such that 

\noindent$\bullet$ $H_L$ is ideal with constant $\gamma_L$

\noindent$\bullet$ the weights of the output layer $q^K$ are strong feasible weights w.r.t. $H_L$

\noindent then the mean expected activity of an output neuron $k\in K$ satisfies
$$p_o^k(q^k)= \gamma_L \bar{p}^k_o(q^k,H_L)$$
and the output layer of the network is ideal.

\subsection{Analogy with Transformers} \label{sec transformers}
The key innovation of the transformer architecture~\citep{vaswani2017attention} is the self-attention mechanism, which allows the model to weigh the importance of different input tokens when processing each token in a sequence. This mechanism enables the model to capture long-range dependencies more effectively than traditional recurrent or convolutional neural networks. When using the specific expert aggregation EWA, CHANI shares similarities with transformers.

\noindent\textbf{Input embedding.} Objects to be classified have features which are embedded into point processes by the input layer. 

\noindent\textbf{Attention layers.} Let $l\in [L]$, $j = j_{S_1\cup S_2}\in J_l$ where $S_1$ and $S_2$ are feature subsets encoded in the previous layer, and $m\in [M]$. Then, the conditional spiking probability of neuron $j$ during the presentation of object $o_m^l$ of its learning phase reads
\[
p^{j,l}_{m,t}(w^j_{m}) = \softmax(\eta^j A^j_{m-1} \cdot B_{m-1} )\cdot X^{\hat{J}_{l-1},l}_{m,t-1}
\]
where $X^{\hat{J}_{l-1},l}_{m,t-1} = (X^{j',l}_{m,t-1})_{j'\in \hat{J}_{l-1}}$ is the activity of the previous layer at the previous time step of current object $o_m^l$, $A^j_{m-1} := (X^{j_{S_1},l}_{m',t} X^{j_{S_2},l}_{m',t})_{\begin{subarray}{c}
    1\leq m' \leq m-1,
    1\leq t \leq T
\end{subarray}}$ is the vector of correlations between $j_{S_1}$ and $j_{S_2}$ until previous object $o^l_{m-1}$, $B_{m-1} := (X^{j',l}_{m',t})_{\begin{subarray}{c}
    j'\in \hat{J}_{l-1},
    m' \in [m-1] ,
    t \in [T]
\end{subarray}}$ is the vector of previous layer activity until object $o^l_{m-1}$, and $A^j_{m-1} \cdot B_{m-1} :=\Big(\sum_{\begin{subarray}{c}
   m' \in [m-1] \\ t \in [T]
\end{subarray}}X^{j_{S_1},l}_{m',t} X^{j_{S_2},l}_{m',t}X^{j',l}_{m',t}\Big)_{j'\in \hat{J}_{l-1}}$ is the matrix of correlations between $j_{S_1}$, $j_{S_2}$ and neurons of the previous layer until previous object $o^l_{m-1}$. Therefore, the hidden layer of depth $l$ acts as an attention layer enlightening correlations between neurons of the previous layers. At the end of the learning phase, the resulting weight matrix $w^{J_L}_{M+1}$ can be seen as a score matrix which is then used to select neurons coding for high correlations. The accumulation of hidden layers is analogous to a succession of attention layers. 

\noindent\textbf{Final linear layer.} The output layer acts as a linear layer at the end of an attention module as the conditional spiking probability is linear in the activity of the last hidden layers.

\section{Main theoretical results} \label{sec theo results specif set} 

In this section, we present the principal findings of the article by conducting a thorough analysis of the network’s performance. We begin by stating the main theorem, followed by a series of intermediate theoretical results that both support its proof and offer insights extending beyond the theorem itself.
\subsection{Main theorem} \label{sec main th}

Let us define by recursion the sets $\J_l$ for all $l\in \{0,\dots,L\}$ as follows, with $s_1,\dots,s_l$ the thresholds defined in \eqref{thresh}.

\noindent $\bullet$ $\J_0:= I$.

\noindent$\bullet$ For $l\geq 1$,
\[
\J_l := \left\{\, j_{S_1\cup S_2} \;\middle|\;
\begin{aligned}
&j_{S_1}, j_{S_2}\in\J_{l-1}\text{ such that }S_1\cap S_2=\emptyset,\\
&\text{and }\ \exists o\in\Obj:\ \rho_o(S_1\cup S_2)> 2^{2^l - 1} s_l
\end{aligned}
\,\right\}.
\]
For $l\geq 1$, the sets $\J_l$ are sets of neurons coding for high feature correlations for at least one object. They are the sets targeted during hidden layer pruning phases.
\newline

Consider the network defined with $\J_l$ as hidden layers, with fixed and deterministic weights defined by 
$$\forall j'\in \J_{l-1} \text{ and }j=j_{S_1\cup S_2} \in \J_l, \quad \bar{w}^{j'\to j}= \frac{1}{2} \mathbf{1}_{j'\in \{j_{S_1},j_{S_2}\}}.$$
Define for the output layer, the subsets
$$ \J^k_L=\arg\max_{j\in \J_L} \Disc^{j\to k}$$
that maximize the feature discrepancy and finish to construct the network with weights
$$\bar{w}^{j\to k}= \frac{1}{|\J_L^k|} \mathbf{1}_{j \in \J_L^k}.$$
Our main result consists in proving that this network is CHANI's limit under certain assumptions, several of which are detailed after the theorem because of their complexities.

\begin{theorem} \label{theo principal}
Suppose we are under the CHANI EWA framework described in Section \ref{sec framework spec sett} and Assumptions \ref{assump EWA out} and \ref{assump hom} described in Sections \ref{sec out lim} and \ref{sec classes feat corr} hold. In the regime where $N\to \infty$, $ T / (NM^{L-1})\to \infty$ and $N^{-1/2} e^{cM^{1/2}} \to \infty$ where $C$ is a constant independent from $M,N$ and $T$ then, almost surely,

$\bullet$ each hidden layer converges:  $\hat{J}_l \to \J_l$,

$\bullet$ the family of weights converges: $w \to \bar{w}$.

\noindent Moreover, we have the following equivalence:

$\bullet$ the limit output family weight is a strong feasible family w.r.t. $\J_L$

$\bullet$ there exists a strong feasible family w.r.t. $\J_L$ 

$\bullet$ each  class $k \in K$ can be decomposed as unions of features correlations of size $2^L$ (Assumption \ref{assump decomp}).

\noindent If this is the case, the limit network is ideal.
 \end{theorem}

\begin{tcolorbox}[breakable, width=\textwidth]
\begin{defi}[About the limit network]
 The main Theorem~\ref{theo principal} shows the most refined result of the article: the {\bf limit network is ideal} when classes are defined by combinations of feature correlations of size $2^L$: more precisely, an object $o$ belongs to class $k$ if it possesses all the features of a given subset of size $2^L$, or all the features of another such subset, or yet another, and so on. A neuron $j_S$ in a hidden layer encodes exactly the correlations among the features in its subset $S$; that is, when presented with an object $o$, its spiking probability is proportional to $\rho_o(S)$. Similarly, a neuron $k$ in the output layer encodes class $k$, meaning its spiking probability is strictly positive if and only if $o \in k$. Moreover, the output layer is entirely noise-free: an output neuron remains completely silent when the presented object belongs to a different class. While certain technical assumptions were required to establish this result, our numerical experiments in Section \ref{sec num res} indicate that CHANI performs well in practice even without all of these assumptions.
 \end{defi}
\end{tcolorbox}
 
\begin{tcolorbox}[breakable, width=\textwidth]
\begin{defi}[Spiking Neuron Convergence Conjecture] \label{intuit SNCC}
 Besides, thanks to the equivalence between the existence of a strong feasible weight family and the decomposition of each class as unions of feature correlations (Assumption \ref{assump decomp} (class decomposition)), we know that \textbf{as soon as there exists a strong feasible weight family, the weights of the output 
layer of our network converge to one of these families with high probability.} We can see this statement as 
a version of the Spiking Neuron Convergence Conjecture (SNCC) \citep{legenstein2005can} which says that spiking neural networks can learn to 
 implement any achievable transformation: this conjecture is true for CHANI in this simplified framework. This result is unprecedented since, in our study \cite{jaffard:hal-04065229}, we had to assume that the output limit weights were a feasible weight family in order to draw conclusions 
 (see Corollary $3.5$), without providing any criteria ensuring it. 
 \end{defi}
 \end{tcolorbox}
 

In the following sections, we provide detailed assumptions and theoretical results which ultimately lead to Theorem \ref{theo principal}. In this purpose, we need to work within a precise framework that we call CHANI EWA, defined by the assumptions of the following section.

\subsection{CHANI EWA Framework} \label{sec framework spec sett}

\begin{assumption}[EWA for hidden layers] \label{assump EWA hid}
   For any $l\in [L]$, the expert aggregation $f^l$ used to train the hidden layer of depth $l$ is EWA with learning rate $\eta^l = \sqrt{\frac{8\ln(\abs{\hat{J}_{l-1}})}{M}}$.
\end{assumption}
Note that this choice of $\eta^l$ maximizes the regret bound of EWA given in \cite{cesa2006prediction}. This assumption imposes no restriction on the expert aggregation algorithm $f^{L+1}$ used to train the output layer.
\begin{assumption}[Balanced] \label{assump nb obj}
    During each training session of a layer, each nature of object $o\in \Obj$ is presented the same amount of times to the network.
\end{assumption}
This is a technical assumption that we make to facilitate computations.
\begin{assumption}[$1/2$ bias] \label{assump nu}
    All hidden neurons have bias $\nu = \frac{1}{2}$.
\end{assumption}
This assumption is here to ensure that hidden neurons activity converge to ideal hidden layers. See discussion about the choice of $\nu$ in Section \ref{sec discuss nu}.
\begin{assumption}[Decreasing correlation] \label{assump rho}
    For all $I'\subset I$ such that $\rho(I')>0$, for all $i\in I\setminus I', \rho(I'\cup \{i\}) < \rho(I')$.
\end{assumption}
This assumption means that adding a neuron to a set of correlated neurons results in a loss in correlation for at least one object. It is obviously verified when input neurons spike independently of one
 another. To ensure that input neurons activity verify this assumption, one can add a small independent perturbation to 
 their activity.

\begin{assumption}[Sparse features] \label{assump half sets}
    There exist thresholds $s_1,\dots,s_L \in (0,1]$ such that the sets $(\J_l)_{0\leq l \leq L}$ defined in section \ref{sec main th} verify
    \[
\forall l\in \{0,\dots,L\} \text{ and } o\in \Obj, \quad   \abs{\{j_S\in \J_l, \rho_o(S)>0\}} \leq \frac{\abs{\J_l}}{2}.
\]
and
\begin{align*}
      &\forall l\in \{1,\dots, L\}, S\in I_l \text{ not encoded in } \J_l \text{ and } o\in \Obj, \quad \rho_o(S) < 2^{2^l-1} s_l.
\end{align*}
These thresholds are used for neuron selection \eqref{thresh}.
\end{assumption}
This assumption states that there exist thresholds such that for any object $o\in \Obj$, at most half the neurons of $I$ and $\J_l$ for $l\in[L]$ are active and spike together. Moreover, these thresholds strictly separate the neurons of the set $\J_l$ from other neurons during selection phases. This is a technical assumption that we need to study CHANI asymptotic behavior.

\subsection{Hidden layers limit behavior: formation of neuronal assemblies} \label{sec hidden layers}

In order to study CHANI's performance, which is reflected by the activity of its ouput layer, we start by investigating the behavior of its hidden layers:
in the framework CHANI EWA, we can explicitly compute the limit of hidden neurons weights.

\begin{theorem} \label{theorem lim hid}
    Suppose we are in the CHANI EWA framework. Let $\alpha>0$. There exist constants $C_1$ and $C_2$ independent of $M$ and $T$ and constants $c_1,c_2 >0$ independent of $M,T$ and the network parameters such that if $M\geq C_1$ and $T\geq C_2 M^{L-1}$,
     then with probability $1-\alpha$:
    
  \noindent $(i)$ For all $l\in [L]$, $\hat{J}_l = \J_l$.
        
    \noindent$(ii)$ Let $l\in [L]$. Then, for all $j = j_{S_1\cup S_2} \in \J_l$ where $S_1$ and $S_2$ are encoded in layer $\J_{l-1}$, at the end of learning the synaptic weights are such that
    \begin{align*}
    \Norm{w^{j}_{M+1} - \frac{1}{2}\mathbb{1}_{\{j_{S_1},j_{S_2}\}}}_2 \leq c_1 
        \Big(\mathcal{J}_{l-1} &\ln(\mathcal{J}_{l-1})^{\frac{1}{2}} \ln(L \mathcal{J}_{l-1}\alpha^{-1})^{\frac{1}{2}}\Big)^{l-1}  \Big(\Big(\frac{M^{l-1}}{T}\Big)^{1/2}
        \\
       & + e^{-c_2 \min_{l'\leq l-1} \big\{2^{-2^{l'}}\rho_{l'} \ln(\abs{\J_{l'}})^{1/2}\big\} M^{1/2}}
        \Big)
    \end{align*}
        where the {\bf limit weights} are $\mathbb{1}_{\{j_{S_1},j_{S_2}\}} := ( \mathbb{1}_{j'\in \{j_{S_1},j_{S_2}\}})_{j'\in \J_{l-1}} \in \{0,1\}^{\abs{\J_{l-1}}}$ and where $\mathcal{J}_{l-1}:= \max_{l'\leq l-1} \abs{\J_{l'}}$, whereas $\rho_{l'} := \min_{j_S\in \J_{l'}} \rho(S)$ is the minimum average correlation of a feature subset of a hidden neuron of depth $l'$.
\end{theorem}
This theorem states that  the sets of interest $\J_l$ are selected with high probability. Besides, the weights of a hidden neuron $j=j_{S_1\cup S_2}$ converge to uniform distribution on neurons $j_{S_1}$ and $j_{S_2}$ with an error term which is small when $M\gg 1$ and $T\gg M^{L-1}$. The dependency on the size of the selected sets $\abs{J_l}$ also shows that the fewer neurons selected, the better. When taking only into account the dependency on $M$ and $T$, the error term of the weights of layer $l$ is in  $O((\frac{M^{l-1}}{T})^{1/2} + e^{-C\sqrt{M}})$ where $C$ is a constant independent of $M$ and $T$. The exact error term is given in the proof of Theorem \ref{theorem lim hid} (see Appendix \ref{sec proof th hid}). It can be decomposed in two terms: the first one comes from the randomness induced by the spikes, as well as the accumulation of the errors coming from previous layers. The second one comes from the specific use of the expert aggregation algorithm EWA. This second term decreases as the average correlation of the feature subset of hidden neurons grows: the more hidden neurons describe high correlations between input neurons the better.


\begin{tcolorbox}[breakable, width=\textwidth]
\begin{defi}[About the convergence of hidden layers.]
With our assumptions, it means that the hidden neuron $j_S$  meant to detect correlations of features in $S=S_1\cup S_2$ finishes its learning by being linked only to $j_{S_1}$ and $j_{S_2}$. It could have been wired like that in hard from the beginning and just the pruning part could have been kept. But this would have missed the biological interpretation point of view, which is that there is progressive learning of the correlations. An even more realistic point of view would have been, as quickly mentionned in Box \ref{idneur}, that a neuron $j$ of an hidden layer picks at random two presynaptic neurons and focus on detecting the correlation of these two. Then we would need even more neurons as a first step in an hidden layer before pruning to detect all necessary connections. If the theoretical results are likely to be of the same flavor, this would have overcomplexified our procedure and we decided to not go down this path. 

Note also that the gains of the presynaptic neurons used here \eqref{gain hidd}  are there to eventually reinforce link of $j_S$ with other neurons $j'$ that migh detect the same correlation as what $j_S$ wants to do. In this sense, it might be looked as a classical "fire together, wire together" rule, ie the Hebbian rule which would focus only on correlation, and this is also this mecanism that makes the link with Transformers. However, with Assumption \ref{assump rho}, we are in fact assuming that adding $j' \neq j_{S_1}$ and $\neq j_{S_2}$ will degrade the correlation strength, hence the fact that neuron $j'$ is not reinforced through the local learning rule. It is possible that in a more complex network, with for instance loops,  one can use additional connections to some $j'$ to strengthen the answer of the post-synaptic neurons.
\end{defi}
\end{tcolorbox}

\begin{corollary}
    \label{coro ideal hid}
  Suppose we are in the CHANI EWA framework. When $M$ tends to infinity with $T /M^{L-1}$ tending to infinity as well, then almost surely 
  
\noindent $\bullet$ For all $l\in [L]$,  $\hat{J}_l\to \J_l$.

\noindent $\bullet$ For all $j=j_{S_1\cup S_2} \in \J_l$, $w_{M+1}^j\to \frac{1}{2}\mathbb{1}_{\{j_{S_1},j_{S_2}\}}$.

\noindent Moreover, each of these limit hidden layers are ideal with constant $\gamma_l=2^{-2^l +1}$, for $l\in [L]$.  
\end{corollary} 

\begin{tcolorbox}[breakable, width=\textwidth]
\begin{defi}[About neuronal assemblies.]
Combined with Theorem \ref{theorem lim hid}, this corollary ensures that 
 at the limit the hidden layers of our network code exactly for feature correlations. 
 A hidden neuron is activated when presented with objects having a certain combination of features (\ie, having simultaneously all the features of a certain set). 
 If this combination can be found in several classes, then the same neuron will be activated, and can therefore help to represent more than one class: in this sense, CHANI produces neuronal assemblies.
 These results, 
 valid for any number of hidden layers, are unprecedented since in our previous work \cite{jaffard:hal-04065229} 
 HAN did not have hidden layers.

\end{defi}
\end{tcolorbox}

\subsection{Average learning}

Now that we studied the hidden layers behavior, we can analyze the network's performance in the classification task. In order to so, we compare the network discrepancy \eqref{netdisc}, which measures the average performance of the network during learning, to 
the largest ideal discrepancy of a feasible weight family \eqref{discID}, which represents an ideal performance. Let 
\begin{equation}\label{Ehid}
E_{\text{hid}}(M, T, \alpha) :=L \sqrt{\mathcal{J}_L} \max_{l \in [L], j\in\J_l}  \Norm{w^j_{M+1} - \Bar{w}^j}_2
\end{equation}
the maximum error term of a neuron of hidden layer, bounded in Theorem \ref{theorem lim hid}.

Note that if we have assumed EWA on the hidden layer (Assumption \ref{assump EWA hid}), one can use any expert agregation algorithm that they want on the last layer as long as the following assumption holds.

\begin{assumption}[Output neurons regret bound] \label{assumption regret bound out}
The expert aggregation algorithm $f^{L+1}$ used to train the output layer is such that for any set of experts $E$ and any deterministic sequence of gains $g^E_{[N]}:=(g^e_m)_{\begin{subarray}{l}
    e\in E \\ m \in [N]
\end{subarray}}$ taking value in $ [a,b]$, 
\[R_M \leq C(b-a) \sqrt{\ln(\abs{E})M}\]
where $C$ is a constant. 
\end{assumption}
Note that both EWA and PWA verify this bound for certain parameters (see \cite{cesa2006prediction} for details). 

By combining Theorem \ref{theorem lim hid} and Assumption \ref{assumption regret bound out} (Output neurons regret bound), we get the following result. 

\begin{corollary} \label{coro average}
    Suppose  we are in the CHANI EWA framework and Assumption \ref{assumption regret bound out} (output neurons regret bound) holds. Let $\alpha >0$. Suppose $M\geq C_1$ and $T\geq C_2 M^{L-1}$ where $C_1$ and $C_2$ are the constants defined in 
    Theorem \ref{theorem lim hid} and $\mathcal{Q}_{\J_L}$, the set of feasible weight families w.r.t. the ideal set of neurons $\J_L$, is non-empty. Then with
     probability $1-\alpha$ the conclusions of Theorem \ref{theorem lim hid} hold and there exists a constant $c>0$ independent of the network parameters and 
     $\gamma_L = 2^{-2^L + 1}$ such that
    \begin{align*}
        \Disc_N(w^K_{[N]}, X_{[N],[T]}^{\hat{J}_L,L+1} ) \geq &\gamma_L \max_{q^K\in \mathcal{Q}_{\J_L}} \Disc^{\id}(q^K, \J_L) \\
        &-c \left(\sqrt{\frac{\abs{\Obj}}{NT}\ln\Big(\frac{L\abs{\J_L}}{\alpha}\Big)} +\abs{\Obj} \sqrt{\frac{\ln(\abs{\J_L})}{N}} + E_{\text{hid}}(M,T, \alpha)  \right).
    \end{align*}
\end{corollary}

This corollary establishes that CHANI's network discrepancy exceeds the ideal discrepancy of the best feasible weight family possible multiplied by a constant $\gamma_L$, minus an error term which is small
 when $M\gg 1$, $T \gg M^{L-1}$ and $N\gg 1$, and increases with the number of selected neurons. This error term is threefold: the first part comes from the randomness of the spikes, the second one comes from the regret bound of the expert aggregation of the output layer $f^{L+1}$, and the third part comes from the approximation error between hidden layers weights and their limit. When taking only into account the dependency in $T$, $M$ and $N$, the overall error term is in
\[O\Big(N^{-1/2} + \Big(\frac{M^{L-1}}{T}\Big)^{1/2} + e^{-C\sqrt{M}}\Big)\]
where $C>0$ is a constant independent of $M, N$ and $T$. Note that this corollary holds for any expert aggregation used to train the output layer, as long as it verifies Assumption \ref{assumption regret bound out} (output neurons regret bound).

In other words, on average, CHANI performs as well as it would if the output layer were connected to the ideal hidden layer $\J_L$ with constant $\gamma_L$ with the best possible feasible weight family. 
This corollary compares CHANI's network discrepancy to a deterministic quantity which does not depend on CHANI itself, and which represents 
an ideal performance where the hidden layers detect neuronal synchronization and the output layer succeeds to rewrite the classes in terms of combination of correlations between $2^L$ features.  In Section \ref{sec reg output} Proposition \ref{prop reg class disc}, we give a weaker version of this result under more general assumptions, in which we compare CHANI's network discrepancy to the best network discrepancy using CHANI's own hidden layers.

Note that the error term increases with the depth $L$, and the constant $\gamma_L$ decreases with $L$, which suggest that deeper networks give worse bounds. However, a major assumption is the existence of a feasible weight family. Since adding layers increases the complexity of the classes which can be represented by the network, this assumption is more likely to be verified for large $L$. 

This corollary can be related to Theorem $3.3$ of \cite{jaffard:hal-04065229}, where HAN network discrepancy is compared to the best safety discrepancy of a feasible weight family, a quantity describing how well 
a feasible weight family classifies objects. However, this new result is much more general, for several reasons. Firstly, in the current work, CHANI can comprise any number of hidden layers, whereas in our previous work
 HAN is restricted to having solely an input and an output layer. Secondly, both of these results assume the existence of a feasible weight family. However, in \citep{jaffard:hal-04065229}, this assumption is considerably more restrictive, as a feasible weight family can only exist if the classes can be described as simple combinations of features. In our current study, a feasible weight family is achievable when classes can be characterized as combinations of feature correlations, the complexity of which grows with the depth $L$ (see conditions for the existence of strong feasible weights families in section \ref{sec classes feat corr}). In particular, if a feasible weight family does not already exist, the addition of layers may facilitate its emergence.


\begin{tcolorbox}[breakable, width=\textwidth]
\begin{defi}[Oracle inequality and global optimization]
This result can be seen as an oracle inequality in the sense that each individual neuron does not know what the other neurons in the network do since all rules are local. Still the performance of the global network is as good as (up to a multiple constant) the best choice of weight $q^K$ on the output layer, coupled with ideal hidden layers that perfectly detect  correlation. In this sense, it could also be compared to a global optimisation of the network discrepancy on the weights on the output layer. However it cannot be compared to a global optimization of the network discrepancy with respect to the weights on all the layers, because nothing says that another choice of weights for the hidden layers would not lead to a better discrepancy, in particular if the hidden layer do not encode correlations of features anymore.

\end{defi}
\end{tcolorbox}

\subsection{Whole network limit behavior}
\subsubsection{Output layer limit behavior} \label{sec out lim}
In this section, we work in the framework CHANI EWA, with the additional assumption that the expert aggregation used for the output layer is EWA as well (which is a stronger assumption than Assumption \ref{assumption regret bound out} (Output neurons regret bound)). This allows us to end the asymptotic study started in section \ref{sec hidden layers} by computing the limit output weights.

\begin{assumption}[EWA for the output layer] \label{assump EWA out}
    The expert aggregation $f^{L+1}$ used to train the output layer is EWA with learning rate $\eta^{L+1} = \frac{1}{\abs{\Obj}}\sqrt{\frac{2\ln(\abs{\J_L})}{N}}$.
\end{assumption}
Note again that this choice of $\eta^{L+1}$ minimizes the regret bound of EWA given in \cite{cesa2006prediction}. For $\alpha \in (0,1]$ let 
\[
E_{\text{out}}(T,\alpha):=  \sqrt{\frac{\ln(\abs{\J_L})}{\abs{\Obj}T} \ln\Big(\frac{L\abs{\J_L}\abs{K}}{\alpha}\Big) }
\]

We recall that the feature discrepancy $\Disc^{j\to k}$ of a hidden neuron $j=j_S$ of the last hidden layer w.r.t. an output neuron $k$ quantifies how much the feature subset $S$ is relevant to describe class $k$ \eqref{featdisc}.

\begin{theorem} \label{theorem lim out}
     Suppose we are in the CHANI EWA framework and Assumption \ref{assump EWA out} holds. Let $\alpha \in (0,1]$. 
     Suppose $M\geq C_1$ and $T\geq M^{L-1}C_2$ where $C_1$ and $C_2$ are the constants defined in 
     Theorem \ref{theorem lim hid}. We recall that 
     $\J_L^k = \arg\max_{j\in \J_L} \Disc^{j\to k}$ and $\mathbb{1}_{\J_L^k} = (\mathbb{1}_{j\in \J_L^k})_{j\in \J_L}$.  
     Then there exist constants $c_3, c_4>0$ independent of the network parameters, such that with probability $1-\alpha$, the conclusions of Theorem \ref{theorem lim hid} hold and for all $k\in K$
     
$\bullet$ if $\J_L^k = \J_L$ then 
\begin{align*}
    &\Norm{w^{k}_{N+1} -\frac{1}{\abs{\J_L^k}}\mathbb{1}_{\J_L^k} }_2 \leq c_3 \Big( \sqrt{\abs{\J_L}\ln(\abs{\J_L})N} E_{\text{hid}}(M, T, \alpha) + E_{\text{out}}(T,\alpha)  \Big).
\end{align*}

$\bullet$ if $\J_L^k \varsubsetneq \J_L$  let $\Delta^k := \max_{j\in \J_L^k} \Disc^{j\to k} 
        - \max_{j\in \J_L \setminus \J_L^k} \Disc^{j\to k}$. Then
        \begin{align*}
            \Norm{w^{k}_{N+1} -\frac{1}{\abs{\J_L^k}}\mathbb{1}_{\J_L^k} }_2  \leq c_3 \Big( \sqrt{\abs{\J_L}\ln(\abs{\J_L})N} & E_{\text{hid}}(M, T, \alpha) \\
            & + E_{\text{out}}(T,\alpha)  +e^{-c_4 \frac{\Delta^k }{2^{2^L}}\sqrt{\ln(\abs{\J_L})N}}
            \Big).
        \end{align*}
\end{theorem}
Note that the bound given by Theorem \ref{theorem lim hid} on the error term $E_{\text{hid}}(M, T, \alpha)$ \eqref{Ehid} is independent of the size of the last hidden layer $\abs{\J_L}$ and the number of presented objects for the training phase of the output layer $N$. The overall error is threefold: the first two parts come from the approximation error between hidden and output layer weights and their limit, and the third part comes from the use of the expert aggregation EWA for the output layer. When taking only into account the dependency on $M$, $N$ and $T$, the error is in \[O\Big(\Big(\frac{NM^{L-1}}{T}\Big)^{1/2} + N^{1/2} e^{-CM^{1/2}} + e^{-DN^{1/2}}\Big)\] where $C$ and $D$ are constants independent of $M$, $N$ and $T$: it is small when $N\gg 1$, $M\gg \ln(N)^2$, $T \gg N M^{L-1}$ and increases with the number of selected neurons.

\begin{corollary}
    \label{coro lim output}
   Suppose we are in the CHANI EWA framework and Assumption \ref{assump EWA out} holds. When $T$, $M$ and $N$ tend to $\infty$ with $T /(NM^{L-1})$ and $e^{CM^{1/2}}N^{-1/2}$, where $C$ is a constant independent of $M,N$ and $T$, tending to infinity as well, then almost surely, the conclusions of Corollary \ref{coro ideal hid} hold and for all $k\in K$, 
   \[w_{N+1}^k \to \frac{1}{\abs{\J_L^k}}\mathbb{1}_{\J_L^k}.\]
\end{corollary} 

This corollary states that the weights of an output neuron $k$ converge to the family which uniformly 
distributes weights on presynaptic neurons with maximal feature discrepancy. Therefore, at the limit, neuron $k$ will be connected only to neurons coding for correlations which are sensible to class $k$. 

Theorem \ref{theorem lim out} and Corollary \ref{coro lim output} can be related to Theorem $3.4$ of \cite{jaffard:hal-04065229}, where we computed HAN limit weights with no hidden layers. Here, we generalize this asymptotic analysis to CHANI. Besides, thanks to
the addition of hidden layers, CHANI's limit output weight family is more interesting than HAN's from a learning point of view since it decomposes classes in terms of correlations of $2^L$ features (against combinations
  of simple features). In a general case, we cannot
say if this limit family is a strong feasible weight family, or even if it enables the network to correctly classify objects when noise is allowed.
However, we can conclude in a more precise framework defined below where classes are defined by feature correlations.

\subsubsection{When classes are defined by feature correlations} \label{sec classes feat corr}
First, let us define some assumptions that we will need to perform our study.

\begin{assumption}[Binary correlations] \label{assump hom}
    There exists $p\in (0,1]$ such that for all $j_S\in \J_L$, for all $o\in \Obj$, $\rho_o(S) \in \{0,p\}$.
    If $\rho_o(S) = p$ then we say that object $o$ has features $S$. Besides, for $j_S\in \J_L$, let $\Obj^S:=\{o\in \Obj, \rho_o(S) = p\}$
    the set of objects with features $S$. Then all the $\Obj^S$ have the same cardinality.
\end{assumption}
 The first part of this assumption is verified if for instance the input neurons spike independently of each other and have a fixed spiking probability when they are activated. The restriction on the cardinality of the sets $\Obj^S$ is a technical assumption that we make in order to facilitate the computations.

\begin{assumption}[Class decomposition] \label{assump decomp}
 For every $k\in K$, there exists a set $E^k\subset I_L$ such that for every $S\in E^k$, $j_S\in \J_L$ and $k = \bigcup_{S\in E^k} \Obj^S$.
 \end{assumption}
This assumption means that the classes in which the objects are classified are defined by feature correlations (encoded in the set $\J_L$). To be more precise, an object $o$ belongs to class $k$ if it possesses all the features of a given subset of size $2^L$, or all the features of another such subset, or yet another, and so on. Thus, the complexity of these correlations increases with $L$: as the network's depth grows, a set $\Obj^S$ captures more intricate patterns within the objects.

 \begin{theorem} \label{theo specific case}
    Suppose Assumption \ref{assump hom} (binary correlations) holds. Then the following statements are equivalent.

    \noindent $\bullet$ Assumption \ref{assump decomp} (class decomposition) is verified.
    
    \noindent $\bullet$ There exists a strong feasible weight family w.r.t. the set $\J_L$.
    
    \noindent $\bullet$ The limit output weights family defined in Corollaries \ref{coro ideal hid} and \ref{coro lim output} is a strong feasible weight family.

    \noindent In particular, when Assumption \ref{assump decomp} (class decomposition) is verified, the network with limit weights is ideal. 
 \end{theorem}
 
 This final theorem provides the last ingredients needed to establish the main result of our article, which is Theorem \ref{theo principal}. It establishes that if the classes can be characterized by correlations among $2^L$ features encoded in the input layer, then CHANI is capable of successfully learning to classify them. Furthermore, as long as there exists an output weight family that enables correct classification when the hidden layers encode correlations, CHANI’s output layer will converge to one of these optimal weight families. This learning guarantee is analogous to the Perceptron Learning Theorem \citep{rosenblatt1962principles}, which states that if a set of weights exists that can classify the data—equivalent to assuming linear separability—then the perceptron learning algorithm will find such weights. Additionally, this result can be linked to the Spiking Neuron Convergence Conjecture that we described in Intuition $\&$ Sketch \ref{intuit SNCC}: we proved this conjecture for CHANI within a specific framework.

In this framework, we can also compute the Vapnik-Chervonenkis dimension of our network. For a given classification algorithm, its VC-dimension corresponds to the size of the largest set of points that it can shatter. It measures the complexity of the class of functions that can be learned. In the literature, this dimension appears in generalization error bounds.

We assimilate the objects to be classified as sequences $o=(o_i)_{i\in I}$ where $o_i=1$ if $o$ has feature $i$ and $o_i = 0$ otherwise: they live in a space of dimension $\abs{I}$. Neuron $j_S\in \J_L$ is activated when $o$ has all the features of the set $S$.
We consider a binary output: let the hypothesis set $H:=\{\mathbb{1}_{F} \text{ such that } F\subset J_L\}$, where for an object $o$, $\mathbb{1}_{F}(o) = 1$ if and only if there exists $j\in F$ activated by $o$, and $0$ otherwise. It corresponds to the functions that CHANI can learn. Then we have the following result.

\begin{proposition} \label{prop vcdim}
The VC-dimension of the hypothesis set $H$ is the size of the last hidden layer
\[\operatorname{VCdim}(H) = \abs{\J_L}.\]
\end{proposition}

Unlike multilayer perceptrons, whose VC-dimension
scales as $O(W\log W)$ (where W denotes the total number of weights) \citep{Shalev-Shwartz_Ben-David_2014}, CHANI’s VC-
dimension depends solely on the complexity of the last hidden layer. This arises because
each hidden neuron autonomously learns correlations in an almost unsupervised learning
manner, as its learning rule does not account for the true class of the presented object.
Besides, the complexity of one hidden layer is not added to the next, but rather included
in it. However, as the number of selected neurons increases, the errors in the results regarding the asymptotic behavior of CHANI also increase. Therefore, it is essential to strike a balance between accurately representing complex classes and learning efficiently.

\subsection{Discussion about the choice of bias.} \label{sec discuss nu}

Assumption \ref{assump nu} ($1/2$ bias) could be relaxed to any $\nu\in [1/2, 1)$: Theorem \ref{theorem lim hid} and Corollary \ref{coro ideal hid} would still hold. Indeed, their key ingredient is that hidden neurons spiking probabilities converge to ideal spiking probabilities. The conditional spiking probability of hidden neuron $j= j_{S_1\cup S_2}$ with bias $\nu$ and fixed weights $\mathbb{1}_{\{j_{S_1},j_{S_2}\}}$ where $S_1$ and $S_2$ are feature subsets encoded in the previous layer (which are the limit weights given by Corollary \ref{coro ideal hid}) at time step $t$ of the presentation of an object $o$ is
$\Big(-\nu + \frac{1}{2}(X^{j_{S_1}}_{o,t-1} + X^{j_{S_2}}_{o,t-1})\Big)_+$, where $X^{j_{S_1}}_{o,t-1}$ and $X^{j_{S_2}}_{o,t-1}$ are the activities of neurons $j_{S_1}$ and $j_{S_2}$ at the previous time step.
For $\nu \in [1/2, 1)$, neuron $j$ is active if and only if neurons $j_{S_1}$ and $j_{S_2}$ spike, and in this case its conditional spiking probability is $1-\nu$. Therefore, the limit spiking probabilities would still be ideal. However, there is a decrease in firing rate from one hidden layer to the next and the choice $1/2$ guarantees the lowest decrease. This decrease appears in the constant $\gamma_L$ of Corollary \ref{coro average}: in fact, $\gamma_L = (1-\nu)^{2^L-1}$ and a large $\gamma_L$ guarantees a large class discrepancy for our network. Therefore, the choice $\nu = 1/2$ is optimal.
\newline

For $\nu < 1/2$, neuron $j$ is active even when only one of the two neurons $j_{S_1}$ and $j_{S_2}$ spikes, 
so this choice of bias does not provide ideal limit spiking probabilities and the extra activity of hidden neurons can be identified as noise. 

Let us study a very simple example in the framework of section \ref{sec classes feat corr} where for any $\nu <1/2$, the limit weights do not enable to correctly classify the objects.

We consider the set of object natures $\Obj =\{{\color{blue}\bigcirc},{\color{blue}\Box}, {\color{blue} \triangle}, {\color{red} \bigcirc}, {\color{red} \Box}, {\color{red} \triangle}, {\color{green} \bigcirc}, {\color{green} \Box}, {\color{green} \triangle}\}$.
The set of features and input neurons is $I = \{\text{circle, square, triangle, blue,  red,  green}\}$.
We assume that input neurons $i$ spike independently of each other with probability $p'$ if and only if the presented object has feature $i$, and we choose depth $L=1$. Then any threshold $s_1\in (0,{p'}^2)$ verifies Assumption \ref{assump half sets} and in this case
\begin{align*}
    \J_1 = \{ j_{\{\text{blue}, \text{square}\}}, &j_{\{\text{red}, \text{square}\}}, j_{\{\text{green}, \text{square}\}}, j_{\{\text{blue}, \text{circle}\}}, j_{\{\text{red}, \text{circle}\}}, \\
    & j_{\{\text{green}, \text{circle}\}}, j_{\{\text{blue}, \text{triangle}\}}, j_{\{\text{red}, \text{triangle}\}}, j_{\{\text{green}, \text{triangle}\}} \}.
\end{align*}
This set can be identified with the set of object natures $\Obj$, and Assumption \ref{assump hom} (binary correlations) is verified with $p = p'$. Therefore, any set of classes $K$ verify Assumption \ref{assump decomp} (class decomposition) since any class can be written as the union of its single elements. Let us choose $K=\{k_1, k_2\}$ with 
$k_1=\{{\color{red}\Large \Box },{\color{green}\Large \Box },{\color{blue} \bigcirc },{\color{blue}\triangle }\}$ and $k_2=\{{\color{blue}\Large \Box },{\color{red} \bigcirc },{\color{green} \bigcirc },{\color{red}\triangle },{\color{green}\triangle }\}$.
This classes are such that class $k_1$ contains every item which is squared or blue except the blue square, and class $k_2$ contains every other possible object.

\begin{proposition} \label{prop nu}
    In this framework, for any $\nu \in [0, 1/2)$, the limit weights do not enable to correctly classify the ${\color{blue}\Large \Box}$.
\end{proposition}

Therefore, even with $\nu$ less but close to $1/2$, the limit network fails to correctly classify the objects, whereas for any $\nu \geq 1/2$ the results of section \ref{sec classes feat corr} hold and the limit network succeeds to accomplish the task. This proposition enlightens that even a small noise can stop the network from working properly in a very simple case. The choice of bias allows us to control this phenomenon.

\section{Additional results in a more general framework} \label{sec reg}

In this section, we examine the guarantees we can obtain about CHANI's behavior in a much more general setting than the CHANI EWA framework: the main assumption needed in this section is that the expert aggregation algorithm admits a regret bound. As each hidden and output neuron runs its own expert aggregation algorithm, it possesses its own regret. Through our choice of gains, we reinterpret these regrets, enabling us to gain meaningful insights into the network's learning.

\subsection{Hidden layers}

\begin{assumption}[Hidden neurons regret bound] \label{assumption regret bound hid}
There exists $C>0$ such that for every $l\in [L]$, the expert aggregation algorithm $f^l$ used to train the hidden layer of depth $l$ is such that for any set of experts $E$ and any deterministic sequence of gains $g^E_{[M]}:=(g^e_m)_{ e\in E, m \in [M]}$ taking value in $[a,b]$ with $a<b$,
\[R_M \leq C(b-a) \sqrt{\ln(\abs{E})M}.\]
\end{assumption}
The interpretation of the regret of hidden neurons leads to the following result.

\begin{proposition}[Hidden neurons regret bound] \label{prop reg hid}
Suppose Assumption \ref{assumption regret bound hid} (hidden neurons regret bound) holds. Then a.s. for every
 $l\in [L]$ and $j=j_{S_1\cup S_2}\in J_l$ neuron of the $l^{\text{th}}$ layer where $S_1$ and $S_2$ are feature subsets encoded in the previous layer,
\[
\max_{q \in \mathcal{P}_{\hat{J}_{l-1}}} 
\Brac{(\psi^{j,l}_{m,t}(q) - \psi^{j,l}_{m,t}(w^j_m))X^{j_{S_1},l}_{m,t}X^{j_{S_2},l}_{m,t}}_{\begin{subarray}{l}
   t \in [T] \\ m \in [M]
\end{subarray}} \leq  C \sqrt{\frac{\ln(\abs{\hat{J}_{l-1}})}{M}}.
\]
\end{proposition}

This regret bound can be interpreted as follows: since $ p^{j,l}_{m,t}(w^j_m) = 
 (\psi^{j,l}_{m,t}(w^j_m))_+$, on average, neuron $j=j_{S_1\cup S_2}$ will tend to spike as frequently as possible
when neurons $j_{S_1}$ and $j_{S_2}$ are highly correlated, that is, when the product $X^{j_{S_1},l}_{m,t}X^{j_{S_2},l}_{m,t}$ is
often equal to $1$.
Although this result has less implications than Theorem \ref{theorem lim hid}, it relies only on Assumption \ref{assumption regret bound hid} (Hidden neurons regret bound). This makes it a significant step toward analyzing the role of hidden layers in CHANI within a very general setting.


\subsection{Output layer} \label{sec reg output}

\begin{assumption}[$\xi$-balanced] \label{assump proportions}
    There exists a constant $\xi>0$ independent of $N$ such that for every $k\in K$, $N_k/N\geq \xi$.
\end{assumption}
This assumption means that during the training phase of the output layer, the proportion of presented objects of any class is at least $\xi$. This ensures that the network sees a significant number of objects of each class. 

\begin{proposition} \label{prop reg class disc}
    Suppose Assumptions \ref{assump proportions} ($\xi$-balanced) and \ref{assumption regret bound out} (Ouptut neurons regret bound) hold. Then,
    \[
\Disc_N(w^k_{[N]}, X_{[N],[T]}^{\hat{J}_L,L+1}) \geq \max_{q^K\in (\mathcal{P}_{\hat{J}_{L}})^{\abs{K}}}  \Disc_N(q^K, X_{[N],[T]}^{\hat{J}_L,L+1}) - \frac{C\abs{K}}{\xi(\abs{K}-1)} \sqrt{\frac{\ln(\abs{\hat{J}_L})}{N}} \quad \text{a.s.}
\]
\end{proposition}
The proposition, established by combining local regret bounds of output neurons, ensures that the network discrepancy of CHANI exceeds the maximum achievable with constant weights for the output layer,
 minus an error term which tends to zero.
Note that in this case the quantity $\Disc_N(q^K, X_{[N],[T]}^{\hat{J}_L,L+1})$, against which we compare CHANI’s network discrepancy, depends on the activity of CHANI’s final hidden layer. This is not the case in Corollary \ref{coro average}, which applies within the more restrictive CHANI EWA framework and relies on the present result as an initial step.

\section{Numerical results} \label{sec num res}

In this section, we present numerical results on two datasets: a simulated dataset consisting of colored shapes, for which the assumptions required by the theoretical results are satisfied, and the \texttt{digits} dataset from \texttt{scikit-learn}, for which not all assumptions hold, yet CHANI is nonetheless able to achieve strong performance.

\subsection{Simulated dataset} \label{sec synth data}
 We work in the framework of the example discussed in Section \ref{sec discuss nu}.
We consider the set of object natures $\Obj =\{{\color{blue}\bigcirc},{\color{blue}\Box}, {\color{blue} \triangle}, {\color{red} \bigcirc}, {\color{red} \Box}, {\color{red} \triangle}, {\color{green} \bigcirc}, {\color{green} \Box}, {\color{green} \triangle}\}$ and the set of features $I=\{\text{circle},\text{square},\text{triangle},\text{blue},\text{red},\text{green}\}$. The features are encoded in the input layer as follows: when presented with an object, the two input neurons coding for its shape and color spike independently with a fixed probability $p>0$ and the other input neurons stay silent. We trained CHANI on two tasks.
\newline

\noindent $\bullet$ A first task where the classes are $k_1 = \{{\color{blue}\bigcirc}, {\color{red}\bigcirc},{\color{green}\bigcirc}\}$ and $k_2= \{{\color{blue}\Box}, {\color{blue} \triangle}, {\color{red} \Box}, {\color{red} \triangle}, {\color{green} \Box}, {\color{green} \triangle}\}$. In this case, the classes can be described as combinations of single features (looking at the shape of the objects is sufficient to classify them): Assumption \ref{assump decomp} (Class decomposition) is verified for the network without hidden layers, and we have
\[
k_1 = \Obj^{\{\text{circle}\}} \quad \text{and} \quad k_2 = \Obj^{\{\text{square}\}}\cup \Obj^{\{\text{triangle}\}}.
\]

\noindent $\bullet$ A second task where the classes are $k_1 =\{{\color{blue}\bigcirc}, {\color{blue} \triangle},  {\color{red} \Box},  {\color{green} \bigcirc}, {\color{green} \Box}\}$ and $k_2 =\{{\color{blue}\Box},  {\color{red} \bigcirc},  {\color{red} \triangle},{\color{green} \triangle}\}$.
In this case, the classes cannot be described as combinations of single features, but as combinations of pairs of features: Assumption \ref{assump decomp} (Class decomposition) is verified for the network with one hidden layer, and we have
\[
k_1 = \Obj^{\{\text{circle},\text{blue}\}}\cup \Obj^{\{\text{triangle},\text{blue}\}} \cup \Obj^{\{\text{square},\text{red}\}}\cup \Obj^{\{\text{circle},\text{green}\}} \cup \Obj^{\{\text{square},\text{green}\}}
\]
and
\[
\quad k_2 = \Obj^{\{\text{square},\text{blue}\}}\cup \Obj^{\{\text{circle},\text{red}\}} \cup \Obj^{\{\text{triangle},\text{red}\}}\cup \Obj^{\{\text{triangle},\text{green}\}} .
\]

\begin{figure}
    \centering
    \includegraphics[width=13cm]{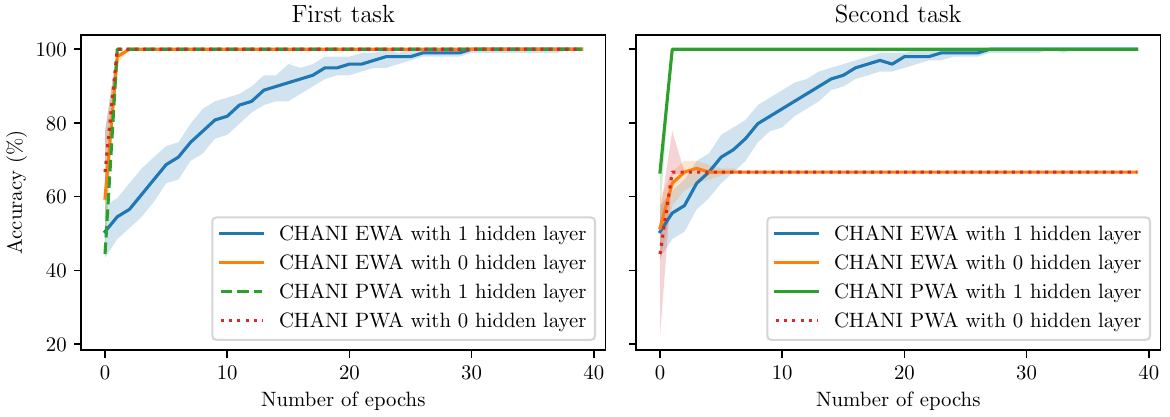}
    \caption{ Accuracy with confidence interval of level $0.9$ on simulated dataset for CHANI with EWA and CHANI with PWA, with zero and one hidden layer. 
    The parameters are $T=2000$, $M=40$, $N = 360$, $p=0.5$ and each configuration was run $100$ times. 
    For CHANI with EWA and zero hidden layer: $\eta^1 = 0.05$. For CHANI with EWA and one hidden layer: $\eta^1 = 3$ and $\eta^2 = 0.05$. For CHANI with PWA: $b=2$. For CHANI with EWA and PWA with one hidden layer: $s_1 = 0.1$. For each configuration, the training data are organized into epochs, where one epoch corresponds to a sequence of the $9$ object natures presented in random order. After each epoch of the output layer training, the network’s accuracy is evaluated on a test set consisting of $99$ objects.}
     \label{fig easy task}
\end{figure}

On Figure \ref{fig easy task} is represented the accuracy with confidence interval of level $0.9$ accross training for several configurations of the network: CHANI with either EWA or PWA as expert aggregation algorithm, and with zero or one hidden layer. 

For the first task, which is theoretically achievable by CHANI EWA with zero or one hidden layer according to Theorem \ref{theo specific case}, we can see that every method achieves $100\%$ accuracy. However, CHANI EWA or PWA with zero hidden layer and CHANI PWA with one hidden layer learn faster and with less variance than CHANI EWA with one hidden layer.

For the second task, which is theoretically achievable by CHANI EWA with one hidden layer and not by CHANI EWA with zero hidden layer according to Theorem \ref{theo specific case}, we can see that CHANI EWA or PWA with zero hidden layer  does not exceed $67\%$ accuracy, whereas CHANI with EWA or PWA with one hidden layer reaches $100\%$. However, CHANI PWA is still faster to learn and with less variance in this case.

\subsection{Handwritten digits dataset \label{sec:digit}}

We trained CHANI on the \texttt{digits} dataset of \texttt{scikit-learn}. This dataset provides $1797$ images of $64$ pixels of handwritten digits which were randomly separated 
in a training set ($80\%$ of the images) and a testing set ($20\%$ of the images). For computational reasons, we decided to use this dataset rather than one with better resolution. 
The numerical results for CHANI with no hidden layer (which corresponds to the network of our previous work HAN) are visible in Table \ref{Table no mid}. Although CHANI with EWA performs much better than CHANI with PWA, its accuracy is only around $50\%$. 

\begin{table}[h!] 
\begin{center}
    \begin{tabular}{|c|c|c|}
        \hline
         Expert Aggregation algorithm & Accuracy & Confidence interval of level $0.9$ \\
         \hline
         EWA & $53.0$ & $[50.8, 55.6]$ \\
         PWA & $14.8$ & $[14.7, 15.0]$ \\
        \hline
    \end{tabular}
\end{center}
\caption{Numerical results with no hidden layer. The parameters are $T=2000$, $N=1437$, $\eta^1 = 0.0005$, $b = 2$. We made $100$ 
    realizations.}
    \label{Table no mid}
\end{table}

The results for CHANI with one hidden layer are visible in Figure \ref{fig hid}. The blue curve (resp. the green curve) represents 
the percentage of correct classifications of the testing set of CHANI with the expert aggregation algorithm EWA (resp. PWA) in function of the number of selected hidden neurons. Before selection, 
there was $2016$ neurons. We can see that no matter the number of selected neurons, CHANI with PWA performs badly and does not exceed a $40\%$ correct classifications. 
Besides, its performance is unstable. On the other hand, the performance of CHANI with EWA increases steadily with the number of selected neurons, to station at around $83.5\%$, and exceeds $80\%$ from $80$ selected neurons. This suggests that for a more complex task than that of the previous section, the expert aggregation EWA performs better.

\begin{figure}
    \centering
    \includegraphics[width=8cm]{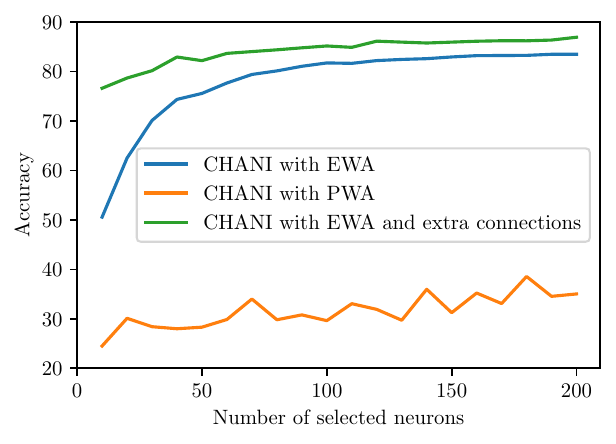}
    \caption{Numerical results on digits dataset with one hidden layer for CHANI with EWA, CHANI with PWA and CHANI with EWA and extra connections. The parameters $T=2000$, $M=40$, $N = 1357$ and $20$
    realizations were made for each number of selected neurons. For CHANI with EWA: $\eta^1 = 3$, $\eta^2 = 0.002$. For CHANI with PWA: $b=2$ for each layer. For CHANI with EWA and extra connections: $\eta^1 = 3$, $\eta^2 = 0.007$, 
    $\alpha=0.7$ and $\beta = 0.25$.}
     \label{fig hid}
\end{figure}

The numerical results with two hidden layers and EWA can be seen in Table \ref{Table 2 mid}. Here, we fixed the number of selected neurons of the first hidden layer and varied the number of selected neurons
 of the second one. We can see that in all cases, the results, comparable to CHANI with no hidden layer, are far less good than with one hidden layer. There are several possible explanations. 
 A good representation of classes could be combination of correlation between two features and not three (\ie a digit could be well represented by combinations of pixel tuples and not triplets).
Numerical results could also be affected by the phenomenon of decreasing spiking probability from one layer to the next: there may not be enough spikes in the system in order to have a second hidden layer
   with meaningful cumulated gains. Another reason could be that this dataset does not respect all the assumptions that we made in order to have theoretical guarantees.
\newline

\begin{table}[h!] 
\begin{center}
    \begin{tabular}{|c|c|}
        \hline
         Number of selected neurons on the second hidden layer & Accuracy \\
         \hline
          $50$ & $44.6$ \\
        $100$ & $49.7$ \\
        $150$ & $51.0$ \\
        \hline
    \end{tabular}
\end{center}
\caption{Numerical results with two hidden layers for CHANI with EWA. The parameters are $T=2000$, $M=40$, $N=1277$, $\eta^3 = 0.0005$. We made $10$ realizations, and for each we selected 
$100$ of hidden neurons in the first hidden layer.}
 \label{Table 2 mid}
\end{table}
\textbf{Addition of connections.} In order to improve CHANI performance, we added connections between input and output neurons in the case $L=1$. Indeed, a digit could be well represented by the combination of single pixels and pixel tuples. To do this, we extend the gains of hidden neurons w.r.t. output neurons defined by \eqref{eq gain output} to input neurons as well, and we add a multiplicative parameter $\alpha >0$ in order to find balance between hidden and input neurons gains. Then the gain of an input neuron $i$ w.r.t. output neuron $k$ when presented with object $o^{L+1}_m$ is
\begin{equation*} 
    g^{j\to k}_m := \left\{
    \begin{array}{ll}
          \alpha  \brac{X^{i,L+1}_{m,t}}_{t\in[T]} \times \frac{N}{N^k} &\text{if } o_m^{L+1}\in k  \\ 
          \\
        -  \alpha \brac{X^{i,L+1}_{m,t}}_{t\in[T]} \!\times\! \frac{N}{N^{k'}} \!\times\! \frac{1}{\abs{K}-1} &\text{if } o_m^{L+1} \in k'\neq k 
    \end{array}
    \right. .
\end{equation*}

Indeed, as the neurons spiking probabilities decrease from one layer to the next, the empirical spiking probabilities of neurons from distinct layers may have very different order of magnitude. For the same reason, we introduce a parameter $\beta \in [0,1]$ as well in the formula of the conditional spiking probability of an output neuron $k$ when presented 
 with object $o^{L+1}_m$:
\[p^k_{m,t}(w^k_m) = w^{\hat{J}_L \to k}_m \cdot X^{\hat{J}_1, L+1}_{m,t-1} + \beta w^{I \to k}_m \cdot X^{I, L+1}_{m,t-1} \]
where $w^k_m:=(w^{l \to k}_m)_{l\in I\cup \hat{J}_L}$, $ w^{\hat{J}_L \to k}_m:= (w^{j \to k}_m)_{j\in\hat{J}_L}$ and  $ w^{I \to k}_m:= (w^{i \to k}_m)_{i\in I}$. This parameter $\beta$ allows to reduce the impact
 of input neurons, which spike more frequently than hidden neurons, on the behavior of output neurons.

The numerical results with this configuration and the expert aggregation EWA correspond to the green curve of Figure \ref{fig hid}. We can see that CHANI performs better with these new connections: with only
 $10$ selected neurons, its accuracy on the testing set is already at $76.6\%$ and it grows reach $87\%$ for $200$ selected neurons, and exceeds $84\%$ from $70$ selected neurons.
\newline

\textbf{Comparison with STDP:} The same dataset has been used to train spiking neural networks with STDP by \cite{rybka2024comparison, 10.1007/978-3-030-96993-6_48}. Their results vary 
from $83\%$ to $95\%$ of accuracy depending on the network settings: in some of them they are comparable to CHANI's performance. Note that these models are more complex and do not have any theoretical guarantees.
\newline

To conclude, CHANI with EWA succeeds the learning task with reasonable results. This is a huge improvement in comparison to our numerical results of \cite{jaffard:hal-04065229}, where we trained HAN solely on the dataset of section \ref{sec synth data}. By adding layers and taking neuronal synchronization into account, CHANI can now learn more realistic tasks.

\section*{Conclusion}

In this paper, we introduced our cognitive network, CHANI (Correlation-based Hawkes Aggregation of Neurons with bio-Inspiration), which provably learns to classify objects, and we presented theoretical insights into CHANI's learning capabilities with any number of hidden layers. First, we interpreted regret bounds resulting from the use of expert aggregation algorithms as learning rule from a learning perspective. Subsequently, in the specific framework CHANI EWA, we demonstrated that at the limit hidden neurons encode feature correlations, and we proved that our neuron selection method prioritizes those encoding the most significant correlations. Furthermore, through an analysis of the asymptotic behavior of output neurons, we demonstrated CHANI's ability to learn the correct classification of objects on average and even asymptotically when classes are defined by feature correlations. Lastly, we calculated the VC-dimension of our network. To our knowledge, this study is the first one to establish that local learning rules enable a biologically inspired network to learn to recognize complex concepts by forming neuronal assemblies, inducing global learning. However, several challenging research avenues remain unexplored. In order to make our model more realistic and computationally efficient, one interesting direction is investigating CHANI's behavior when objects are not presented for a fixed duration, but only until one output neuron has spiked significantly more than the others. This would correspond to a reaction time and would have a cognitive interpretation. Another innovative line of research is extending these findings by establishing regret bounds for other local rules, such as STDP (Spike-Timing-Dependent Plasticity).


\acks{This research was supported by the French government, through CNRS (eXplAIn team), the UCA$^{Jedi}$ and 3iA Côte d'Azur Investissements d'Avenir managed by the National Research Agency (ANR-15 IDEX-01 and ANR-19-P3IA-0002), directly by the ANR project ChaMaNe (ANR-19-CE40-0024-02) and GraVa (ANR-18-CE40-0005), and finally by the interdisciplinary Institute for Modeling in Neuroscience and Cognition (NeuroMod).}


\appendix

All the numerical results can be reproduced thanks to the code available at \url{https://github.com/SophieJaffard/CHANI}.
In Appendix \ref{sec notations}, we summarize all the notations and in Appendix \ref{app reg} and \ref{app spe} we provide the proofs of our results. We do not provide the proofs in the order in which the results are enunciated in the article, but from the most general results to the most specific ones.

\section{Notations} \label{sec notations}
\begin{longtable}{|l|l|} \caption{Table of notations} \label{tab_notations} \\
  \hline
   Notation & Description \\
   \hline
      $\Obj$ & set of objects \\
      \hline
      $o$ & nature of an object \\
     \hline
     $M$ & number of objects presented to the network to train a hidden layer  \\
     \hline
     $N$ & number of objects presented to the network to train the output layer \\
     \hline
     $m$ & index of an object \\
     \hline
     $T$ & number of time steps during which one object is presented \\
    \hline 
    $I$ & set of features and input neurons \\
    \hline
    $i$ & index of an input neuron or a feature\\
    \hline
    $K$ & set of classes and output neurons \\
    \hline
    $k$ & index of an output neuron or a class \\
    \hline
    $L$ & number of hidden layers \\
    \hline
    $l$ & depth of a hidden layer \\
    \hline
    $I_l$ & set of every subsets of $I$ having cardinal $2^l$\\
    \hline
    $J_l$ & set of hidden neurons of depth $l$ before selection \\
    \hline
    $\hat{J}_l$ & set of selected hidden neurons of depth $l$ \\
    \hline
    $\J_l$ & set of hidden neurons of interest of depth $l$ \\
    \hline
    $j = j_S$ & index of a hidden neuron coding for a set $S\subset I$ \\
    \hline
    $o^l_m$ & nature of the $m^{th}$ object of the training phase of layer $l$ \\
    \hline
    $\mathcal{P}_E$ & set of probability distributions over the set $E$ \\
    \hline
   \end{longtable}

\allowdisplaybreaks

\section{Proofs of section \ref{sec reg}} \label{app reg}

\subsection{Proof of Proposition \ref{prop reg hid}}
Let $l\in [L]$, $j=j_S\in J_l$, where $S= S_1\cup S_2\subset I$ with $S_1$ and $S_1$ encoded in the previous layer. Then
\begin{align*}
    \sum_{m=1}^M w^j_m \cdot g^j_m &= \sum_{m=1}^M w^j_m  \cdot (\hat{\rho}^l_m(\{j_{S_1},j_{S_2},j'\})_{j'\in \hat{J}_{l-1}} \\
& = \sum_{m=1}^M w^j_m  \cdot (\brac{X^{j_{S_1},l}_{m,t}X^{j_{S_2},l}_{m,t}X^{j',l}_{m,t}}_{t\in[T]})_{j'\in \hat{J}_{l-1}}  \\
& =  \sum_{m=1}^M \brac{ X^{j_{S_1},l}_{m,t}X^{j_{S_2},l}_{m,t} ( w^j_m  \cdot X^{\hat{J}_{l-1},l}_{m,t}) }_{t\in[T]} 
\end{align*}
and
\begin{align*}
     \sum_{m=1}^M w^j_m \cdot g^j_m + \sum_{m=1}^M \brac{ -\nu X^{j_{S_1},l}_{m,t}X^{j_{S_2},l}_{m,t} }_{t\in[T]} &=  \sum_{m=1}^M \brac{ X^{j_{S_1},l}_{m,t}X^{j_{S_2},l}_{m,t} \psi_{m,t}^{j,l}(w^j_m) }_{t\in[T]}.
\end{align*}
Similarly, for $q\in \mathcal{P}_{\hat{J}_{l-1}}$,
\begin{align*}
     \sum_{m=1}^M q \cdot g^j_m + \sum_{m=1}^M \brac{ -\nu X^{j_{S_1},l}_{m,t}X^{j_{S_2},l}_{m,t} }_{t\in[T]} &=  \sum_{m=1}^M \brac{ X^{j_{S_1},l}_{m,t}X^{j_{S_2},l}_{m,t} \psi_{m,t}^{j,l}(q) }_{t\in[T]}.
\end{align*}
Hence the regret of neuron $j$ is 
$
R^j_M = \max_{q\in \mathcal{P}_{\hat{J}_{l-1}}}  \sum_{m=1}^M \Brac{ X^{j_{S_1}}_{m,t}X^{j_{S_2}}_{m,t} (\psi_{m,t}^{j,l}(q) - \psi_{m,t}^{j,l}(w^j_m)) }_{t\in[T]}$.
The gains take value in $[0,1]$ so we get the bound by applying the regret bound of the expert aggregation given by Assumption \ref{assumption regret bound hid}.

\subsection{Proof of Proposition \ref{prop reg class disc}}

Let $k\in K$, $j\in \hat{J}_L$. Then
\begin{align*}
   \frac{1}{N} \sum_{m=1}^N w^k_m \cdot g^k_m  &= \frac{1}{N^k}\sum_{m, o_m^{L+1}\in k }  w^k_m \cdot \brac{X^{j,L+1}_{m,t}}_{t\in[T]} \\
   & \quad - \Brac{ \frac{1}{N^{k'}} \sum_{m, o_m^{L+1}\in k'}  w^k_m \cdot \brac{X^{j,L+1}_{m,t}}_{t\in[T]} }_{k' \ / \ k'\neq k} \\
   & = \Disc^k_N(w^k_{[N]}, X_{[N],[T]}^{\hat{J}_L,L+1} )
\end{align*}
and similarly for $q^k\in \mathcal{P}_{\hat{J}_{L-1}}$, we have
    $\frac{1}{N} \sum_{m=1}^N q^k \cdot g^k_m = \Disc^k_N(q^k, X_{[N],[T]}^{\hat{J}_L,L+1})$.
So the regret of neuron $k$ divided by $N$ is
\begin{equation} \label{eqpr}
    \frac{R^k_N}{N} = \max_{q^k\in \mathcal{P}_{\hat{J}_{L-1}}} \Disc^k_N(q^k,X_{[N],[T]}^{\hat{J}_L,L+1}) - \Disc^k_N(w^k_{[N]},X_{[N],[T]}^{\hat{J}_L,L+1}) \leq \frac{C\abs{K}}{\xi(\abs{K}-1)} \sqrt{\frac{\ln(\abs{\hat{J}_L})}{N}}
\end{equation}
according to Assumption \ref{assumption regret bound out}, because the gains take value in $[- \frac{1}{\xi(\abs{K}-1)}, \frac{1}{\xi}]$. Then

\begin{align*}
    \Disc_N(w_{[N]}^K,X_{[N],[T]}^{\hat{J}_L,L+1}) &= \Brac{\hat{p}^k_m(w^k_m) - \hat{p}^{k'}_m(w^{k'}_m)  }_{\begin{subarray}{l}
        k\in K \\ k'\neq k \\ m, o_m^{L+1} \in k
    \end{subarray}} \\
    &= \Brac{\hat{p}^k_m(w^k_m)  }_{\begin{subarray}{l}
        k\in K \\ m, o_m^{L+1} \in k
    \end{subarray}} - \Brac{\hat{p}^{k'}_m(w^{k'}_m)}_{\begin{subarray}{l}
        k\in K \\ k'\neq k \\ m, o_m^{L+1} \in k
    \end{subarray}}
\end{align*}
Let us exchange the name of the indexes $k$ and $k'$ in the second term.
\begin{align*}
    &\Disc_N(w_{[N]}^K,X_{[N],[T]}^{\hat{J}_L,L+1}) = \Brac{\hat{p}^k_m(w^k_m)  }_{\begin{subarray}{l}
        k\in K \\ m, o_m^{L+1} \in k
    \end{subarray}} - \Brac{\hat{p}^{k}_m(w^{k}_m)}_{\begin{subarray}{l}
        k\in K \\ k'\neq k \\ m, o_m^{L+1} \in k'
    \end{subarray}} \\
    &= \Brac{ \brac{\hat{p}^k_m(w^k_m)  }_{\begin{subarray}{l}
         m, o_m^{L+1} \in k
    \end{subarray}} - \brac{\hat{p}^{k}_m(w^{k}_m)}_{\begin{subarray}{l}
        k'\neq k \\ m, o_m^{L+1} \in k'
    \end{subarray}} }_{k\in K}
    = \brac{ \Disc_N^k(w^k_{[N]},X_{[N],[T]}^{\hat{J}_L,L+1}) }_{k\in K}.
\end{align*}
and similarly $\Disc_N(q^K, X_{[N],[T]}^{\hat{J}_L,L+1}) = \brac{ \Disc_N^k(q^k, X_{[N],[T]}^{\hat{J}_L,L+1}) }_{k\in K}$. Therefore, according to \eqref{eqpr},
\begin{equation*}
     \Disc_N(w_{[N]}^K,X_{[N],[T]}^{\hat{J}_L,L+1}) \geq \max_{q^K\in \mathcal{P}_{\hat{J}_L}^{\abs{K}}}  \Disc_N(q^K, X_{[N],[T]}^{\hat{J}_L,L+1}) -  \frac{C\abs{K}}{\xi(\abs{K}-1)} \sqrt{\frac{\ln(\abs{\hat{J}_L})}{N}}.
\end{equation*}

\section{Proofs of section \ref{sec theo results specif set}} \label{app spe}
\subsection{Preliminary propositions}

\begin{proposition} \label{prop prelim hoef}
    Let $(\Omega, \mathcal{F}, \mathbb{P})$ be a probability space, let $\mathcal{G}\subset \mathcal{F}$ be a $\sigma$-algebra, and let
    $E,F$ be $\mathcal{G}$-measurable random finite sets. 
    
    Let $A,B\in \mathbb{N}^*$, and let $(Z^{e\to f}_{a,b})_{e\in E, f\in F, 1\leq a \leq A, 1\leq b \leq B}$ be
    random variables bounded by $1$ such that for every $f\in F$, $e\in E$ the variables 
    $(Z^{e\to f}_{a,b})_{1\leq a \leq A, 1\leq b \leq B}$ are independent knowing the $\sigma$-algebra $\mathcal{G}$.
    Let $\alpha>0$. Then with probability $1-\alpha$, for all $e\in E, f\in F$
    \[
        \left| \sum_{a=1}^A \brac{Z^{e\to f}_{a,b}}_{b\in [B]} - 
        \mathbb{E}\Big[ \sum_{a=1}^A \brac{Z^{e\to f}_{a,b}}_{b\in [B]} \mid \mathcal{G}\Big] \right| <   \sqrt{\frac{A}{2B}\ln\Big( \frac{2\abs{E}\abs{F}}{\alpha}\Big)}.  
    \]
\end{proposition}

\begin{proof}
Let $\alpha>0$, $e\in E$, $f\in F$. The variables 
$(Z^{e\to f}_{a,b})_{1\leq a \leq A, 1\leq b \leq B}$ are independent knowing the $\sigma$-algebra $\mathcal{G}$ and bounded 
by $1$ so according to Hoeffding's inequality, for $\beta>0$ and $\mathcal{G}$-measurable,
\[
\mathbb{P}\Bigg(
  \left| \sum_{a=1}^A \brac{X^{e\to f}_{a,b}}_{b\in [B]} - 
\mathbb{E}\Big[ \sum_{a=1}^A \brac{X^{e\to f}_{a,b}}_{b=1,\dots,B} \mid \mathcal{G}\Big] \right| \geq \beta \mid \mathcal{G}
\Bigg)  
\leq 2e^{-2\beta^2 \frac{B}{A}}.   
\] 
Let  
\[D^c := \Big\{\exists e\in E, f\in F,  \left| \sum_{a=1}^A \brac{X^{e\to f}_{a,b}}_{b=1,\dots,B} - 
\mathbb{E}\Big[ \sum_{a=1}^A \brac{X^{e\to f}_{a,b}}_{b\in [B]} \mid \mathcal{G}\Big] \right| \geq \beta\Big\}.\]
Then $\mathbb{P}(D^c \mid \mathcal{G}) \leq \abs{E}\abs{F}2e^{-2\beta^2 \frac{B}{A}}$. Let us choose
$\beta = \sqrt{\frac{A}{2B}\ln\Big( \frac{2\abs{E}\abs{F}}{\alpha}\Big)}$. Then $\mathbb{P}(D^c \mid \mathcal{G})\leq \alpha$.
By integrating we get $\mathbb{P}(D^c) \leq \alpha$.

\end{proof}

\begin{proposition} \label{prop prelim EWA}
        Let $d\in \mathbb{N}$, $A^1,\dots, A^d\in \mathbb{R}$. Let $E$ be the subset of $[d]$ defined by $E:= \arg\max_i A^i$, and let
        $w_i(M) := \frac{\exp(\sqrt{M}A^i)}{\sum_{l=1}^d \exp(\sqrt{M}A^l)}$. Let $A=\max_i A^i$. 
        \begin{itemize}
            \item If $E=[d]$ then for every $i\in [d]$, $w_i(M) = \frac{1}{d}.$
            \item If $E\subsetneq [d]$, let $\Delta := \max_i A^i - \max_{i\notin E} A^i$. Then, for all $i$,
            \[\Abs{w_i(M) - \frac{1}{\abs{E}}\mathbb{1}_{i\in E}} \leq \frac{1}{\abs{E}}\max\Big(1,\frac{d-\abs{E}}{\abs{E}}\Big) \exp(-\sqrt{M}\Delta).\]
        \end{itemize}
\end{proposition}

\begin{proof}
Let $A_{\text{max}}:= \max_i A^i$.
\newline
\textbf{Case} $E=[d]$. Then for all $i\in [d]$
    \[
    w_i(M) =  \frac{\exp(\sqrt{M}A_{\text{max}})}{\sum_{l=1}^d \exp(\sqrt{M}A_{\text{max}})} = \frac{1}{d}.
    \]
\textbf{Case} $E\subsetneq [d]$. Let $A_{\text{max bis}}:= \max_{i\notin E} A^i$. Then
\begin{equation} \label{eq1}
    w_i(M) \leq \frac{\exp(\sqrt{M}A^i)}{ \abs{E}\exp(\sqrt{M}A_{\text{max}})} = \frac{1}{\abs{E}} \exp(-\sqrt{M}( A_{\text{max}} - A^i))
\end{equation}
  Let $i\in E$.
  \begin{align*}
      w_i(M) &\geq \frac{\exp(\sqrt{M}A_{\text{max}})}{ \abs{E}\exp(\sqrt{M}A_{\text{max}}) + (d-\abs{E})\exp(\sqrt{M}A_{\text{max bis}}) } \\
      &= \frac{1}{\abs{E}}\frac{1}{1+\frac{d-\abs{E}}{\abs{E}} \exp(-\sqrt{M}(A_{\text{max}} - A_{\text{max bis}} ))} \\
      &\geq \frac{1}{\abs{E}} \Big(
1 - \frac{d-\abs{E}}{\abs{E}} \exp(-\sqrt{M}(A_{\text{max}} - A_{\text{max bis}} ))
      \Big) \\
      &= \frac{1}{\abs{E}} - \frac{d-\abs{E}}{\abs{E}^2} \exp(-\sqrt{M}(A_{\text{max}} - A_{\text{max bis}} )) 
  \end{align*}
And thanks to \eqref{eq1}, $w_i(M) \leq \frac{1}{\abs{E}}$.
Hence
\[\frac{1}{\abs{E}} - \frac{d-\abs{E}}{\abs{E}^2} \exp(-\sqrt{M}(A_{\text{max}} - A_{\text{max bis}} ))\leq  w_i(M) \leq \frac{1}{\abs{E}}.\]
Let $i\in [d]\setminus E$. According to \eqref{eq1}
\[
0 \leq  w_i(M) \leq  \frac{1}{\abs{E}} \exp(-\sqrt{M}( A_{\text{max}} - A_{\text{max bis}})) .
\]
\end{proof}
As proved in \cite[Proposition 4]{gao2018properties}, we have
\begin{proposition} \label{prop iaf}
Let $d\in \mathbb{N}$, $\eta\in \mathbb{R}$. Then the softmax function $f$ with temperature $\eta$ defined on $\mathbb{R}^d$ by $x \mapsto \softmax(\eta x)$ is $\eta$-Lipschitz.
\end{proposition}

\subsection{Definition of the network activity and coupled network activity} \label{def coupled}
Let us define more precisely the variables $X^{j,l}_{m,t}$ and $X^{j,l}_{o,t}$. For $l\in [L]$, let $\mathcal{F}_l$  be the $\sigma-$algebra generated by every event that happened until the end of the selection phase of layer $l$.

\noindent$\bullet$ For $i\in I, m\in [M]$, $l\in [L]$, the variables $(X^{i,l}_{m,t})_{1\leq t\leq T}$ are i.i.d and follow a Bernoulli distribution with parameter $p^i_o$ where $o_m^l =o$. Similarly, for $i\in I, o\in \Obj$, $l\in [L]$, the variables $(X^{i,l}_{o,t})_{1\leq t\leq T}$ are i.i.d and follow a Bernoulli distribution with parameter $p^i_o$.

\noindent $\bullet$ For $l\in [L]$, $j\in \hat{J}_1\cup\dots\cup\hat{J}_{l-1}\cup J_{l}$, we define conditionally to the past $\mathcal{F}_{l-1}$ the variables $(U_{m,t}^{j,l})_{1\leq m \leq M, 1\leq t\leq T}$ and $(U_{o,t}^{j,l})_{o\in \Obj, 1\leq t\leq T}$ which are i.i.d and follow a uniform distribution on $[0,1]$. Then, if $j\notin J_l$ we define its activity by $X^{j,l}_{m,t} := \mathbb{1}_{p^{j,l}_{m,t}(w^j_{M+1})\geq U^{j,l}_{m,t}}$, if $j\in J_l$ then $ 
    X^{j,l}_{m,t} := \mathbb{1}_{p^{j,l}_{m,t}(w^j_m)\geq U^{j,l}_{m,t}},
    $
and in any case
$
X^{j,l}_{o,t} := \mathbb{1}_{p^{j,l}_{o,t}(w^j_{M+1})\geq U^{j,l}_{o,t}}
$.

Let $l\in [L]$, $l'\leq l$ and $w_{\text{bis}}:=(w^j_{\text{bis}})_{j\in \hat{J}_1\cup\dots\cup\hat{J}_{l-1}\cup J_{l}}$ an  $\mathcal{F}_{l-1}$-measurable weight family. Then we define the coupled network activity with weights $w_{\text{bis}}$ as the variables $Z^{j,l}_{m,t} := \mathbb{1}_{p^{j,l}_{m,t}(w^j_{\text{bis}})\geq U^{j,l}_{m,t}}$ and $Z^{j,l}_{o,t} := \mathbb{1}_{p^{j,l}_{o,t}(w^j_{\text{bis}})\geq U^{j,l}_{o,t}}$.

\subsection{Proof of Theorem \ref{theorem lim hid}} \label{sec proof th hid}

We recall that for $l\in [L]$, $j=j_{S_1\cup S_2}\in \J_l$, we denote $\Bar{w}^j = \mathbb{1}_{j\in \{j_{S_1},j_{S_2}\}}$. Let $\alpha>0$. For $l\geq 1$ and $j=j_S$ where $S\subset I$, let 
 \[
        E_{\text{EWA}}^{j,l}(M,\abs{\J_{l-1}}) := \frac{1}{2}\max(1, \frac{\abs{\J_{l-1}}-2}{2}) \exp\Big(-\rho(S)2^{2-2^l}\sqrt{ 2M\ln(\abs{\J_{l-1}})}\Big).
        \]
For $A,B\geq 1$, let
\[
E_{\text{approx}}(T,A,B,\alpha) := 2 \sqrt{\frac{A\ln(A)}{T} \ln\Big(\frac{2AB}{\alpha}\Big)} .
\]
For $l\in [L]$ let
\[
E_w(l) := \max_{j\in \J_l} \sqrt{\abs{\J_{l-1}}} \norm{w^j_{M+1} - \Bar{w}^j}_2
\]
and 
\[
E_{\text{prec}}(M,\abs{\J_{l}}, l) := 6 \sqrt{2\ln(\abs{\J_l})M} \sum_{l'=1}^l E_w(l').
\]
Moreover, we denote $\J_l^* := \{j_{S_1\cup S_2} \text{ such that } j_{S_1}, j_{S_2}\in \J_l \text{ and } S_1\cap S_2 =\emptyset \}$.
Let us prove by induction the following proposition for every $l\in [L]$.

\begin{proposition}\label{prop induction}
    There exist constants $C_1^l$ and $C_2^l$ and events $(D^{l'})_{l'\leq l}$ and $(D^{l'}_{\text{sel}})_{l'\leq l}$, each of probability more than $1-\alpha$, such that if $T/M^{l-1}\geq C_1^l$ and $M\geq C_2^l$ then on $D^1\cap D^1_{\text{sel}}\cap\dots \cap D^{l-1}\cap D^l_{\text{sel}}$, for every $l'\leq l$, $\hat{J}_{l'}=\J_{l'}$ and for all $j\in \J_{l'}$, $j'\in \J_{l'-1}$,
    \begin{align*}
         \abs{w^{j'\to j}_{M+1} -\Bar{w}^{j'\to j}} \leq E_{\text{approx}}(T,\abs{\J_{l'-1}},\abs{\J_{l'-1}^{*}},\alpha) + &E^{j,l}_{\text{EWA}}(M,\abs{\J_{l'-1}}) \\
         &+ E_{\text{prec}}(M, \abs{\J_{l'-1}}, l'-1)
    \end{align*}
\end{proposition}

\textbf{Case l=1}: For $i\in I$, $j=j_{\{i_1, i_2\}}\in J_1$, we remind that the cumulated gain of neuron $i$ w.r.t. neuron $j$ is $G^{i\to j}_M = \sum_{m=1}^M \brac{X^{i,1}_{m,t}X^{i_1,1}_{m,t}X^{i_2,1}_{m,t}}_{t\in[T]}$.
Let 
\[
\Bar{G}^{i\to j}_M := \sum_{m=1}^M \mathbb{P}(X^{i,1}_{m,t}=1, X^{i_1,1}_{m,t}=1, X^{i_2,1}_{m,t}=1).
\]
The variables $(X^{i,1}_{m,t}X^{i_1,1}_{m,t}X^{i_2,1}_{m,t})_{1\leq m \leq M, 1\leq t \leq T}$ are independent so according to Proposition \ref{prop prelim hoef} there exists an event $D^1$ of probability more than $1-\alpha$ on which for all $j\in J_1$, $i\in I$
\[
\abs{G^{i\to j}_M -\Bar{G}^{i\to j}_M  } \leq \sqrt{\frac{M}{2T}\ln\Big( \frac{2\abs{J_1}\abs{I}}{\alpha}\Big)}.
\]
Let $\Bar{w}^{i\to j}_{M+1} := \frac{\exp(\eta_1 \Bar{G}^{i\to j}_M) }{\sum_{i'\in I}\exp( \eta_1 \Bar{G}^{i'\to j}_M )}$. Then according to \eqref{prop iaf},
\[
\norm{w^{j}_{M+1} - \Bar{w}^{j}_{M+1} }_2 \leq \eta_1  \sqrt{\sum_{i'\in I} \abs{G^{i'\to j}_M -\Bar{G}^{i'\to j}_M }}.
\]
Hence on $D^1$, with $\eta_1 = \sqrt{\frac{8\ln(\abs{I})}{M}}$ we get
\begin{equation} \label{eq6}
    \norm{w^{j}_{M+1} - \Bar{w}^{j}_{M+1} }_2 \leq 2 \sqrt{\frac{\abs{I}\ln(\abs{I})}{T} \ln\Big(\frac{2\abs{I}\abs{J_1}}{\alpha}\Big)} = E_{\text{approx}}(T,\abs{I},\abs{J_1},\alpha).
\end{equation}
Besides, for $j=j_{\{i_1,i_2\}}\in J_1$ and $i\in I$, $\Bar{G}^{i\to j}_M = \sum_{m=1}^M \rho_{o^1_m}(\{i,i_1, i_2\})$. According to Assumption \ref{assump nb obj}, each nature of object is presented the same amount of times so
\[
\Bar{G}^{i\to j}_M = \frac{M}{\abs{\Obj}} \sum_{o\in \Obj} \rho_{o}(\{i,i_1, i_2\}) = M \rho(\{i,i_1, i_2\}).
\]
and $\eta_1 \Bar{G}^{i\to j}_M = \sqrt{8\ln(\abs{I})M}\rho(\{i,i_1, i_2\}$. We can distinguish between two cases.\\
1. If $j=j_{\{i_1,i_2\}}$ is such that $ \rho(\{i_1, i_2\}) =0$ then for all $i\in I,  \rho(\{i,i_1, i_2\})=0$. Hence according to Proposition \ref{prop prelim EWA}, we have  
\begin{equation} \label{eq2}
    \Bar{w}^{i\to j}_{M+1} =  \Bar{w}^{i\to j}
\end{equation}
 where $ \Bar{w}^{i\to j} := \frac{1}{\abs{I}}$. This definition of $  \Bar{w}^{i\to j}$ does not conflict with the one given at beginning of the proof because according to the definition of $\J_1$ given in section \ref{sec main th}, $j\notin \J_1$.
 \newline
 2. If $j=j_{\{i_1,i_2\}}$ is such that $ \rho(\{i_1, i_2\}) > 0$ then according to Assumption \ref{assump rho}, $\arg\max_{i\in I} \rho(\{i,i_1,i_2\}) = \{i_1,i_2\}.$ Hence according to Proposition \ref{prop prelim EWA}, we have
 \begin{equation}\label{eq3}
     \abs{\Bar{w}^{i\to j}_{M+1} -  \Bar{w}^{i\to j}} \leq E^{j,1}_{\text{EWA}}(M,\abs{I}) .
 \end{equation} 
Hence at the end of the learning phase of layer $1$, by combining \eqref{eq6} and \eqref{eq2} we have for all $j_S$ such that $ \rho(S) =0$:
\begin{equation} \label{eq7}
    \norm{w^{j}_{M+1} - \Bar{w}^{j} }_2 \leq  E_{\text{approx}}(T,\abs{I},\abs{J_1},\alpha)
\end{equation}
and by combining \eqref{eq6} and \eqref{eq3} we have for all $j_S$ such that $\rho(S)>0$:
 \begin{equation} \label{eq8}
     \norm{w^{j}_{M+1} - \Bar{w}^{j}}_2 \leq E_{\text{approx}}(T,\abs{I},\abs{J_1},\alpha) + \sqrt{\abs{I}}E^{j,1}_{\text{EWA}}(M,\abs{I}) .
 \end{equation}
Now let us study the selection phase. Let $\mathcal{F}_1^{\text{learn}}:= \mathcal{F}^1_{M,T}$ be the $\sigma$-algebra generated by every event that happened until the end of the learning phase of layer $1$. Let $j\in J_1$. The variables $(X^{j,1}_{o,t})_{o\in \Obj, 1\leq t \leq T}$ are independent knowing the $\sigma$-algebra $\mathcal{F}_1^{\text{learn}}$ since the weights $w^j_{M+1}$ are frozen and are $\mathcal{F}_1^{\text{learn}}$-measurable. Hence according to Proposition \ref{prop prelim hoef}, there exists a set $D^1_{\text{sel}}$ of probability more than $1-\alpha$ on which for every $j\in J_1$, $o\in \Obj$, 
\begin{equation*}
    \Abs{\brac{X^{j,1}_{o,t}}_{t\in[T]} - \mathbb{E}[X^{j,1}_{o,1} \mid \mathcal{F}_1^{\text{learn}}]} \leq \sqrt{\frac{1}{2T}\ln\Big( \frac{2\abs{J}\abs{\Obj}}{\alpha}\Big)}
\end{equation*}
since $\mathbb{E}[X^{j,1}_{o,t} \mid \mathcal{F}_1^{\text{learn}}]$ does not depend on $t$.
\newline

Let $(Z^{j,1}_{o,t})_{j\in J_1, o\in \Obj, 1\leq t \leq T}$ be the coupled network activity with weights $(\Bar{w}^j)_{j\in J_1}$ as defined in section \ref{def coupled}. Then for $j\in J_1$, $o\in \Obj$,
\begin{align*}
    \mathbb{E}[ \abs{X^{j,1}_{o,t} - Z^{j,1}_{o,t}} \mid \mathcal{F}_1^{\text{learn}} ] &\leq \abs{p^{j,1}_{o,t}(w^j_{M+1}) -p^{j,1}_{o,t}(\Bar{w}^j)} 
    \leq \abs{(w^j_{M+1} - \Bar{w}^j) \cdot X^{I,1}_{o,t}} .
\end{align*}
By Cauchy-Schwartz inequality and since the variables $X^{i,1}_{o,t}$ are bounded by $1$,
\begin{align*}
    \mathbb{E}[ \abs{X^{j,1}_{o,t} - Z^{j,1}_{o,t}} \mid \mathcal{F}_1^{\text{learn}} ] \leq \norm{w^j_{M+1} - \Bar{w}^j}_1
    \leq \sqrt{\abs{I}} \norm{w^j_{M+1} - \Bar{w}^j}_2 = E_w(1) .
\end{align*}
Therefore, on $D^1_{\text{sel}}$, for all $j\in J_1$ and $o\in \Obj$
\begin{equation*}
    \Abs{\brac{X^{j,1}_{o,t}}_{t\in[T]} - \mathbb{E}[Z^j_{o,1} \mid \mathcal{F}_1^{\text{learn}}]} \leq \sqrt{\frac{1}{2T}\ln\Big( \frac{2\abs{J}\abs{\Obj}}{\alpha}\Big)} + E_w(1).
\end{equation*}
Let $j=j_{\{i_1,i_2\}}\in J_1$. Let us compute $ \mathbb{E}[Z^j_{o,1} \mid \mathcal{F}_1^{\text{learn}}]$. If $\rho(\{i_1,i_2\}) >0$ then 
\begin{align*}
    \mathbb{E}[Z^j_{o,1} \mid \mathcal{F}_1^{\text{learn}}] &= \mathbb{E}\Big[\Big(-\frac{1}{2}+ \frac{1}{2}(X^{i_1,1}_{o,1} + X^{i_2,1}_{o,1})\Big)_+\Big] = \frac{1}{2} \rho_o(\{i_1,i_2\}).
\end{align*}
If  $\rho(\{i_1,i_2\}) = 0$ then
$
    \mathbb{E}[Z^j_{o,1} \mid \mathcal{F}_1^{\text{learn}}] = \mathbb{E}\Big[\Big(-\frac{1}{2}+ \frac{1}{\abs{I}}\sum_{i\in I} X^{i,1}_{o,1}\Big)_+\Big]  =0   = \frac{1}{2} \rho_o(\{i_1,i_2\})
$
according to Assumption \ref{assump half sets} because less than half the neurons of $I$ are active when presented with $o$ and because $\rho_o(\{i_1,i_2\})=0$. Therefore on $D^1_{sel}$
\begin{equation*}
    \Abs{\brac{X^{j,1}_{o,t}}_{t\in[T]} - \frac{1}{2}\rho_o(\{i_1,i_2\})} 
    \leq \sqrt{\frac{1}{2T}\ln\Big( \frac{2\abs{J}\abs{\Obj}}{\alpha}\Big)} + E_w(1).
\end{equation*}
Besides, according to Assumption \ref{assump half sets} and the definition of the sets $\J_l$, there exists $p_1$ such that if $j\in \J_1$ then there exists
 $o\in \mathcal{O}$ such that $ \rho_o(\{i_1,i_2\})\geq p_1$ and if $j\notin \J_1$ then for all $o\in \mathcal{O}$, 
  $\rho_o(\{i_1,i_2\})< p_1$. Let $q_1:= \max_{j_S\in J_1\setminus \J_1, o\in \Obj} \rho_o(S) $. Let us choose 
  the threshold $s_1 := \frac{1}{2}(\frac{1}{2}p_1 + \frac{1}{2}q_1)$. Then, on $D^1\cap D^1_{\text{sel}}$,
   \[
   \Abs{\brac{X^{j,1}_{o,t}}_{t\in[T]} - \frac{1}{2}\rho_o(\{i_1,i_2\})} = O\Big(T^{-1/2} + e^{-c\sqrt{M}}\Big)\]
   so there exist 
   constants $C_1^1$ and $C_2^1$ independent of $M,N$ and $T$ such that for $T\geq C^1_1$ and $M\geq C^1_2$ we have
   for all $j_S\in \J_1$, $o\in \Obj$
\begin{equation*}
    \Abs{\brac{X^{j,1}_{o,t}}_{t\in[T]} - \frac{1}{2}\rho_o(S)} 
    < \frac{1}{2}\Big(\frac{1}{2}p_1 - \frac{1}{2}q_1\Big).
\end{equation*}
so if $j\in \J_1$ then
    $\brac{X^{j,1}_{o,t}}_{t\in[T]} \geq s_1$
and otherwise $   \brac{X^{j,1}_{o,t}}_{t\in[T]} < s_1.$
Hence with this choice of threshold, on $D^1\cap D^1_{\text{sel}}$ we have $\hat{J}_1 = \J_1$ and Proposition \ref{prop induction} is true for rank $1$.
\newline
\textbf{Case $l >1$ :} Suppose Proposition \ref{prop induction} is true for rank $l-1$. Let $D^1,D^1_{\text{sel}},\dots,$ $D^{l-1},$ $D^{l-1}_{\text{sel}}$ the events and $C_1^{l-1}, C_2^{l-1}$ the constants given by the proposition at rank $l-1$. Suppose $\frac{T}{M^{l-2}}\geq C_1^{l-1}$ and $M\geq C_2^{l-1}$. Then the conclusions of Proposition \ref{prop induction} hold for rank $l-1$. 
\newline

       We remind that for $j=j_{S_1\cup S_2}\in J_l$ with $j_1 = j_{S_1}, j_2 = j_{S_2}\in \hat{J}_{l-1}$ and $j'\in \hat{J}_{l-1}$, the cumulated gain of neuron $j'$ w.r.t neuron $j$ is $G^{j'\to j}_M = \sum_{m=1}^M \brac{X^{j_1,l}_{m,t}X^{j_2,l}_{m,t}X^{j',l}_{m,t}}_{t\in[T]}$.
       Let
       \[\mathring{G}^{j'\to j}_M :=
        \sum_{m=1}^M \mathbb{E}[ X^{j_1,l}_{m,1}X^{j_2,l}_{m,1}X^{j',l}_{m,1} \mid \mathcal{F}_{l-1}].\]
       The variables $(X^{j_1,l}_{m,t}X^{j_2,l}_{m,t}X^{j',l}_{m,t})_{1\leq m \leq M, 1\leq t \leq T}$ are independent
       knowing $\mathcal{F}_{l-1}$. 
        According to Proposition \ref{prop prelim hoef}, there exists a set $D^l$ of probability greater than $1-\alpha$
        on which for all $j\in \J_l, j'\in \hat{J}_l$,
        \begin{equation} \label{eq10}
            \abs{G^{j'\to j}_M - \mathring{G}^{j'\to j}_M} \leq 
            \sqrt{\frac{M}{2T}\ln\Big(\frac{2\abs{\hat{J}_{l-1}}\abs{J_l}}{\alpha}\Big)}
        \end{equation}
since $\mathbb{E}[ X^{j_1,l}_{m,t}X^{j_2,l}_{m,t}X^{j',l}_{m,t} \mid \mathcal{F}_{l-1}]$ does not depend on $t$.

        Let $(Z^{j,l}_{m,t})_{j\in \hat{J}_1\cup\hat{J}_{l-1}, 1\leq m \leq M, 1\leq t \leq T}$, 
        $(Z^{j,l}_{o,t})_{j\in \hat{J}_{l-1}, o\in \Obj, 1\leq t \leq T}$ be the coupled network activity with arbitrary $\mathcal{F}_{l-1}$-measurable
         weights $(w^j_\star)_{j\in \hat{J}_1\cup \dots \cup \hat{J}_{l-1}}$ as defined in section \ref{def coupled}. For $j=j_{S_1\cup S_2}\in J_l$ with $j_1 =j_{S_1}, j_2 = j_{S_2}\in \hat{J}_{l-1}$, $j'\in \hat{J}_{l-1}$, let 
        $\Bar{G}^{j'\to j}_M := \sum_{m=1}^M 
        \mathbb{E}[Z^{j_1,l}_{m,1}Z^{j_2,l}_{m,1}Z^{j',l}_{m,1} \mid \mathcal{F}_{l-1}]. 
        $
        Then since the variables are bounded by $1$, we have
        \begin{multline} \label{eq16}
            \abs{\mathring{G}^{j'\to j}_M - \Bar{G}^{j'\to j}_M} \leq \\
            \sum_{m=1}^M\Big(\mathbb{E}[\abs{X^{j',l}_{m,t} - Z^{j',l}_{m,t}} \mid \mathcal{F}_{l-1}] +
            \mathbb{E}[\abs{X^{j_1,l}_{m,t} - Z^{j_1,l}_{m,t}} \mid \mathcal{F}_{l-1}] +  \mathbb{E}[\abs{X^{j_2,l}_{m,t} - Z^{j_2,l}_{m,t}} \mid \mathcal{F}_{l-1}]\Big).
        \end{multline}
        Let $j''\in \{j',j_1,j_2\}$. Note that $j''\in \J_{l-1}$. We have
\begin{align*}
   & \mathbb{E}[\abs{X^{j'',l}_{m,t} - Z^{j'',l}_{m,t}} \mid \mathcal{F}_{l-1}]
   \leq \mathbb{E}[\abs{p^{j'',l}_{m,t}(w^{j''}_{M+1}) - p^{j'',l}_{m,t}(\Bar{w}^{j''})} \mid \mathcal{F}_{l-1}] \\
    & \leq \mathbb{E}[\abs{ w^{j''}_{M+1} \cdot X^{\J_{l-1},l}_{m,t-1}  -  w^{j''}_{\star} \cdot Z^{\J_{l-1},l}_{m,t-1} } \mid \mathcal{F}_{l-1}] \\
    & \leq \mathbb{E}[\abs{ w^{j''}_{M+1} \cdot X^{\J_{l-1},l}_{m,t-1} -  w^{j''}_{\star} \cdot X^{\J_{l-1},l}_{m,t-1}} \mid \mathcal{F}_{l-1} ] + 
    \mathbb{E}[\abs{ w^{j''}_{\star} \cdot X^{\J_{l-1},l}_{m,t-1} -  w^{j''}_{\star} \cdot Z^{\J_{l-1},l}_{m,t-1}} \mid \mathcal{F}_{l-1} ] \\
    &\leq \norm{w^{j''}_{M+1} -  w^{j''}_{\star}}_1 + \max_{a\in \J_{l-2}} \mathbb{E}[\abs{X^{a,l}_{m,t} - Z^{a,l}_{m,t}} \mid \mathcal{F}_{l-1}] 
\end{align*}
    because the variables $X^{\J_{l-1},l}_{m,t-1} = (X^{a,l}_{m,t-1})_{a\in \J_{l-1}}$ are bounded by $1$ and the weights $w^{j''}_{\star}$ and $w^{j''}_{M+1}$ are $\mathcal{F}_{l-1}$-measurable and bounded by $1$. By iteration we get
\begin{equation}\label{eq17}
      \mathbb{E}[\abs{X^{j'',l}_{m,t} - Z^{j'',l}_{m,t}} \mid \mathcal{F}_{l-1}]  \leq \sum_{l'=1}^{l-1} \max_{a\in \hat{J}_{l'}} \norm{w^{a}_{M+1} -  w^{a}_{\star}}_1 
\end{equation}
    Let us now work on the event $D^1\cap D^1_{\text{sel}}\cap\dots \cap D^{l-1}\cap D^{l-1}_{\text{sel}}\cap D^l$. Then for every $l'\leq l-1$, $\hat{J}_{l'} = \J_{l'}$. For every $a\in \J_{l'}$, let us choose $w^a_{\star} = \Bar{w}^a$. Then
    \begin{equation} \label{eq14} 
         \mathbb{E}[\abs{X^{j'',l}_{m,t} - Z^{j'',l}_{m,t}} \mid \mathcal{F}_{l-1}]  \leq \sum_{l'=1}^{l-1} E_w(l')
    \end{equation}
 and we can control the $E_w(l')$. Then by combining \eqref{eq16} and \eqref{eq14} we get
\begin{equation} \label{eq9}
     \abs{\mathring{G}^{j'\to j}_M - \Bar{G}^{j'\to j}_M} \leq 3M\sum_{l'=1}^{l-1} E_w(l').
\end{equation}
For $j\in J_l, j'\in \J_{l-1}$, let $\Bar{w}^{j'\to j}_{M+1} := \frac{\exp(\eta^j \Bar{G}^{j'\to j}_M)}{\sum_{j''\in \J_{l-1}} \exp(\eta^j\Bar{G}^{j''\to j}_M)}$. According to Prop. \ref{prop iaf},
\begin{align*}
    &\norm{w^{j}_{M+1} - \Bar{w}^{j}_{M+1}}_2
    \leq \eta^j \norm{( \Bar{G}^{j'\to j}_M )_{j'\in \J_{l_1}} - ( G^{j'\to j}_M )_{j'\in \J_{l_1}}}_2 \\
    &\leq  \eta^j \Big(\norm{( \mathring{G}^{j'\to j}_M )_{j'\in \J_{l_1}} - ( G^{j'\to j}_M )_{j'\in \J_{l_1}}}_2 + \norm{( \Bar{G}^{j'\to j}_M )_{j'\in \J_{l_1}} - ( \mathring{G}^{j'\to j}_M )_{j'\in \J_{l_1}}}_2 \Big)
\end{align*}
and by combining \eqref{eq10} and \eqref{eq9}, with  $\eta^j = \sqrt{\frac{8\ln(\abs{\J_{l-1}})}{M}}$ we get
\begin{equation} \label{eq12}
    \norm{w^{j}_{M+1} -  \Bar{w}^{j}_{M+1}}_2 \leq E_{\text{prec}}(M, \abs{\J_{l-1}}, l-1) +  E_{\text{approx}}(T,\abs{\J_{l-1}},\abs{\J_{l-1}^{*}},\alpha)
\end{equation}
Let us study $\Bar{G}^{j'\to j}_M$ where $j=j_{S_1\cup S_2}$. We denote $S=S_1\cup S_2$, and for a set $S_e$ with $e$ an index, we denote $j_e$ its associated neuron. For $j''=j_{S_a\cup S_b}\in \{j',j_1,j_2\}$, the conditional spiking 
probability of $Z^{j'',l}_{m,t}$ is $\Big(-\frac{1}{2} + \frac{1}{2}(Z^{j_a,l}_{m,t-1} + Z^{j_b,l}_{m,t-1})\Big)_+$. Then $j''$ spikes with probability $\frac{1}{2}$ if and only if neurons $j_a$ and $j_b$ 
spiked before. Let $S_1 = S_{a_1}\cup S_{a_2}$, $S_2 = S_{b_1}\cup S_{b_2}$, $j'=j_{S'}$ with $S' = S_{d_1}\cup S_{d_2}$. Let  $c_{l'}$ is the number of distinct variables of layer $l'$ which intervene in the tensors $j_1$, $j_2$ and $j'$.
\begingroup
\allowdisplaybreaks
\begin{align*}
    \Bar{G}^{j'\to j}_M &=  \sum_{m=1}^M \mathbb{E}[Z^{j_1,l}_{m,t}Z^{j_2,l}_{m,t}Z^{j',l}_{m,t} \mid \mathcal{F}_{l-1}] \\
    &=  \sum_{m=1}^M \mathbb{E}[ \frac{1}{2^{c^{j'\to j}_{l-1}}} Z^{j_{a_1},l}_{m,t} Z^{j_{a_2},l}_{m,t} Z^{j_{b_1},l}_{m,t} Z^{j_{b_2},l}_{m,t} Z^{j_{d_1},l}_{m,t}Z^{j_{d_2},l}_{m,t} \mid \mathcal{F}_{l-1}] \\
& \dots \\
& =  \sum_{m=1}^M 2^{-(c_{l-1}^{j'\to j} + c_{l-2}^{j'\to j} + \dots + c_{1}^{j'\to j})} \mathbb{E}[\prod_{i\in S'\cup S_1\cup S_2} X^{i,l}_{m,t} \mid \mathcal{F}_{l-1} ] \\
& = \sum_{m=1}^M 2^{-(c_{l-1}^{j'\to j} + c_{l-2}^{j'\to j} + \dots + c_{1}^{j'\to j})}\prod_{i\in  S'\cup S_1\cup S_2}  \mathbb{E}[X^{i,l}_{m,t} \mid \mathcal{F}_{l-1} ] \\
& = \sum_{m=1}^M 2^{-(c_{l-1}^{j'\to j} + c_{l-2}^{j'\to j} + \dots + c_{1}^{j'\to j})}\prod_{i\in  S'\cup S_1\cup S_2}  \mathbb{E}[X^{i,l}_{m,t}] \\
& = \sum_{m=1}^M 2^{-(c_{l-1}^{j'\to j} + c_{l-2}^{j'\to j} + \dots + c_{1}^{j'\to j})} \rho_m(S_1\cup S_2 \cup S') 
\end{align*}
\endgroup
because the variables $X^{i,l}_{m,t}$ are mutually independent knowing $\mathcal{F}_{l-1}$ and are independent of $\mathcal{F}_{l-1}$. Therefore, according to Assumption \ref{assump nb obj},
\begin{equation*}
    \Bar{G}^{j'\to j}_M = 2^{-(c_{l-1}^{j'\to j} + c_{l-2}^{j'\to j} + \dots + c_{1}^{j'\to j})}\rho(S' \cup S)M. 
\end{equation*}
There are two cases. If $\rho(S)=0$ then for all $j'\in \J_{l-1}$, $\Bar{G}^{j'\to j}_M=0$. Let $\Bar{w}^{j'\to j} := \frac{1}{\abs{\J_{l-1}}}$. This definition does not conflict with the one recalled at the beginning of the proof because $j\notin \J_l$ according to the definition of $\J_l$ given in section \ref{sec main th}.

Then for all $j'\in \J_{l-1}$, 
\begin{equation} \label{eq11}
    \Bar{w}^{j'\to j}_{M+1} =\Bar{w}^{j'\to j} .
\end{equation}
If $\rho(S) >0$ then according to the definition of $J_l$, $S_1\cap S_2 = \emptyset$ so for $j'\in \{j_1,j_2\}$, $c^{j'\to j}_{l'} = 2^{l - l'}$ for any $l'\leq l-1$. So $ \Bar{G}^{j'\to j}_M = 2^{-2^l + 2}\rho(S)M $.  If $j'\notin \{j_1, j_2\}$ then $c_{l-1}^{j'\to j} = 3$, otherwise $c_{l-1}^{j'\to j} = 2$. Besides, it is clear that if $j'\in \{j_1,j_2\}$ and $j''\notin \{j_1,j_2\}$ then for every $l'\leq l-2$, $c_{l'}^{j'\to j}  \leq c_{l'}^{j''\to j} $. In addition, $\rho(S) \geq \rho(S' \cup S)$ for any $j'\in \J_{l-1}$. 
Therefore for $j'\neq j_1, j_2$, $\Bar{G}^{j'\to j}_M \leq 2^{-2^l+1}\rho(S)M $. Hence 
\[\arg\max_{j'\in \J_{l-1}} 2^{-(c_{l-1}^{j'\to j} + c_{l-2}^{j'\to j} + \dots + c_{1}^{j'\to j})}\rho(S' \cup S)M = \{j_1,j_2\}\]
and 
\begin{align*}
    \max_{j'\in \J_{l-1}}& 2^{-(c_{l-1}^{j'\to j}   + c_{l-2}^{j'\to j} + \dots + c_{1}^{j'\to j})} \rho(S' \cup S) \\
    &- \max_{j'\neq j_1, j_2}2^{-(c_{l-1}^{j'\to j} + c_{l-2}^{j'\to j} + \dots + c_{1}^{j'\to j})}\rho(S' \cup S) \geq 2^{-2^l+1}\rho(S).
\end{align*}
Let $\Bar{w}^{j'\to j} := \mathbb{1}_{j'\in \{j_1, j_2\}}$. According to Proposition \ref{prop prelim EWA}, with $\eta^j = \sqrt{\frac{8\ln(\abs{\J_{l-1}})}{M}}$ we have for all $j'\in \J_{l-1}$
\begin{align} \label{eq13}
    \abs{\Bar{w}^{j'\to j}_{M+1} - \Bar{w}^{j'\to j}} &\leq  E_{\text{EWA}}^{j,l}(M,\abs{\J_{l-1}}).
\end{align}
Therefore, for $j=j_S\in J_l$,
\begin{itemize}
    \item if $\rho(S) = 0$ then by combining \eqref{eq11} and \eqref{eq12} we get
    \begin{equation*}
 \norm{w^{j}_{M+1} -  \Bar{w}^j}_2 \leq E_{\text{prec}}(M, \abs{\J_{l-1}}, l-1) +  E_{\text{approx}}(T,\abs{\J_{l-1}},\abs{\J_{l-1}^{*}},\alpha).
\end{equation*}
    \item if  $\rho(S) >0 $ then by combining \eqref{eq12} and \eqref{eq13} we get
    \begin{align*}
         \norm{w^{j}_{M+1} -  \Bar{w}^j}_2 \leq E_{\text{prec}}(M, \abs{\J_{l-1}}, l-1) + & E_{\text{approx}}(T,\abs{\J_{l-1}},\abs{\J_{l-1}^{*}},\alpha) +  \\
         &\sqrt{\abs{\hat{J}_{l-1}}} E_{\text{EWA}}^{j,l}(M,\abs{\J_{l-1}}).
    \end{align*}
\end{itemize}
Now let us study the selection phase. Let $\mathcal{F}_l^{\text{learn}} := \mathcal{F}^l_{M,T}$ be the $\sigma$-algebra generated by every event that happened until the end of the learning phase of layer $l$. Let $j\in J_l$. The variables $(X^{j,l}_{o,t})_{o\in \Obj, 1\leq t \leq T}$ are independent knowing the $\sigma$-algebra $\mathcal{F}_l^{\text{learn}}$ since the weights $w^j_{M+1}$ are frozen and are $\mathcal{F}_l^{\text{learn}}$-measurable. Hence according to Proposition \ref{prop prelim hoef}, there exists a set $D^l_{\text{sel}}$ of probability more than $1-\alpha$ on which for every $j\in J_l$, $o\in \Obj$,
\begin{equation*}
    \Abs{\brac{X^{j,l}_{o,t}}_{t\in[T]} - \mathbb{E}[X^{j,l}_{o,1} \mid \mathcal{F}_l^{\text{learn}}]} \leq \sqrt{\frac{1}{2T}\ln\Big( \frac{2\abs{J_l}\abs{\Obj}}{\alpha}\Big)}
\end{equation*}
since $\mathbb{E}[X^{j,l}_{o,t} \mid \mathcal{F}_l^{\text{learn}}]$ does not depend on $t$. 

Let $(Z^{j,l}_{o,t})_{j\in \hat{J}_1\cup\dots\cup\hat{J}_{l-1}\cup J_l, o\in \Obj, 1\leq t \leq T}$ be the coupled network activity with any weights $(w^j_\star)$ which are $\mathcal{F}_l^{\text{learn}}$-measurable as defined in section \ref{def coupled}. Then for $j = j_S\in J_l$, $o\in \Obj$, similarly as in \eqref{eq17} we get
\begin{align*}
     \mathbb{E}[\abs{X^{j,l}_{o,1} - Z^{j,l}_{o,1}}\mid \mathcal{F}_l^{\text{learn}}] &\leq \sum_{l'=1}^{l-1} \max_{a\in \hat{J}_{l'}} \norm{w^a_{M+1} - w^a_{\star}}_1 +  \max_{a\in J_{l}} \norm{w^a_{M+1} - w^a_{\star}}_1 .
\end{align*}
Let us work on $D^1\cap D^1_{\text{sel}}\cap \dots \cap D^l \cap D^l_{\text{sel}}$. Let $w^j_{\star}:= \Bar{w}^j$ for $j\in \J_1\cup\dots\cup\J_{l-1}\cup J_l$. Then we by Cauchy-Schwartz inequality we have
\begin{align*}
    \Abs{\brac{X^{j,l}_{o,t}}_{t\in[T]} - \mathbb{E}[Z^{j,l}_{o,1} \mid \mathcal{F}_l^{\text{learn}}]} \leq \sqrt{\frac{1}{2T}\ln\Big( \frac{2\abs{J_l}\abs{\Obj}}{\alpha}\Big)} + & \sum_{l'=1}^l E_w(l') +  \\
    &\sqrt{\abs{\hat{J}_{l'-1}}} \max_{a\in J_{l}} \norm{w^a_{M+1} - \Bar{w}^a}_2
\end{align*}
where we control every term. Let us study $\mathbb{E}[Z^{j,l}_{o,1} \mid \mathcal{F}_l^{\text{learn}}]$. If $\rho(S)>0$ then similarly as in the computation of $\Bar{G}^{j'\to j}_M$ we get that $\mathbb{E}[Z^{j,l}_{o,1} \mid \mathcal{F}_l^{\text{learn}}] = 2^{-2^{l}+1}\rho(S)$.
If $\rho(S)=0$ then
\begin{align*}
     \mathbb{E}[Z^{j,l}_{o,t} \mid \mathcal{F}_l^{\text{learn}}] =  \mathbb{E}\Big[\Big(-\frac{1}{2}+ \frac{1}{\abs{\J_{l-1}}} \sum_{j'\in \J_{l-1}} Z^{j',l}_{o,t-1}\Big)_+ \mid \mathcal{F}_l^{\text{learn}}\Big]. 
\end{align*}
The computation in the case $\rho(S)>0$ also holds for neurons of layer $l-1$, which are all in this case because $\hat{J}_{l-1}=\J_{l-1}$. Therefore, according to Assumption \ref{assump half sets}, at most half the neurons of layer $l-1$ are active when presented with $o$. Hence
\begin{equation*}
    \mathbb{E}[Z^{j,l}_{o,1} \mid \mathcal{F}_l^{\text{learn}}] = 0 = 2^{-2^{l}+1}\rho(S).
\end{equation*}
Therefore for all $j\in J_l$, $o\in \Obj$,
\begin{align*}
     &\Abs{\brac{X^{j,l}_{o,t}}_{t\in[T]} - 2^{-2^{l}+1}\rho(S)} \\
     &\leq  \sqrt{\frac{1}{2T}\ln\Big( \frac{2\abs{J_l}\abs{\Obj}}{\alpha}\Big)}+ \sum_{l'=1}^l E_w(l')+  \sqrt{\abs{\hat{J}_{l'-1}}}\max_{a\in J_{l}} \norm{w^a_{M+1} - \Bar{w}^a}_2 := E(l,M,T)
\end{align*}
Besides, by definition of $\J_l$, if $j\in \J_l$ then there exists $o\in \Obj$ such that $\rho_o(S)\geq 2^{2^l - 1} s_l$ and if $j\in J_l\setminus\J_l $ then for all $o\in \Obj$, $\rho_o(S)< 2^{2^l - 1} s_l$. Let $q_l:=2^{-2^l + 1} \max_{j_S\in J_l\setminus \J_l}\rho_o(S)$. Then $q_l < s_l$. Since $E(l,M,T) = O\Big((\frac{M^{l-1}}{T})^{1/2} + e^{-C\sqrt{M}}\Big)$ then for $\frac{T}{M^{l-1}}$ and $M$ large enough, 
\begin{equation} \label{eq15}
   E(l,M,T) < s_l - q_l.
\end{equation}
Hence there exists constants $C_1^l$ and $C_2^l$ such that for $\frac{T}{M^{l-1}}\geq C_1^l$ and $M\geq C_2^l$, \eqref{eq15} is true. Besides, $\frac{T}{M^{l-2}} \geq \frac{T}{M^{l-1}}$ so let us take $C_1^l\geq C_1^{l-1}$ and $C_2^l\geq C_2^{l-1}$ so that we still have $\frac{T}{M^{l-2}}\geq C_1^{l-1}$ and $M\geq C_2^{l-1}$. Under this condition and this choice of threshold, if $j\in \J_l$ then $\brac{X^{j,l}_{o,t}}_{t\in[T]} \geq s_l$
and otherwise
$\brac{X^{j,l}_{o,t}}_{t\in[T]} < s_l$
so $\hat{J}_l = \J_l$ and Proposition \ref{prop induction} is true for rank $l$.
\newline

Now that Proposition \ref{prop induction} is proved for every $l\in [L]$, we can directly deduce Theorem \ref{theorem lim hid} by bounding $\abs{J_l}$ by $\abs{\J_{l-1}}^2$.

\subsection{Proof of Corollary \ref{coro ideal hid}}

Let $l\in [L]$, $j\in \J_l$, $o\in \Obj$. With the same calculation as in the proof of Theorem \ref{theorem lim hid} (see section \ref{sec proof th hid}), we have $\mathbb{E}[p^j_o(\Bar{w}^j)] = 2^{-2^l+1}\rho(S)$. Besides, for $T$, $M$ and $T/M^{L-1}$ tending to infinity, the error term of Theorem \ref{theorem lim hid} tends to zero so for $T$, $M$ and $T/M^{L-1}$ large enough, for every $l\in [L]$, the set of selected neurons is always $\J_{l}$ and the limit layer is an ideal hidden layer with constant $ 2^{-2^l+1}$.

\subsection{Proof of Corollary \ref{coro average}} \label{sec proof coro}

Suppose Assumption \ref{assumption regret bound out} and the assumptions of Theorem \ref{theorem lim hid} hold. Let $\alpha>0$. Suppose $T/M^{L-1}\geq C_1$ and $M\geq C_2$ where $C_1$ and $C_2$ are the constants defined in Theorem \ref{theorem lim hid} and  $\Q_L$ is non-empty. Then the conclusions of Theorem \ref{theorem lim hid} hold. Let us use the notations introduced in the proof of Theorem \ref{theorem lim hid}.

Let $k\in K, j\in \hat{J}_L$. Let us define $S^{j\to k}_N := \sum_{m, o^{L+1}_m\in k} \brac{X^{j, L+1}_{m,t}}_{t\in[T]}$
and $\mathring{S}^{j\to k}_N := \sum_{m, o^{L+1}_m\in k} \mathbb{E}[X^{j,L+1}_{m,1}\mid \mathcal{F}_L]$.
The variables $(X^{j, L+1}_{m,t})_{1\leq t \leq T, m \text{ s.t. } o_m^{L+1}\in k}$ are bounded by $1$ and independent knowing $ \mathcal{F}_L$. Hence according to Proposition \ref{prop prelim hoef}, there exists an event $D^{L+1}$ of probability more than $1-\alpha$ such that on $D^{L+1}$, for every $k\in K$, $j\in \hat{J}_L$ we have
\begin{equation*}
    \abs{S^{j\to k}_N - \mathring{S}^{j\to k}_N} \leq \sqrt{\frac{N^k}{2T}\ln\Big(\frac{2\abs{\hat{J}_L}\abs{K}}{\alpha}\Big)}.
\end{equation*}
Let $(Z^{j,L+1}_{m,t})_{j\in \hat{J}_1\cup\dots\cup\hat{J}_L, 1\leq m\leq M, 1\leq t \leq T}$ be the coupled network activity with arbitrary weights $(w^j_\star)_{j\in \hat{J}_1\cup\dots\cup\hat{J}_L}$ which are $\mathcal{F}_L$-measurable as defined in section \ref{def coupled}. For $k\in K, j\in \hat{J}_L$ let
\[
\Bar{S}^{j\to k}_N := \sum_{m, o_m^{L+1}\in k} \mathbb{E}[Z^{j,L+1}_{m,1} \mid \mathcal{F}_L].
\]
Then,
\begin{align*}
   \abs{\Bar{S}^{j\to k}_N  - \mathring{S}^{j\to k}_N} &\leq \sum_{m, o_m^{L+1}\in k} \mathbb{E}[\abs{X^{j,L+1}_{m,1} - Z^{j,L+1}_{m,1}}\mid \mathcal{F}_L ]
\end{align*}
and similarly as in the proof of Theorem \ref{theorem lim hid} (see section \ref{sec proof th hid}), we have
\begin{align*}
    \abs{\Bar{S}^{j\to k}_N  - \mathring{S}^{j\to k}_N} &\leq \sum_{m, o_m^{L+1}\in k} \sum_{l=1}^{L} \sqrt{\abs{\hat{J}_l}} \max_{a\in \hat{J}_{l}} \norm{w^a_{M+1} - w^a_{\star}}_2 \\
    & = N^k \sum_{l=1}^L \sqrt{\abs{\hat{J}_l}}  \max_{a\in \hat{J}_{l}} \norm{w^a_{M+1} - w^a_{\star}}_2.
\end{align*}
Let us work on $D := D^1\cap D^1_{\text{sel}}\cap \dots \cap D^L\cap D^L_{\text{sel}}\cap D^{L+1}$ where for $l\in [L]$, $D^l$ and $D^l_{\text{sel}}$ are the events given in the proof of Theorem \ref{theorem lim hid} (see section \ref{sec proof th hid}). Then $D$ is of probability greater than $1-2(L+1)\alpha$. Choose $w_\star = \Bar{w}$. Then, for $j\in \J_L$, $k\in K$,
\begin{equation} \label{eq18}
    \abs{S^{j\to k}_N - \Bar{S}^{j\to k}_N}\leq  \sqrt{\frac{N^k}{2T}\ln\Big(\frac{2\abs{\J_L}\abs{K}}{\alpha}\Big)} + N^k  \sum_{l=1}^L E_w(l).
\end{equation}
Here, according to Assumption \ref{assump nb obj}, every nature of object is presented the same amount of time to the network during the learning phase of the output layer. Hence Assumption \ref{assump proportions} holds for $\xi = \frac{1}{\abs{\Obj}}$. Therefore, according to Proposition \ref{prop reg class disc},
  \[
\Disc_N(w^k_{[N]}, X_{[N],[T]}^{\hat{J}_L,L+1} ) \geq \max_{q^K\in (\mathcal{P}_{\J_{L}})^{\abs{K}}}  \Disc_N(q^K, X_{[N],[T]}^{\hat{J}_L,L+1}) - \frac{C\abs{\Obj}\abs{K}}{\abs{K}-1}\sqrt{\frac{\ln(\abs{\hat{J}_L})}{N}}.
\]
Let $q^K\in (\mathcal{P}_{\J_{L}})^{\abs{K}}$ be a feasible weight family.
\begin{align*}
     \Disc_N(q^K, X_{[N],[T]}^{\hat{J}_L,L+1}) &= \Brac{\hat{p}^k_m(q^k) - \hat{p}^{k'}_m(q^{k'}) }_{\begin{subarray}{l}
       k\in K \\ k'\neq k \\  m, o^{L+1}_m\in k 
    \end{subarray}} \\
    & = \Brac{ (q^k - q^{k'}) \cdot \Big(\frac{1}{N^k} S^{j\to k}_N\Big)_{j\in \J_L} }_{\begin{subarray}{l}
       k\in K \\ k'\neq k 
    \end{subarray} }.
\end{align*}
Besides, similarly as in the proof of Theorem \ref{theorem lim hid}, for $j=j_S$ we have $\mathbb{E}[X^{j,L+1}_{m,1}\mid \mathcal{F}_L] = 2^{-2^L +1}\rho_{o_m^{L+1}}(S)$
so
\begin{align*}
    \Brac{ (q^k - q^{k'}) \cdot \Big(\frac{1}{N^k} \Bar{S}^{j\to k}_N\Big)_{j\in \J_L} }_{\begin{subarray}{l}
       k\in K \\ k'\neq k 
    \end{subarray} } &=   2^{-2^L +1}\Brac{ (q^k - q^{k'}) \cdot \rho_o(S) }_{\begin{subarray}{l}
       k\in K \\ k'\neq k \\ o\in k
    \end{subarray} } \\
    &= \gamma_L\Disc^{\id}(q^K, \J_L)
\end{align*}
with $\gamma_L = 2^{-2^L +1}$. Therefore, by lower bounding $N^k$ by $\frac{N}{\abs{\Obj}}$ for every $k$ we get
\[
\Disc_N(q^K, X_{[N],[T]}^{\hat{J}_L,L+1}) \geq \gamma_L\Disc^{\id}(q^K, \J_L) - 2 \sqrt{\frac{\abs{\Obj}}{2TN}\ln\Big(\frac{2\abs{\J_L}\abs{K}}{\alpha}\Big)} -2 \sum_{l=1}^L E_w(l).
\]
Hence
\begin{align*}
    \Disc_N(w^k_{[N]}, X_{[N],[T]}^{\hat{J}_L,L+1} ) \geq \gamma_L &\Disc^{\id}(q^K, \J_L) - 2 \sqrt{\frac{\abs{\Obj}}{2TN}\ln\Big(\frac{2\abs{\J_L}\abs{K}}{\alpha}\Big)}\\
   & -2 \sum_{l=1}^L E_w(l) -  \frac{C\abs{\Obj}\abs{K}}{\abs{K}-1}\sqrt{\frac{\ln(\abs{\hat{J}_L})}{N}}.
\end{align*}

\subsection{Proof of Theorem \ref{theorem lim out}}

Suppose Assumption \ref{assumption regret bound out} and the assumptions of Theorem \ref{theorem lim hid} hold. Let $\alpha>0$. Suppose $T/M^{L-1}\geq C_1$ and $M\geq C_2$ where $C_1$ and $C_2$ are the constants defined in Theorem \ref{theorem lim hid}. Then the conclusions of Theorem \ref{theorem lim hid} hold. For $k\in K$, recall that $(\Bar{w}^{j\to k})_{j\in \J_L} = \frac{1}{\abs{\J_L^k}}\mathbb{1}_{\abs{\J_L^k}}$. Let us work on the event $D$ defined in the proof of Corollary \ref{coro average} and let us use the notations introduced in section \ref{sec proof coro}. Then $
    G_N^{j\to k} = \frac{N}{N^k} S^k_N - \Brac{ \frac{N}{N^{k'}} S^{k'}_N }_{k'\neq k}$.
Let 
$\Bar{G}_N^{j\to k} = \frac{N}{N^k} \Bar{S}^k_N - \Brac{ \frac{N}{N^{k'}} \Bar{S}^{k'}_N }_{k'\neq k}$.
Then according to \eqref{eq18},
\begin{align*}
    \abs{G^{j\to k}_N - \Bar{G}^{j\to k}_N}\leq  N\sqrt{\frac{1}{2TN^k}\ln\Big(\frac{2\abs{\J_L}\abs{K}}{\alpha}\Big)} &+\Brac{ N\sqrt{\frac{1}{2TN^{k'}}\ln\Big(\frac{2\abs{\J_L}\abs{K}}{\alpha}\Big)}}_{k'\neq k} \\
    &+2 N  \sum_{l=1}^L E_w(l).
\end{align*}
Let us study $\Bar{G}_N^{j\to k}$. Similarly as in the proof of Theorem \ref{theorem lim hid} (see section \ref{sec proof th hid}), for $j=j_S$
\begin{equation} \label{eq19}
     \mathbb{E}[Z^{j,L+1}_{m,1}\mid \mathcal{F}_L] = 2^{-2^L +1}\rho_{o_m^{L+1}}(S).
\end{equation}
Then $
    \Bar{G}_N^{j\to k} = 2^{-2^L +1} \Disc^{j\to k}$.
Let $\Bar{w}^{j\to k}_{N+1} := \frac{\exp(\eta^{L+1}\Bar{G}^{j\to k}_N)}{\sum_{h\in \J_L} \exp(\eta^{L+1} \Bar{G}^{h\to k}_N)}$.
Let $k\in K$, $E^k := \arg\max_{j\in \J_L} \Disc^{j\to k}$. Then according to Proposition \ref{prop prelim EWA}, if $E^k = \J_L$ then for every $j\in \J_L$ we have $\Bar{w}^{j\to k}_{N+1} = \Bar{w}^{j\to k}$.
Otherwise, 
\[
\abs{\Bar{w}^{j\to k}_{N+1} - \Bar{w}^{j\to k}} \leq E^k_{\text{EWA}}(N, \abs{\J_L})
\]
where \[
            E^k_{\text{EWA}}(N, \abs{\J_L}) := \frac{1}{\abs{\J_L}}
            \max\Big(1,\frac{\abs{\J_L} - \abs{\J_L^k}}{\abs{\J_L^k}}\Big) \exp\Big(-2^{-2^L+1}\Delta^k\sqrt{8N\ln(\abs{\J_L})}\Big).
            \]
Let $\xi(N,\abs{\J_L},\abs{O},T,K) = \sqrt{\frac{\ln(\abs{\J_L})}{\abs{\Obj}T} \ln\Big(\frac{2\abs{\J_L}\abs{K}}{\alpha}\Big)} + \frac{1}{\abs{\Obj}}\sqrt{8\abs{\J_L}\ln(\abs{\J_L})N}\sum_{l=1}^L E_w(l)$.
Besides, according to Proposition \ref{prop iaf}, with $\eta^{L+1} = \frac{1}{\abs{\Obj}}\sqrt{\frac{2\ln(\abs{\J_L})}{N}}$, by bounding every $\frac{N}{N^{k'}}$ by $\abs{\Obj}$ we get that $\norm{w^{k}_{N+1} - \Bar{w}^{k}_{N+1}}_2 \leq \xi(N,\abs{\J_L},\abs{O},T,K)$.
Hence,
\[
\norm{w^{k}_{N+1} - \Bar{w}^{k}}_2 \leq 
\begin{cases}
\xi(N,\abs{\J_L},\abs{O},T,K) & \text{if }E^k = \J_L, \\
\xi(N,\abs{\J_L},\abs{O},T,K) + \sqrt{\abs{\J_l}} E^k_{\text{EWA}}(N, \abs{\J_L}) & \text{otherwise.}
\end{cases}
\]

\subsection{Proof of Corollary \ref{coro lim output}}
When $T, M, N$ and $\frac{T}{NM^{L-1}}$ tend to infinity, the error terms of Theorem \ref{theorem lim out} tend to zero so we have the result.

\subsection{Proof of Theorem \ref{theo specific case}}
Suppose Assumption \ref{assump hom} (binary correlations) is verified. We denote $(i), (ii)$ and $(iii)$ the three points of the theorem.
\newline
Proof of $(i)\Rightarrow (iii)$. Suppose Assumption \ref{assump decomp} (class decomposition) is verified.
Let us compute the feature discrepancies of the neurons of $\J_L$ in order to compute the weights $\Bar{w}^K$ given in Corollary \ref{coro lim output}. For every $k\in K$, we choose the set $E^k$ given in Assumption \ref{assump decomp} (class decomposition) with maximal size. Let $j = j_S\in \J_L$. 
\begin{align*}
    \Disc^{j\to k} &= \rho^k(S) - \brac{\rho^{k'}(S)}_{k'\neq k} = \brac{\rho_o(S)}_{o\in k} - \brac{\rho_{o}(S)}_{
        k'\neq k, o\in k'
    }.
\end{align*}
Then according to Assumption \ref{assump hom} (binary correlations), $\rho_o(S) = p$ if and only if $o$ has features $S$, \ie if and only if $o\in \Obj^S$, and  $\rho_o(S) = 0$ otherwise.

Let $j=j_S$ where $S\in E^k$. Let $o\in k$. We have $ \brac{\rho_o(S)}_{o\in k} = \frac{\abs{\Obj^j}}{n^k} p = \frac{Cp}{n^k}$. Let $o\in k'\neq k$. If $\rho_{o}(S) >0$, it would mean that $o$ has features $S$ so $o$ would belong to $O^S$, then $o$ would belong to class $k$. This is impossible so $\rho_{o}(S)=0$. Therefore, $\Disc^{j\to k} = \frac{Cp}{n^k}$.
Let $j=j_S$ where $S\notin E^k$. This means that there exists $o\in \Obj^S$ such that $o\notin k$. Therefore $ \brac{\rho_o(S)}_{o\in k} \leq \frac{(C-1)p}{n^k} $ and $\Disc^{j\to k} \leq \frac{(C-1)p}{n^k}$.
Hence $\J^k = E^k$ (where $\J^k$ is the set defined in section \ref{sec main th} and we identify each neurons $j_S$ with its set $S$). So according to Corollary \ref{coro lim output}, for every $k\in K$ $\Bar{w}^K = \frac{1}{\abs{E^k}}\mathbb{1}_{j\in E^k}$. 

Let $\J_L$ an ideal hidden layer where we denote by $(p^j_o)_{o\in \Obj, j\in \J_L}$ the neurons spiking probabilities. Let $\gamma_L$ its constant. Let $k\in K$, $o\in \Obj$. The spiking probability of neuron $k$ when presented with object $o$ is
\begin{align*}
    p_o^k(\Bar{w}^k) &= \Bar{w}^K \cdot p^{\J_L}_o  = \gamma_L  \Bar{w}^K \cdot (\rho_o(S))_{j_S\in \J_L}  = \frac{\gamma_L}{\abs{E^k}} \sum_{j_S\in E^k} \rho_o(S)  = \frac{p\gamma_L n_o^j}{\abs{E^k}} ,
\end{align*}
where $n_o^j$ is the amount of sets $O^j$ containing $o$. If $o\in k$, according to Assumption \ref{assump decomp} (class decomposition) $n_o^j \geq 1$. If $o\notin k$ then $n_o^j =0$ (otherwise $o$ would belong to a set $O^j$ in the composition of class $k$ so $o$ would belong to $k$). Therefore, $p_o^k(\bar{w}^k) >0$ if, and only if, $o\in k$ so $\Bar{w}^K$ is a strong feasible weight family for the ideal layer $\J_L$.
\newline

Proof of $(ii)\Rightarrow (i)$. Suppose there exists a strong feasible weight family $q^K$ w.r.t. $\J_L$. Let $k\in K$ and $E^k := \{S\subset I_L \text{ such that } j_S\in \J_L \text{ and } q^{j\to k} > 0\}$.
Let us show that $k = \bigcup_{S\in E^k} \Obj^S$.
The ideal spiking probability of neuron $k$ when presented with $o\in \Obj$ is $\bar{p}^k_o(q^k, \J_L) = \sum_{j_S\in \J_L} q^{j\to k} \rho_o(S)$.
Let $o\in k$. Then $\bar{p}^k_o(q^k, \J_L) >0$ so there exists $j=j_S\in \J_L$ such that $q^{j\to k}>0$ and $ \rho_o(S)>0$. Hence $j\in E^k$ and $o$ has features $S$ so $o\in \Obj^S$. Therefore
$k\subset \bigcup_{S\in E^k} \Obj^S$.
Let $o \in  \bigcup_{S\in E^k} \Obj^S$. There exists $j=j_S\in E^k$ such that $o\in \Obj^S$ and $q^{j\to k}>0$. Then $\bar{p}^k_o(q^k, \J_L) >0$ so $o\in k$. Therefore, $\bigcup_{S\in E^k} \Obj^S \subset k$.
It is obvious that $(iii)\Rightarrow (ii)$ so we can conclude.

\subsection{Proof of Theorem \ref{theo principal}}

Under the assumptions of Theorem \ref{theo principal}, the conclusions of Corollaries \ref{coro ideal hid} and \ref{coro lim output} and Theorem \ref{theo specific case} hold: by combining them, we get the conclusions of Theorem \ref{theo principal}.

\subsection{Proof of Proposition \ref{prop vcdim}}

Let $j=j_S\in J_L$. Let $o^j:= (o^{i\to j})_{i\in I}$ where $o^{i\to j}= 1$ if and only if $i\in S$. Then $j$ is the only neuron of layer $\J_L$ activated by $o^j$. Indeed, all the sets $S'$ for $j_{S'}\in J_L$ have the same size and only a neuron $j_{S'}$ with smaller $\abs{S'}$ can also be activated by $o^j$. Let $E:= \{o^j, j\in \J_L\}$. Then $\abs{E}=\abs{\J_L}$ because all the $o^j$ are distinct. This set can be fully shattered by $H$: indeed, let $E'\subset E$, and $F:= \{j\in \J_L, \exists o \in E', o = o^{j}\}$. Then $E'$ and $E\setminus E'$ are separated by $\mathbb{1}_{F}$.

Let $E$ be a set of objects that can be fully shattered by $H$. In particular, for each $o\in E$, there exists $F\subset \J_L$ non empty such that such that $o$ is the unique object in $E$ which activates a neuron in $F$. Let $j_o$ be such a neuron. Then $o \mapsto j_o$ is an injection so $\abs{E}\leq \abs{\J_L}$. 

\subsection{Proof of proposition \ref{prop nu}}

Let us use the notations of Theorem \ref{theorem lim out}. Let $j\in \J_L$, $k\in K$. The limit cumulative gain in $M$ and $T$ of a hidden neuron is of the form \[\Bar{G}^{j\to k}_N := C N\Big( \brac{p^j_{o,\infty} }_{o\in k} - \brac{p^j_{o,\infty} }_{\begin{subarray}{l}
    k' \ / \ k'\neq k \\
    o\in k'
\end{subarray}}\Big) \]
where $C>0$ is a constant and $p^j_{o,\infty}$ is the limit spiking probability of hidden neuron $j$ when presented with object of nature $o$. Since limit hidden layers are not ideal layers, another quantity replaces the feature discrepancy: the quantity
\[
c^{j\to k} := \brac{p^j_{o,\infty} }_{o\in k} - \brac{p^j_{o,\infty} }_{\begin{subarray}{l}
    k' \ / \ k'\neq k \\
    o\in k'
\end{subarray}}
\]
and according to Proposition \ref{prop prelim EWA}, the output weights of neuron $k\in K$ converge to uniform distribution of weight on the set $\arg\max_{j\in \J_L} c^{j\to k}$.
In order to compute the limit weights of neuron $k_2$, let us compute the quantities $c^{j\to k_2}$. Let $m\in [N]$. The limit conditional spiking probability of a hidden neuron $j=j_{\{i_1,i_2\}}$ is
$
\Big(-\nu + \frac{1}{2}(X^{i_1}_{o,t} + X^{i_2}_{o,t})\Big)_+$.
Therefore, when only one of the two input neurons $i_1$ and $i_2$ is active, the spiking probability of $j$ is $q:= \Big(\frac{1}{2} - \nu\Big) p'$ and when the two neurons $i_2$ and $i_2$ are active, the spiking probability of $j$ is $p:= {p'}^2(1-\nu) + 2p'(1-p') (\tfrac{1}{2} - \nu)$.
By definition of the feature discrepancy we have
\begin{align*}
    d^{j_{\{\text{Blue},\text{Square}\}}\to k_2}=&\frac{1}{5}p -  q \\
     d^{j_{\{\text{Red},\text{Circle}\}}\to k_2}=&\frac{1}{5}p + \frac{2}{5} q - \frac{1}{2} q = \frac{1}{5}p - \frac{1}{10} q\\
     d^{j_{\{\text{Blue},\text{Circle}\}}\to k_2}=&\frac{3}{5}q -  \frac{1}{4}p - \frac{1}{4} q = \frac{7}{20} q - \frac{1}{4}p
\end{align*}
By symmetry, we know every $d^{i\to k_2}$.
Since $\nu < \frac{1}{2}$ we have $q >0$. Then
$d^{j_{\{\text{Blue},\text{Square}\}}\to k_2} <  d^{j_{\{\text{Red},\text{Circle}\}}\to k_2}$.
Besides, $d^{j_{\{\text{Red},\text{Circle}\}}\to k_2}- d^{j_{\{\text{Blue},\text{Circle}\}}\to k_2} = \frac{9}{20}p - \frac{9}{20}q >0$,
because $p>q$. Therefore, the limit weights of $k_2$ converge to the family putting weight $1/4 $ on $j_{\{\text{Red},\text{Circle}\}}$, $j_{\{\text{Red},\text{Triangle}\}}$, $j_{\{\text{Green},\text{Circle}\}}$ and $j_{\{\text{Green},\text{Triangle}\}}$. Hence, the limit activity of $k_2$ is equal to zero when presented with ${\color{blue}\Large \Box }$, so this objects is not well classified.

\vskip 0.2in

\bibliography{biblio}

\end{document}